\documentclass[12pt]{amsart}
\usepackage[utf8]{inputenc}
\usepackage{amsfonts,amsmath,amssymb,amsxtra,stmaryrd,mathdots}

\usepackage{a4wide}

\usepackage{colonequals}
\usepackage[utf8]{inputenc}
\usepackage[british]{babel}
\usepackage{hyperref}
\hypersetup{colorlinks=true,urlcolor=blue,citecolor=blue,linkcolor=blue}

\usepackage[all,cmtip]{xy}
\usepackage{float}

\usepackage[dvipsnames]{xcolor}
\usepackage[normalem]{ulem}
\usepackage{cleveref}
\Crefname{equation}{}{}

\usepackage{tikz}
\usepackage{tikz-cd}

\usepackage{amsmath}
\usepackage{mathtools}
\usepackage{thmtools}
\usepackage{thm-restate}
\usepackage{enumitem}

\usepackage{tikz-cd}
\usetikzlibrary{patterns,shapes,cd,arrows}
\usepackage[auto]{contour}
\contourlength{2pt}

\newcommand{\Q}{\mathbb{Q}}

\newcommand{\C}{\mathbb{C}}
\newcommand{\F}{\mathbb{F}}

\newcommand{\R}{\mathbb{R}}
\newcommand{\Z}{\mathbb{Z}}
\newcommand{\Qbar}{\overline{\Q}}

\DeclareMathOperator{\End}{\operatorname{End}}

\DeclareMathOperator{\GL}{GL}
\DeclareMathOperator{\SL}{SL}

\DeclareMathOperator{\Hom}{Hom}
\DeclareMathOperator{\Gal}{Gal}

\DeclareMathOperator{\GSp}{GSp}

\DeclareMathOperator{\Aut}{Aut}
\DeclareMathOperator{\Hg}{Hg}
\DeclareMathOperator{\MT}{MT}
\DeclareMathOperator{\ST}{ST}

\DeclareMathOperator{\Frob}{Frob}
\DeclareMathOperator{\Jac}{Jac}

\DeclareMathOperator{\Spec}{Spec}

\DeclareMathOperator{\HH}{H}

\newcommand{\KconnJ}{K(\varepsilon_{J_m})}
\newcommand{\KconnA}{K(\varepsilon_A)}
\newcommand{\KendA}{K(\End(A))}

\usepackage{mathrsfs}

\numberwithin{equation}{subsection}

\newtheorem{theorem}[equation]{Theorem}

\newtheorem{conjecture}[equation]{Conjecture}
\newtheorem{corollary}[equation]{Corollary}

\newtheorem{lemma}[equation]{Lemma}

\newtheorem{proposition}[equation]{Proposition}

\theoremstyle{definition}
\newtheorem{definition}[equation]{Definition}

\newtheorem{notation}[equation]{Notation}

\theoremstyle{remark}
\newtheorem*{remark*}{Remark}
\newtheorem{remark}[equation]{Remark}
\newtheorem{example}[equation]{Example}

\newcommand{\Gl}{\mathcal{G}_\ell}

\DeclareMathOperator{\Res}{\operatorname{Res}}
\DeclareMathOperator{\HHpn}{\HH_{\operatorname{prim}}^n} %

\DeclareMathOperator{\prim}{\operatorname{prim}}

\DeclareMathOperator{\Fp}{\Frob_\mathfrak{p}}

\newcommand{\Gam}[2]{\Gamma\left(\frac{#1}{#2}\right)}
\newcommand{\muu}{\mu} %
\newcommand{\mmZ}{\frac{1}{m}\mathbb{\Z}/\mathbb{Z}}

\title[Monodromy groups and exceptional Hodge classes]{Monodromy groups and exceptional Hodge classes, I: Fermat Jacobians}

\author{Andrea Gallese, Heidi Goodson, and Davide Lombardo}

\begin{document}

\begin{abstract}
Denote by $J_m$ the Jacobian variety of the hyperelliptic curve defined by the affine equation $y^2=x^m+1$ over $\mathbb{Q}$, where $m \geq 3$ is a fixed positive integer.
We compute several interesting arithmetic invariants of $J_m$: its decomposition up to isogeny into simple abelian varieties, the minimal field $\mathbb{Q}(\operatorname{End}(J_m))$ over which its endomorphisms are defined, and its connected monodromy field $\mathbb{Q}(\varepsilon_{J_m})$.
Currently, there is no general algorithm that computes the last invariant. 
For large enough values of $m$, the abelian varieties $J_m$ provide non-trivial examples of high-dimensional phenomena, such as degeneracy and the non-triviality of the extension $\mathbb{Q}(\varepsilon_{J_m})/\mathbb{Q}(\operatorname{End}(J_m))$.
\end{abstract}

\maketitle
\let\thefootnote\relax
\let\thefootnote\relax

\setcounter{tocdepth}{2}

\tableofcontents

\section{Introduction}
This paper is the first in a series of two (see \cite{part2}), in which we investigate how the existence of exceptional algebraic cycles (or, more precisely, exceptional Hodge classes) on an abelian variety $A$ affects some of its associated invariants. We consider both general abelian varieties and a specific family of interesting examples.

In complex algebraic geometry, an \emph{exceptional} Hodge class is an element of the Hodge ring not generated by divisor classes. 
Abelian varieties %
that support such exceptional classes are called \emph{degenerate}.
While the definition of degeneracy is a statement about the Hodge ring, we can see the effects of degeneracy in {several} %
invariants of the abelian variety $A$: the Mumford-Tate group and, assuming $A$ is defined over a number field, the connected monodromy field and the Sato-Tate group.
In the rest of this introduction we introduce the abelian varieties considered in this work, we describe the invariants that we wish to investigate, and we state our main results.

\subsubsection*{Fermat Jacobians}
We focus in particular on a family of abelian varieties that contains many degenerate ones. 
Specifically, we study the Jacobians $J_m$ of the hyperelliptic curves
$$C_m: y^2=x^m+1,$$
where $m\geq 3$ is any integer. Such abelian varieties are known to be degenerate if, for example, $m$ is an odd composite number (see \cite[Theorem 1.1]{Heidi}) and are nondegenerate if $m$ is a power of 2 (\cite[Theorem 1.2]{EmoryGoodson2024}), prime (\cite{Shimura}), or twice an odd prime (\cite{EmoryGoodson2022}). 
Moreover, the Mumford-Tate conjecture is known to hold for these varieties because they have complex multiplication. In particular, we know unconditionally that there is a close relationship between exceptional Hodge classes in the cohomology of $J_m$ and its Galois representations. 
We also note that the curves $C_m$ arise as quotients of the Fermat curves $x^m + y^m + z^m = 0 \subset \mathbb{P}_{2, \mathbb{Q}}$, which is why we call $J_m$ a \textit{Fermat Jacobian}.

Working with these concrete examples allows us to demonstrate different phenomena that can occur for degenerate abelian varieties and develop methods for understanding their structure.

The Jacobian varieties $J_m$ %
were studied by Shioda in \cite{Shioda3}, where he proved the Hodge conjecture for $J_m$ for infinite families of values of $m$, including some for which $J_m$ is degenerate. It can be shown that in most cases the Hodge ring of $J_m$ contains exceptional Hodge cycles (see \cite[Section 6]{Shioda3} and \cite{Heidi}).

\subsubsection*{The connected monodromy field} One of the main results of this paper is the explicit description of the connected monodromy field of the abelian varieties $J_m$, which we briefly define here (see \Cref{section: intro connected monodromy field} for more details). Let $A$ be an abelian variety defined over a number field $K$ and denote its dimension by $g= \dim A$. Fix an algebraic closure $\overline{K}/K$ and denote by $\Gamma_K = \Gal(\overline{K}/K)$ the corresponding absolute Galois group. Let $\ell$ be a prime number. Consider the Tate module $T_\ell A = \varprojlim_n A[\ell^n] $ and set $V_\ell A = T_\ell A \otimes_{\Z_\ell} \Q_\ell$. We denote by $\rho_{A, \ell}$ the Galois representation
\[ \rho_{A, \ell}\colon \Gamma_K \to \Aut(V_\ell A) \simeq \GL_{2g}(\Q_\ell) \]
induced by the natural $\Gamma_K$-action on the $\ell$-adic Tate module of $A$. The $\ell$-adic monodromy group $\Gl$ is defined as the Zariski closure of the image of $\rho_{A, \ell}$. Note that $\Gl$ need not be connected. The connected monodromy field $K(\varepsilon_A)$ is then the minimal extension $L/K$ such that the Zariski closure of $\rho_{A,\ell}(\Gamma_L)$ in $\Gl$ is connected \cite[no.~133]{serre-IV}, \cite[Propositions (6.12) and (6.14)]{MR1150604}. 
There is at present no general algorithm to compute $K(\varepsilon_A)$, but see work of Zywina \cite[Remark 1.18]{MR4496693} for a heuristic technique to guess this field, at least in certain cases.

\subsubsection*{The endomorphism field} Denote by $K(\End (A))/K$ the field of definition of all endomorphisms of $A$. We always have the containment $K(\End (A)) \subseteq \KconnA$, as proven, for example, in \cite[Proposition 2.10]{MR1355128}, and the equality $\KconnA = K(\End (A))$ is known to hold in many cases (for example, whenever $g \leq 3$ \cite[Theorem 6.10]{MR3320526}). However, there are known examples of abelian varieties of higher dimension for which the extension $\KconnA/K(\End(A))$ is non-trivial \cite[Section 4]{MR1630512}, \cite[Theorem 1.2.1]{cantoralfarfan2023monodromy}. This phenomenon is related to the existence of exceptional algebraic classes \cite[Corollary 2.4.4]{cantoralfarfan2023monodromy}. In the present paper, we expand this pool of examples and give a complete characterization of the field $\Q(\varepsilon_{J_m})$ for the Jacobians $J_m$, showing that it is often a non-trivial extension of $\Q(\operatorname{End}(J_m))$.

In order to prove the non-triviality of the extension $\Q(\varepsilon_{J_m})/\Q(\End(J_m))$ we first of all need to know the field $\Q(\End(J_m))$.
In \Cref{th: endomorphism field of Jm complete version}, we determine the endomorphism field $\Q(\End(J_m))$ of $J_m$ for every $m$; in particular, we show that when $m$ is odd the endomorphism field $\Q(\End(J_m))$ coincides with the $m$-th cyclotomic field $\Q(\zeta_m)$. We also describe completely the endomorphism algebra itself, something that does not seem to be known for general values of $m$ (see \Cref{th:Jm-factorization} and \Cref{lemma: ring of endomorphism of Xd}).
We mention in passing that this answers the question at the end of \cite{auffarth2024jacobian}: the largest integer for which $J_m$ is isogenous to a product of elliptic curves is $m=24$, and the largest prime with the property considered in \cite{auffarth2024jacobian} is $p=7$.

\subsubsection*{The Mumford-Tate group.}
Let $A$ be an abelian variety defined over a number field $k$. Fix an embedding $k \hookrightarrow \C$ and consider the rational Betti cohomology group $V = \HH^1(A_\C(\C), \Q)$. The Mumford-Tate group $\MT(A) \subseteq \GL_V$ is a connected, reductive linear algebraic group over $\Q$, constructed from the Hodge structure on $A_\C$ (see \Cref{sub:MTdef} for a definition).

We use the Mumford-Tate group to analyze Hodge classes supported on $A$ and its powers.
As a first approximation, notice that its dimension reflects the complexity of the Hodge ring: the larger the number of generators for the Hodge ring $A$, the smaller the dimension of the Mumford–Tate group tends to be.
For instance, a generic elliptic curve $E/\Q$ without complex multiplication has $\MT(E) = \GL_V$.
By contrast, if $E/\Q$ admits complex multiplication, then the graph of a non-trivial endomorphism in $E \times E$ defines an algebraic cycle whose class is an additional generator of the Hodge ring. In this case, $\MT(E)$ is a two-dimensional torus.

Our interest in exceptional algebraic classes motivates the need for explicit equations describing $\MT(A)$.
In \Cref{section:MT} we explain the connection between Hodge classes and the defining equations of $\MT(A)$, in the case of Fermat Jacobians, and provide explicit computations. A broader treatment in the general setting is given in the second part of this work \cite[Theorem 2.2.2]{part2}. 

The Mumford-Tate conjecture predicts that Hodge classes and Tate classes (the analogue in étale cohomology of Hodge classes) should correspond to one another, see \Cref{conj:MTconj}. In particular, the Hodge ring of $A_\C$ imposes constraints on the $\ell$-adic Galois representation attached to $A/\Q$.
This highlights one of the central themes of our work: when $A$ (or a power $A^r$) is degenerate, the Mumford-Tate group $\MT(A)$ is defined by additional algebraic equations and is accordingly “smaller” -- for example, in terms of dimension -- than in the generic case. In turn, this is reflected on the arithmetic side: the $\ell$-adic monodromy group $\Gl$ is smaller and often disconnected.

\subsubsection*{The Sato-Tate group}
We describe one further invariant, the Sato-Tate group, which, while not crucial for the present work, forms the main object of study of the companion paper \cite{part2}.
Roughly speaking, the algebraic Sato-Tate conjecture predicts that the (possibly disconnected) monodromy groups of the various $\ell$-adic Galois representations of an abelian variety are all interpolated by a single algebraic group defined over $\Q$, called the \textit{algebraic Sato-Tate group} $\operatorname{AST}(A)$. Cantoral-Farf\'{a}n and Commelin \cite{MR4530050} have shown that the algebraic Sato-Tate conjecture is true for all abelian varieties that satisfy the Mumford-Tate conjecture, hence in particular for all CM abelian varieties, since the Mumford-Tate conjecture is known in the CM case \cite{Pohlmann}.
When the algebraic Sato-Tate conjecture holds, one can define the Sato-Tate group $\operatorname{ST}(A)$ as a maximal compact subgroup of the complex points of $\operatorname{AST}(A)$.

The interest in the Sato-Tate group comes from the (generalized) Sato-Tate conjecture, which predicts the asymptotic distribution of characteristic polynomials of Frobenius of an abelian variety $A$ in terms of $\ST(A)$. This conjecture has been proven for elliptic curves with complex multiplication (CM) and for those defined over totally real number fields (see \cite{Barnet2011, Clozel2008,Harris2010,Taylor2008}), as well as for abelian varieties of arbitrary dimension with potential complex multiplication \cite[Proposition 16]{Joh17}.

At the moment, there is no general strategy to compute the Sato-Tate group of a given abelian variety though there has recently been progress for small dimension and for nondegenerate abelian varieties. Classification results for the Sato-Tate groups of dimension 2 and 3 abelian varieties are given in \cite{Fite2012} and \cite{fitekedlayasuth2023}. 
For computations of the Sato-Tate groups of nondegenerate abelian varieties, see, for example, \cite{Arora2016, EmoryGoodson2022,EmoryGoodson2024, FGL2016, FiteLorenzo2018, GoodsonCatalan,GoodsonHoque2024,KedlayaSutherland2009, LarioSomoza2018}.

Many of these results were made possible by work of Banaszak and Kedlaya, who proved \cite[Theorem 6.1]{MR3502937} that the algebraic Sato-Tate group of a `stably non-degenerate' abelian variety coincides with its twisted (decomposable) Lefschetz group $\operatorname{TL}_A$ \cite[Definition 3.4]{MR3502937}. This is the subgroup of $\GL_V$, where $V$ is the first Betti cohomology of $A$, 
given by those automorphisms that preserve the polarization up to scalars and commute with the endomorphisms of $A$ up to the Galois action. %
However, this method %
does not extend to degenerate abelian varieties. This is one of the limitations we aim to overcome in the second part of this work \cite{part2}: we describe a generalization of the twisted Lefschetz group that takes into account all Hodge classes, not just the endomorphisms and the polarization.

\smallskip

Regarding our main family of examples, the Sato-Tate groups of certain varieties $J_m$ (and twists thereof) were described in work of Fité and Sutherland, see \cite{MR3218802} for $m=6$ and \cite{MR3502940} for $m=8$. These results, however, don't provide a uniform framework for computing %
$\operatorname{ST}(J_m)$
for arbitrary $m$, and all fall within the context of nondegenerate abelian varieties (all abelian varieties of dimension up to 3 are nondegenerate).
In \cite{part2}, we expand on the results developed in this paper to solve completely the problem of describing the Sato-Tate group of $J_m$ for every $m$. The crucial ingredient is the determination of the Tate classes on the varieties $J_m$ and their powers, together with the action of $\Gal\left(\overline{\Q}/\Q\right)$ on them. This builds crucially on the results of the present paper, in which we find generators for the space of Tate classes. In \cite{part2}, our main tasks will be to describe the Galois action on these Tate classes and to explain how to derive a computable description of the Sato-Tate group from this information. We emphasize that this paper already provides a description of the action of the Galois group $\operatorname{Gal}(\overline{\Q}/\Q(\zeta_m))$ on the space of Tate classes. The main difficulty that we address in \cite{part2} is extending this description to the full Galois group $\operatorname{Gal}(\overline{\Q}/\Q)$.

\subsubsection*{Statement of results}
In this paper, we focus primarily on the invariants that we described above -- the Mumford-Tate group, the field of definition of the endomorphisms, and the field of connected monodromy -- for the case of the Fermat Jacobians $J_m$.

In \Cref{subsec:MTcomp} we explain how to obtain explicit equations for $\operatorname{MT}(J_m)$. Due to the presence of complex multiplication, $\operatorname{MT}(J_m)$ is an algebraic torus, hence its base-change to an algebraically closed field can be realized as a group of diagonal matrices. More precisely, it turns out that the base-change of $\operatorname{MT}(J_m)$ to $\mathbb{Q}(\zeta_m)$ is naturally a subgroup of the diagonal torus $\mathbb{G}_{m, \mathbb{Q}(\zeta_m)}^{2g}$ of $\operatorname{GL}_{2g, \mathbb{Q}(\zeta_m)}$. Any such subgroup is the intersection of the kernels of finitely many homomorphisms $\mathbb{G}_{m}^{2g} \to \mathbb{G}_m$, each of which is of the form
\[
\operatorname{diag}(x_1, \ldots, x_{2g}) \mapsto x_1^{d_1} \cdots x_{2g}^{d_{2g}}
\]
for certain integers $d_1, \ldots, d_{2g}$. We show how to compute a generating set for the group of homomorphisms whose kernels cut out $\operatorname{MT}(J_m)_{\mathbb{Q}(\zeta_m)}$ and call it a \textit{set of equations} for the Mumford-Tate group (see \Cref{def: equation for MT} for more details).

The connected monodromy field is a much subtler invariant. We describe it completely in the following statement, where -- for each element $f_i =x_1^{d_1} \cdots x_{2g}^{d_{2g}}$ of a set of equations for the Mumford-Tate group -- the notation $\Gamma(f_i)$ denotes an explicit algebraic number which can be computed from the exponents $d_i$ using Euler's $\Gamma$ function (see \Cref{def: gamma value associated to equation for MT} for the precise definition).

\begin{restatable*}{theorem}{maincorollary}
{\label{cor: connected monodromy field for odd m in terms of equations}}
    Let $m \geq 3$ be an odd positive integer and $J_m / {\Q}$ be the Jacobian of the smooth projective curve {over $\Q$} with affine equation $y^2=x^m+1$.
    Let $f_1,\, f_2, \,  \dots,\, f_r$ be a finite set of equations for $\MT(J_m)$. The field $\Q(\varepsilon_{J_m})$ is generated over $\Q$ by the (algebraic) complex numbers
    $$\zeta_m,\, \Gamma(f_1),\, \Gamma(f_2), \,  \dots, \,\Gamma(f_r).$$
\end{restatable*}
We also establish similar results in the case of even $m$, but they are slightly more cumbersome to state: see \Cref{th:Kconn-even}. We remark that the numbers $\Gamma(f_i)$ are defined in terms of Euler's $\Gamma$ function, so they are easy to determine numerically, but not as exact algebraic numbers. We solve this problem in \Cref{th:exactly-gamma}, whose proof gives an algorithm to compute them exactly. As already mentioned, in \Cref{section:MT} we also show how to obtain algorithmically a finite set of equations for $\MT(J_m)$, so that \Cref{cor: connected monodromy field for odd m in terms of equations} gives a complete algorithmic recipe to determine $\Q(\varepsilon_{J_m})$ for any odd $m$.

We note that \Cref{cor: connected monodromy field for odd m in terms of equations} is in essence a statement about (exceptional) algebraic classes and their fields of definition: to prove it, for each equation $f_i$ we construct a Tate class whose field of definition over $\Q(\zeta_m)$ is generated by $\Gamma(f_i)$. The equations $f_i$ to consider, in turn, essentially correspond to generators of the Hodge ring of $J_m$. Since the Mumford-Tate conjecture is known to hold for the CM abelian variety $J_m$, the spaces of Tate classes and Hodge classes determine each other, so that, as already pointed out, the knowledge of Hodge classes on $J_m(\C)$ gives us direct information on the Galois representations attached to $J_m/\Q$.

\Cref{cor: connected monodromy field for odd m in terms of equations} allows us to compute the connected monodromy field of $J_m$ for all odd $m \leq 105$, see \Cref{example: odd numbers up to 105}. This yields many natural examples of degenerate abelian varieties $J_m$ such that $\Q(\varepsilon_{J_m})$ strictly contains $\Q(\End(J_m))$, the most exotic one being $J_{105}$, for which $\Q(\varepsilon_{J_{105}})/\Q(\End(J_{105}))$ has degree $8$.

\subsubsection*{Results in families}

Some of the challenges of working with the abelian varieties $J_m$ are easier to overcome when $m=p^k$, for $p$ an odd prime and $k \geq 1$. Our general results allow us to prove properties of $J_m$ for all $m=p^k$.
The first one is a generalization of a theorem of Kubota \cite[Theorem 2]{Kubota1965}, who proved that $J_p$ is a simple, nondegenerate CM abelian variety. %

\begin{restatable*}{theorem}{nondeg}
\label{th:nondegenarcy}
Let $p$ be an odd prime number, $k$ a positive integer, and $X$ be any simple factor of $J_{p^k}$. %
Then $X$ is a nondegenerate abelian variety.
\end{restatable*}

We note that this confirms and extends Conjecture 5.9 of \cite{Heidi}. Our next result goes in the opposite direction, showing that, while each simple factor of $J_{p^k}$ is non-degenerate, the full Jacobian $J_{p^k}$ itself is degenerate as soon as $k>1$. This degeneracy manifests as an exceptionally small Mumford-Tate group:

\begin{restatable*}{proposition}{MTprojection}
\label{th: MT projection prime power case}
    Let $p$ be an odd prime number, $k$ be a positive integer, and let $X_{p^k}$ be the largest simple factor of $J_{p^k}$. The projection $J_{p^k} \to X_{p^k}$ induces an isomorphism $\MT(J_{p^k})\to \MT(X_{p^k})$ between the corresponding Mumford-Tate groups.
\end{restatable*}

These two results do not necessarily hold for more general $m$. For example, Shioda shows in \cite{Shioda3} that the largest simple factor of the Jacobian $J_{21}$ is degenerate. %
Moreover, for $m=15$, the simple factors of $J_{15}$ are all nondegenerate but the projection of the Mumford-Tate group of $J_{15}$ to the Mumford-Tate group of its largest simple factor is an isogeny, not an isomorphism (see \cite{Heidi}). We study the specific case of $m=15$ further in Examples \ref{ex:m15-equations}, \ref{ex: using Kedlaya}, and \ref{ex: connected monodromy field for m=15}.

Another general property that we can prove for all $m$ concerns the relationship between the field of definition of the endomorphisms $\Q(\End(J_m))$ and the connected monodromy field $\Q(\varepsilon_{J_m})$:
\begin{restatable*}{theorem}{TheoremMultiquadratic}\label{thm: Kconn is multiquadratic}
    Let $m \geq 3$ be an integer and $J_m/\mathbb{Q}$ be the Jacobian of the smooth projective curve with affine equation $y^2=x^m+1$. The connected monodromy field $\Q(\varepsilon_{J_m})$ is a multiquadratic extension of $K=\mathbb{Q}(\zeta_m)$ (and therefore also of $\Q(\End(J_m))$, by \Cref{prop: endomorphism field easy containment}).
\end{restatable*}

\subsection*{Organization of the paper}
In \Cref{sect: preliminaries} we review some of the basic tools used throughout the paper: Mumford-Tate groups, the theory of complex multiplication, and the de Rham cohomology of the Jacobians $J_m$. In \Cref{section:decomposition} we completely describe the decomposition up to isogeny of $J_m$, the field of definition of the simple factors and of the isogenies between them, and their CM types. In \Cref{section:MT} we describe how to compute explicit equations for the Mumford-Tate group of $J_m$ (seen as an algebraic subgroup of a matrix group) and show how to construct a Tate class on a power of $J_m$ starting from any equation of $\MT(J_m)$ of a specific form. In \Cref{section:FermatVarieties} we apply the theory of Fermat varieties -- namely, the projective hypersurfaces $X_m^n : x_0^m + \cdots + x_{n+1}^m = 0$ -- to embed in a Galois-equivariant way the Tate classes constructed in \Cref{section:MT} in the cohomology of $X_m^n$. From this, we deduce a description of the action of $\Gal\left(\overline{\Q}/\Q(\zeta_m) \right)$ on these classes, which suffices to determine the monodromy field $\Q(\varepsilon_{J_m})$. We also describe the effect of twisting the abelian variety $J_m$ on its connected monodromy field. In \Cref{sec:deligne} we combine the previous results to determine the field $\Q(\varepsilon_{J_m})$ as the field generated by algebraic numbers that can be expressed in terms of the $\Gamma$ function. We also explain how to use properties of $\Gamma$ to reduce the computation of these special values to a purely algebraic procedure.
In \Cref{section:further-properties}, we prove the results mentioned above concerning the case where $m$ is a prime power.


\subsection*{Acknowledgements} We thank Gregory Pearlstein for a useful discussion on Hodge theory. We are grateful to Johan Commelin, Nir Elber, Francesc Fité, and Drew Sutherland for their comments on the first version of this manuscript. We thank the anonymous referee for constructive comments that helped improve the paper.

H.G. was supported by NSF grant DMS-2201085. D.L.~was supported by the University of Pisa through grant PRA-2022-10 and by MUR grant PRIN-2022HPSNCR (funded by the European Union project Next Generation EU). D.L. and A.G.~are members of the INdAM group GNSAGA.

\section{Preliminaries}\label{sect: preliminaries}
In this section we introduce the notion of the connected monodromy field of an abelian variety and review the two main tools we will need in the rest of the paper.
The first is the Mumford-Tate group of an abelian variety $A$, which we use (exploiting the fact that the Mumford-Tate conjecture holds for $J_m$) to handle Tate classes in étale cohomology in \Cref{section:MT} and \Cref{section:FermatVarieties}.
The second is the theory of complex multiplication (CM). The abelian varieties $J_m$ have a non-trivial endomorphism ring and decompose into simple abelian varieties of CM type. In \Cref{subsec:MTcomp} we use the explicit description of the Galois representations provided by CM theory to compute the Mumford-Tate group of $J_m$; the results of Deligne in \Cref{sec:deligne} are also based on CM theory. Apart from the basic definitions, here we content ourselves with stating a criterion for the simplicity of an abelian variety of CM type (\Cref{prop:LangSimple}), which will be used in \Cref{section:decomposition} to obtain the isogeny decomposition of $J_m$ into simple factors.

Finally, we describe a basis -- denoted by $x^i \, dx/y$ -- of the De Rham cohomology of $J_m$.

\subsection{The connected monodromy field}
{\label{section: intro connected monodromy field}}
Let $A$ be an abelian variety defined over a number field $K$ and denote its dimension by $g= \dim A$. Fix an algebraic closure $\overline{K}$ of $K$ and denote by $\Gamma_K = \Gal(\overline{K}/K)$ the corresponding absolute Galois group. Let $\ell$ be a prime number. Consider the Tate module $T_\ell A = \varprojlim_n A[\ell^n] $ and set $V_\ell A = T_\ell A \otimes_{\Z_\ell} \Q_\ell$. We denote by $\rho_{A, \ell}$ the Galois representation
\[ \rho_{A, \ell}\colon \Gamma_K \to \Aut(V_\ell A) \simeq \GL_{2g}(\Q_\ell) \]
induced by the natural $\Gamma_K$-action on the Tate module.
Define the $\ell$-adic monodromy group $\Gl$ as the Zariski closure inside $\GL_{V_\ell A} \cong \GL_{2g, \Q_\ell}$ of the image of $\rho_{A, \ell}$:
\[ \mathcal{G}_\ell = \overline{\rho_{A,\ell}\left( \Gamma_K \right)}\subseteq \GL_{2g, \Q_\ell}. \]
This defines $\Gl$ as a group scheme over $\Q_\ell$.

The monodromy group $\Gl$ might not be connected. In this case, consider the composition of the Galois representation $\rho_{A,\ell}$ with the natural projection on the group of components:
\[ \epsilon_\ell \colon \Gamma_K \to \Gl(\Q_\ell) \to \Gl(\Q_\ell)/\Gl(\Q_\ell)^0. \]
Notice that $\Gl(\Q_\ell)$ is Zariski-dense in $\Gl$ by construction, hence every geometric connected component of $\Gl$ contains a $\Q_\ell$-point. This shows that $\Gl(\Q_\ell)/\Gl(\Q_\ell)^0$ agrees with $\left(\Gl/\Gl^0\right)(\Q_\ell)$, the $\Q_\ell$-points of the group of components of $\Gl$.
The kernel of $\epsilon_\ell$ is an open subgroup of $\Gamma_K$, which Serre proved to be independent of the prime $\ell$ \cite[no.~133]{serre-IV} (see also \cite[Proposition 6.14]{MR1150604}).

\begin{definition}
    The \emph{connected monodromy field} $\KconnA$ is the field corresponding to the open subgroup $\ker\epsilon_\ell \subseteq \Gamma_K$ under Galois correspondence, that is, the subfield of $\overline{K}$ fixed by the open subgroup $\ker\epsilon_\ell$.
\end{definition}
The connected monodromy field $\KconnA$ is the minimal extension $L/K$ such that the Zariski closure of $\rho_{A,\ell}(\Gamma_L)$ in $\Gl$ is connected \cite[Proposition 2.3]{zywina2019effective}. 
Notice that $\KconnA/K$ is a finite Galois extension independent of the prime $\ell$.

We denote by $\KendA$ the minimal extension of $K$ over which all the geometric endomorphisms of $A$ are defined. We always have the containment $\KendA\subseteq \KconnA$, as proven, for example, in \cite[Proposition 2.10]{MR1355128} (see also \cite[Remark 2.4.6]{cantoralfarfan2023monodromy}).

\subsection{The Mumford-Tate group} {\label{sub:MTdef}}
We follow Moonen's expository paper \cite{moonen} and identify Hodge structures with representations of the Deligne torus $\mathbb{S} \colonequals \operatorname{Res}_{\C/\mathbb{R}}(\mathbb{G}_{m,\C})$. We describe in these terms the Hodge decomposition of the singular homology of an abelian variety $A$ and use it to define the Mumford-Tate and Hodge groups of $A$. The aim of this section is to highlight the connection between Hodge classes (in singular cohomology) and Tate classes (in étale cohomology), as made precise by the Mumford-Tate conjecture.

\medskip

Let $A$ be a complex abelian variety and consider the $\Q$-vector space $V = \HH^1(A(\C),\Q)$. Its complexification is naturally endowed with a (pure) Hodge structure of weight $1$, namely a decomposition
\[ V \otimes \C = \HH^1(A(\C),\C) = \HH^{1,0} \oplus \HH^{0,1}, \]
where $\HH^{1,0}$ and $\HH^{0,1}$ are $\C$-vector subspaces that are complex-conjugates of each other \cite[Chapter 1.4]{BirkenhakeLange2004}.
Let $\GL_V$ be the general linear group of $V$, considered as a group scheme over $\Q$, and $\mathbb{S}= \Res_{\C/\R} \mathbb{G}_{m,\C}$ be the Deligne torus. Define the weight map $ w\colon \mathbb{G}_{m,\R} \to \mathbb{S}$ as the algebraic map corresponding to the natural inclusion $\mathbb{G}_{m,\R}(\R)=\R^\times \hookrightarrow \C^\times = \mathbb{S}(\R) $, and the norm map Nm$\colon \mathbb{S} \to \mathbb{G}_{m, \mathbb{R}}$ as the one corresponding to the map $z \mapsto z\overline{z}$ on real points.
We denote by $U_1$ the kernel of the norm map, also called the {norm one subtorus} of $\mathbb{S}$. We have a decomposition $\mathbb{S}(\R)=\C^\times \simeq U_1(\R) \times w(\R^\times)$.
The Hodge structure on $V_\R$ corresponds to a representation $h \colon \mathbb{S} \to \GL_{V,\R}$ such that $h \circ w\colon \mathbb{G}_{m,\R} \to \mathbb{S} \to \GL_{V,\R}$ is defined over $\Q$ \cite[3.3]{moonen}.

\begin{definition}
    The Mumford-Tate group $\MT(A)$ of the abelian variety $A$ is the smallest algebraic subgroup of $\GL_V$, defined over $\Q$, such that the representation $h$ factors through $\MT(A)_\R$.
    The Hodge group $\Hg(A)$ is the smallest algebraic subgroup of $\GL_V$, defined over $\Q$, such that $h_{\mid U_1}\colon U_1 \to \GL_{V,\R}$ factors through $\Hg(A)_\R$.
\end{definition}

One can show that the Mumford-Tate group $\MT(A)$ is a connected reductive algebraic group containing the subgroup of homotheties $\mathbb{G}_m \subseteq \GL_V$ \cite[Proposition 3.6]{Deligne}, \cite[4.9]{moonen}. The Hodge group is the connected component of the identity of the intersection $\MT(A)\cap\SL_V$ \cite[Remark 17.3.1]{BirkenhakeLange2004}.  

If $A$ is an abelian variety defined over a number field $K$, the above constructions apply to the complex abelian variety $A \times_{\sigma} \Spec \C$, where $\sigma$ is any fixed embedding $\sigma : K\hookrightarrow \C$.
The choice of embedding $K \hookrightarrow \C$ does not affect any of the above constructions by work of Deligne \cite[Theorem 2.11]{Deligne}.

We now recall the notion of Hodge classes and their relation to the Mumford-Tate group. Fix non-negative integers $r,\, s$. The tensor product $T^{r,s}V \colonequals V^{\otimes r} \otimes {V^\vee}^{\otimes s}$ inherits a Hodge structure from $V$, of weight $n=r-s$.
An element $v \in T^{r,s}V$ is fixed by the action of the Hodge group $\Hg(A)$ if and only if it is of type $(p,p)$ for some integer $p$ in the Hodge decomposition (that is, $v$ is an element of the rational vector space $T^{r,s}V$ that lies in the $(p,p)$-component of the Hodge decomposition of the complexification $T^{r,s}V \otimes \C$). Likewise, $v$ is fixed by $\MT(A)$ if and only if it is of type $(0,0)$ \cite[4.4]{moonen}. When $s=0$, we write simply $T^rV$ for $T^{r,0}V$.

Fix a positive integer $p$. The Tate structure $\Q(p)$ is the unique one-dimensional rational Hodge structure of weight $-2p$ and type $(-p,-p)$. Equivalently, $\Q(p)$ is the Hodge structure corresponding to the representation $h \colon \mathbb{S} \to \GL(\Q(p))_\R = \mathbb{G}_{m,\R}$ given by the $p$-th power of the norm map \cite[3.4]{moonen}.
The Tate twist $V(p)$ of a Hodge structure $V$ is the tensor representation $V\otimes \Q(p)$.

\begin{definition}{\label{definition:Tate-class}}
    Fix a positive integer $p$ and let $A$ be a complex abelian variety. Write $V = \HH^1(A(\C), \Q) $. An element $v \in \HH^{2p}(A(\C), \Q)(p) \simeq \bigwedge^{2p} V \otimes \Q(p) $ is called a \textit{Hodge class} if it is of type $(0,0)$ in the Hodge decomposition of $\HH^{2p}(A(\C), \C)(p)$. From the argument above \cite[4.4]{moonen}, we can also define Hodge classes as the rational elements of $\HH^{2p}(A(\C), \C)(p)$ that are fixed by the action of the Mumford-Tate group $\MT(A)$.
\end{definition}

Let $g = \dim A$. Since $V$ is a pure Hodge structure of weight $1$ and dimension $2g$, the top exterior power $\bigwedge^{2g}V \simeq \HH^{2g}(A(\C),\Q)$ is a $1$-dimensional Hodge structure of pure weight $2g$ and type $(g, g)$, so it is isomorphic to the Tate structure $\Q(-g)$.
The vector space $\bigwedge^{2g}V \simeq \HH^{2g}(A(\C),\Q)$ has a natural action of $\MT(A)$, given by the composition of the natural action of $\MT(A)$ on $V$ with the determinant. In particular, $\MT(V)$ acts as $\det^{-1}$ on $\Q(g)$.

\medskip

Note that the canonical map $g \mapsto (g^*)^{-1}$ induces an isomorphism between the Mumford-Tate groups of $V$ and $V^\vee$ \cite[Remark 1.8]{moonen2}:
we shall then freely think of $\MT(A)$ as being a subgroup of $\GL_V$ or $\GL_{V^\vee}$.

Suppose that the abelian variety $A$ is defined over a number field $K$ and let $\ell$ be a rational prime number. Through the singular-to-étale isomorphism \cite[Chapter III, Theorem 3.12]{MilneEtaleCo}
\begin{equation}{\label{eq: H1et is dual to Tate module}}
    V {\otimes \Q_\ell} = \HH^1(A(\C), \Q) {\otimes \Q_\ell} \simeq \HH^1_{\text{ét}}(A_{\overline{K}}, \Q_\ell) = (V_\ell A)^\vee,
\end{equation}
the base change $\MT(A)_{\Q_\ell}$ can be identified  with an algebraic subgroup of $\GL\left(V_\ell A\right)$. This allows for a comparison with the image of the $\ell$-adic Galois representation and leads to the following famous conjecture.

\begin{conjecture}[Mumford-Tate]{\label{conj:MTconj}}
    The identity component of the $\ell$-adic monodromy group is the base change to $\Q_\ell$ of the Mumford-Tate group: $\mathcal{G}_{\ell}^0(A) = \MT(A)_{\Q_\ell}$.
\end{conjecture}

The conjecture has been proved for abelian varieties of CM-type in \cite{Pohlmann}, building on the work of Shimura-Taniyama \cite{Shimura}.

A Tate class is an element of $\HH^{2p}_{\text{ét}}(A_{\overline{K}}, \Q_{\ell}(p))$ fixed by the action of an open subgroup of the absolute Galois group $\Gamma_{K}$. 

\begin{remark}\label{rmk: MT iff Tate and Hodge classes biject}
\Cref{conj:MTconj} implies that, through the singular-to-étale isomorphism \cite[Chapter III, Theorem 3.12]{MilneEtaleCo}
\[\HH^{2p}(A(\C),\Q)\otimes \Q_\ell(p) \simeq \HH_{\text{ét}}^{2p}(A_{\overline{K}}, \Q_\ell(p)), \]
the span of Hodge classes on the left is mapped bijectively onto the space of Tate classes on the right.
\end{remark}

\subsection{Complex Multiplication}{\label{sub:CM}}
In this section we give an introduction to abelian varieties with complex multiplication (CM); for a complete exposition, see \cite{Lang}. We define the CM type of an abelian variety $A$ with CM and explain its connection with the geometric decomposition of $A$ up to isogeny. In \Cref{section:decomposition} we use this to obtain a complete factorization of $J_m$.

\medskip

A CM field $E$ is a totally imaginary quadratic extension of a totally real number field $E_0$. For a CM field of degree $[E\colon \Q]=2n$, we define a CM-type as a collection $\Phi=\{ \varphi_1, \dots, \varphi_n \}$ of $n$ distinct complex embeddings $\varphi_i\colon E \hookrightarrow \C$ such that for every $\varphi_i \in \Phi$ the conjugate embedding $\overline\varphi_i$ is not in $\Phi$. %
Notice that a CM-type defines a partition
\[ \Hom(E, \C) = \Phi \sqcup \overline\Phi. \]
Let $\tilde{E}/\Q$ be the Galois closure of $E$, with Galois group $G$, and let $H<G$ be the subgroup corresponding to $E$. Note that $\tilde{E}$ is a CM field. We can identify $\Hom(E, \C)$ with the set of left cosets $G/H$. 

Let $F/E$ be a finite extension of CM fields. Given a CM-type $\Phi$ on $E$, one can lift $\Phi$ to the subset $\Phi_F \subset \Hom(F, \C)$ of all complex embeddings $\varphi \colon F \hookrightarrow \C$ whose restriction to $E$ lies in $\Phi$, i.e., we set $\Phi_F \colonequals \left\{ \varphi  \in \Hom(F,\C) \bigm\vert \varphi_{\mid E} \in \Phi \right\}$.
We say that a CM-type is \emph{simple} if it cannot be obtained as the lift of a CM-type from a proper CM-subfield. One shows from \cite[Chapter I, Lemma 2.2]{Lang} that a CM-type is simple if and only if the stabilizer $\{ \sigma \in G \colon \Phi_{\tilde{E}}\cdot \sigma = \Phi_{\tilde{E}} \} \subseteq G$ coincides with $H$ .

The reflex group $H^\ast$ of $(H, \Phi)$ is defined as the stabilizer for the left action of $G$ on $\Phi_{\tilde{E}}$: we set $H^\ast =\{ \sigma \in G \colon \sigma \cdot \Phi_{\tilde{E}} = \Phi_{\tilde{E}} \}$. The subfield $E^\ast$ fixed by $H^\ast$ is called the \textit{reflex field} of $(E, \Phi_E)$. The reflex field is a CM field equipped with the reflex CM-type $\Phi^\ast := \{  g^{-1}H^\ast \colon g \in G, \, gH \in \Phi \},$
see \cite[Chapter I, Theorem 5.1]{Lang}. The reflex norm is the map
\[ N_{\Phi^\ast} \colon (E^\ast)^\times \to E^\times, \quad x \mapsto \prod_{\varphi \in \Phi^\ast} \varphi(x). \] 

Let $ A $ be an abelian variety of dimension $g$ over a number field $K$. {Fix an embedding $\sigma : K \hookrightarrow \C$ and write the complex uniformization of $A(\C) \colonequals (A \times_\sigma \Spec \C)(\C)$ as $W/\Lambda$, where $W$ is a $g$-dimensional complex vector space and $\Lambda\subseteq W$ is a full-rank lattice.}
Let $E/\Q$ be {a CM field} of degree $2g$. Suppose that we have an embedding
\[ \iota\colon E \hookrightarrow \End^0(A_{\overline{K}}) \colonequals \End(A_{\overline{K}})\otimes_\Z \Q. \] %
The rational representation $R_\Q$ attached to the complex torus $A(\C) \simeq W/\Lambda$ is the action of the endomorphism algebra on the {$\Q$-}vector space $V = \Lambda \otimes \Q$. Composing with the embedding $\iota$, the same vector space $V$ becomes a linear representation of $E$.
{The base-change $V \otimes_\Q \C$} splits as the sum $\bigoplus_{i = 1}^{2g} \varphi_i$ over the complex embeddings $E\hookrightarrow \C$. Notice that we have a similar splitting $V \otimes_\Q F \cong \bigoplus_{i = 1}^{2g} \varphi_i$ for any number field $F$ containing the Galois closure of $E$, so that we can consider the embeddings $\varphi_i$ as taking values in $F$.
The rational representation is also the sum of the {analytic} representation $R_\C$ (defined as the action of the endomorphism algebra on $W$) and its conjugate $\overline{R_\C}$. The first of these representations decomposes as $R_\C = \bigoplus_{i = 1}^{g} \varphi_i$ for a choice of half the complex embeddings, no two of which are complex conjugates of each other. {Thus, the set $\Phi =  \{\varphi_i \colon i=1, \dots, g\}$ is a CM-type of $E$, called the CM-type of the abelian variety $A$ \cite[Chapter I, §3]{Lang}. {Note that the CM-type also depends on the embedding $\iota$, but we will suppress this dependence from the notation.}}
In this situation, we say that $A$ has complex multiplication by $E$ of type $\Phi$. {We will mostly be interested in the case of $E$ being the cyclotomic field $\Q(\zeta_m)$.}

\begin{proposition}{\label{prop:LangSimple}}
    Let $ A $ be an abelian variety with complex multiplication by a CM field $E$ and CM-type $\Phi$. Then $A$ is geometrically simple if and only if $\Phi$ is simple.
\end{proposition}
\begin{proof}
    See \cite[Proposition 26, \S II.8.2]{Shimura} or {\cite[Chapter I, Lemma 3.2 and Theorems 3.5 and 3.6]{Lang}}.
\end{proof}

\begin{example}{\label{ex:primecase}}
    Let $p$ be an odd prime number and let $J_p/\Q$ be the Jacobian of the hyperelliptic curve $C_p \colon y^2=x^p+1$. This abelian variety has dimension equal to the genus $g$ of the curve, where $2g = p-1$. Notice that the {group $\mu_p$ of $p$-th roots of unity} acts on $C_p$ via the map $\zeta_p \colon (x,y) \mapsto (\zeta_px, y)$. This gives an embedding
    \[ \iota\colon  \Q(\zeta_p) \hookrightarrow \End^0(J_{p,\Bar{\Q}}). \]%
    To compute the corresponding CM type, consider the complex uniformization given by the cotangent space
    $$\C\langle x^{i} \,{dx}/{y} \colon i = 0, \dots, g-1 \rangle \to J_p $$
    and notice that the action of $\zeta_p$ in this basis is already diagonal, with characters $\chi_i\colon \zeta_p \mapsto \zeta_p^{i+1}$ for $0 \leq i \leq g-1$. These correspond to the complex embeddings 
    \[\Phi_p = \{ a \colon 1 \leq a \leq g, \, (a,p)=1 \} \subseteq (\Z/p\Z)^\times \simeq \Gal( \Q(\zeta_p)/\Q ). \]
    We will show below in \Cref{lemma:simpleCM} that this CM-type is simple, and therefore $J_p$ is {absolutely} simple by \Cref{prop:LangSimple}. {This result is well-known (see for example \cite[Theorem 2]{Kubota1965} or \cite[II.8.4]{Shimura}), but \Cref{lemma:simpleCM} is more general, in that it describes the simplicity (or lack thereof) of the CM type of the Jacobian of $y^2=x^m+1$ when $m$ is not necessarily prime.}
\end{example}

\subsection{De Rham cohomology of \texorpdfstring{$C_m$}{Cm}}{\label{sub:bases}}
Let $C_m/\Q$ be the unique smooth projective curve birational to the affine curve $y^2=x^m+1$ and let $J_m/\Q$ be its Jacobian. {We denote by $g = \lfloor \frac{m-1}{2}\rfloor$ the genus of $C_m$.}
We carefully describe a natural basis for $\HH^1_{\operatorname{dR}}(C_m(\C), \C)$, which we use throughout the paper.

\medskip

Consider the map $x : C_m(\C) \to \mathbb{P}^1(\C)$ and the two open sets $\mathcal{U}_0 = C_m(\C) \setminus x^{-1}(0)$ and $\mathcal{U}_\infty = C_m(\C) \setminus x^{-1}(\infty)$. It is immediate to see that $\mathcal{U}_0$ and $ \mathcal{U}_\infty$ are open and connected and that their union is all of $C_m(\C)$. Using the fact that $\mathcal{U}_0 \cap \mathcal{U}_\infty$ is also connected, the Mayer-Vietoris sequence of the cover $(\mathcal{U}_0,\, \mathcal{U}_\infty)$ yields
\[
\begin{array}{ccccccc}
0 & \to & \HH^1(C_m(\C), \C) & \to & \HH^1(\mathcal{U}_0, \C) \oplus \HH^1(\mathcal{U}_\infty, \C) & \to & \HH^1(\mathcal{U}_0\cap \mathcal{U}_\infty,\mathbb{C}) \\
&&&& ([\omega_0], [\omega_\infty]) & \mapsto & [\omega_0|_{\mathcal{U}_0\cap \mathcal{U}_\infty} - \omega_\infty|_{\mathcal{U}_0\cap \mathcal{U}_\infty}]
\end{array}
\]
where for simplicity we write $\HH^1$ for $\HH^1_{\operatorname{dR}}$. That is, a class in the de Rham cohomology of $C_m(\C)$ can be represented by a pair of closed forms $(\omega_0, \omega_\infty)$, defined and regular over $\mathcal{U}_0, \mathcal{U}_\infty$ respectively, and such that $\omega_0-\omega_\infty$ is an exact form on $\mathcal{U}_0 \cap \mathcal{U}_\infty$. We will represent such data as a triple $(\omega_0, \omega_\infty, f)$, where $\omega_0-\omega_\infty = df$ on $\mathcal{U}_0 \cap \mathcal{U}_\infty$.

Using \cite[Theorem 3.2]{MR3782449}, one can prove that the classes of the following triples $(\omega_0, \omega_\infty, f)$ form a basis of $\HH^1_{\operatorname{dR}}(C_m(\C), \C)$:
\begin{enumerate}
    \item $\omega_{i} = \left( x^{i-1} \frac{dx}{y}, x^{i-1} \frac{dx}{y}, 0 \right)$ for $i=1,\ldots,g$;\label{eqn:basis1}
    \item $\eta_{i} = \left( \frac{-2i}{2x^{i+1}} \frac{dx}{y}, \frac{(2i-m)x^m}{2x^{i+1}} \frac{dx}{y}, \frac{y}{x^i} \right)$ for $i=1, \ldots, g$.\label{eqn:basis2}
\end{enumerate}
Note that the first $g$ basis elements in \eqref{eqn:basis1} generate the subspace $\HH^0(C_m(\C), \Omega^1)$, while the (classes of the) last $g$ in  \eqref{eqn:basis2} span the quotient $\HH^1(C_m(\C), \mathcal{O}_{C_m(\C)})$ appearing in the canonical sequence
\[
0 \to \HH^0(C_m(\C), \Omega^1) \to \HH^1_{\operatorname{dR}}(C_m(\C), \C) \to \HH^1(C_m(\C), \mathcal{O}_{C_m(\C)}) \to 0.
\]
This sequence splits canonically by Hodge theory.
The class $\eta_i$ is represented over $\mathcal{U}_\infty$ by
$ (2i-m)/2 \cdot x^{m-1-i}\, {dx}/{y}$ and we set $\omega_{m-i}=2\eta_i/(2i-m)$ for $i=1, \dots, g$.
{Note that we have thus constructed a basis $(\omega_i)_{i \in I}$ of $\HH^1_{\operatorname{dR}}(C_m(\C), \C)$ where the index set $I$ is given by $\{1,2,\ldots,2g-1, 2g\}$ when $m=2g+1$ is odd and by $\{1,2,\ldots,g,g+2,g+3,\ldots,2g, 2g+1\}$ when $m=2g+2$ is even. Note that there is no basis vector $\omega_{m/2}$. By a slight abuse of notation, we will sometimes write $x^{i-1} dx/y$ for $\omega_i$ (note that for all indices $i$ the class $\omega_i$ can be represented by $x^{i-1} dx/y$ on $\mathcal{U}_\infty$).
}

{From now on, we identify $\Q(\zeta_m)$ with a subfield of $\C$ via the embedding that takes $\zeta_m$ to $\exp(2\pi i/m)$.}
We denote by $\alpha_m$ the automorphism 
\begin{equation}\label{eq: definition of alpha_m}
\alpha_m : (x,y) \mapsto (\zeta_m x,y)    
\end{equation}
of $C_m$; it is defined over $\Q(\zeta_m)$, hence in particular over $\C$. {We denote by $\alpha_m^*$ the automorphism of $\HH^1_{\operatorname{dR}}(C_m(\C), \C)$ induced by $\alpha_m$.}

\begin{proposition}\label{prop:eigenbase}
    The class $\omega_i$ is an eigenvector of $\alpha_m^*$ with eigenvalue $\zeta_m^i$, for $i = 1,\, 2 \,\dots,\, g$ and $i = m-1,\, m-2, \, \dots, \, m-g$.
\end{proposition}
\begin{proof}
This is obviously the case for the holomorphic differentials $x^{i-1}\,{dx}/{y}$ in \eqref{eqn:basis1}. %
As for the classes in (2), we check that they are eigenvectors for $\alpha_m$ upon restriction to $\mathcal{U}_0$ and $\mathcal{U}_\infty$, with the same eigenvalue. This will imply that the classes they represent in $\HH^1_{\operatorname{dR}}(C_m(\C), \C)$ are also eigenvectors for the same eigenvalue. Note that the open sets $\mathcal{U}_0$ and $\mathcal{U}_\infty$ are stable under multiplication by $\zeta_m$. We have
\[
\alpha_m^\ast\left( \frac{-2i}{2x^{i+1}} \frac{dx}{y} \right) = \frac{-2i}{2 \zeta_m^{i+1} x^{i+1}} \cdot \zeta_m \frac{dx}{y} = \zeta_m^{m-i} \cdot \frac{-2i}{2x^{i+1}} \frac{dx}{y},
\]
and similarly
\[
\alpha_m^\ast\left( \frac{(2i-m)x^m}{2x^{i+1}} \frac{dx}{y} \right) = \frac{(2i-m)x^m}{2 \zeta_m^{i+1} x^{i+1}} \cdot \zeta_m \frac{dx}{y} = \zeta_m^{m-i} \cdot \frac{(2i-m)x^m}{2x^{i+1}} \frac{dx}{y}. 
\] 
\end{proof}

\section{The decomposition of Fermat Jacobians}{\label{section:decomposition}}

Denote by $J_m/\Q$ the Jacobian variety of the hyperelliptic curve $C_m/\Q$ defined by the affine equation $y^2=x^m+1$. In this section, we study the factorization of $J_m$ into simple abelian varieties up to isogeny, both over $\Q$ and over its algebraic closure. In doing so, we prove \cite[Conjecture 4.5]{Heidi}, which is stated for the twist of $C_m$ given by the affine equation $y^2=x^m-1$. Partial information about the structure of the decomposition of $J_m$ is scattered throughout the literature (see for example \cite{MR0931215, MR0491708} or Rohrlich's appendix to \cite{MR0480542}), but \Cref{th:Jm-factorization} below is more precise, in that it gives information about the simplicity of the various factors and their fields of definition. We also compute the endomorphism field of $J_m$ for all $m$, see \Cref{th: endomorphism field of Jm complete version}.

\begin{restatable}{theorem}{Jmfactorization}
\label{th:Jm-factorization} Denote by $\phi$ the Euler totient function.
	Let $ m $ be a positive integer. The abelian variety $ J_m $ admits a canonical factorization up to isogeny over $\Q$ which contains exactly one factor $ X_d $ for each positive divisor $ d $ of $ m $ except $1$ and $2$:
    \[ J_m \sim \prod_{\substack{d \mid m,\\ d \neq 1,2}} X_d. \]
    The abelian variety $ X_d $ has dimension $ \phi(d)/2 $. Furthermore:
    \begin{enumerate}
        \item the geometric endomorphism algebra $\End^0(X_{d, \overline{\Q}})$ of $X_d$ contains a copy of $\Q(\zeta_d)$. 
        \item if $d = p$ is an odd prime or $d=4$, then $X_d$ is isogenous to $J_{d} = \Jac(y^2=x^{d}+1)$ over $\Q$.
        \item if $ d $ is odd or $d=4$, then $X_d$ is geometrically simple and has complex multiplication by $\Q(\zeta_d)$.
        \item if $ d = 4k+2 $, then $X_d$ has complex multiplication by $\Q(\zeta_d)$. Moreover, $X_d$ is isogenous to $X_{d/2}$ and {is geometrically} simple; the isogeny is defined over $\Q$.
        \item if $ d = 4k $ with $d \not \in \{ 4,\, 20,\, 24,\, 60 \}$, then $X_d \sim Y_d^2$ is isogenous to the square of a geometrically simple abelian variety $Y_d$ with complex multiplication by $\Q(\zeta_d-\zeta_d^{-1}).$
        The isogeny class of $Y_d$ contains a representative defined over $\Q$; {for this representative,} the isogeny is defined over $\Q(\zeta_d).$
        \item if $d = 20,\, 24,\, 60$, then $X_d \sim Y_d^4$ is isogenous over $\overline{\Q}$ to the 4th power of a geometrically simple abelian variety $Y_d / \overline{\Q}$. 
    \end{enumerate}
    Furthermore, the abelian varieties in the set
    \[
    \{X_d : d \text{ odd}\} \cup
    \{Y_d : d \equiv 0 \pmod 4,\ d>4\}
    \]
    are pairwise non-isogenous over $\Qbar$. The only $\Qbar$-isogeny between $X_4$ and a variety in this set is $X_{4,\Qbar} \sim Y_{12,\Qbar}$.
\end{restatable}

\begin{remark}
    It will be clear from the construction that the isogeny class of the abelian variety $X_d$ depends only on $d$ (and not on the choice of a multiple $m$ of $d$).
\end{remark}

We split the lengthy proof of \Cref{th:Jm-factorization} into smaller steps, which will occupy Sections \ref{subsect: decomposition over Q} through \ref{sect: exceptional cases 20 24 60}.

\subsection{Decomposition over \texorpdfstring{$\Q$}{Q}}\label{subsect: decomposition over Q}

Let $ m = p_1^{e_1} \dots p_r^{e_r} $ be the prime factorization of $ m $. We will argue by induction on the total exponent $ e = \sum_{i=1}^r e_i $. The base case is that of prime $ m $ (i.e., $e=1$), which we already discussed in Example \ref{ex:primecase}. For $m=2$, notice that the curve $y^2=x^2+1$ has genus 0, hence its Jacobian is trivial.
For $m=4$, the elliptic curve $J_4$ is defined by $y^2=x^4+1$ and has complex multiplication by $\Q(i)$.
Denote by $g(m)$ the dimension of $J_m$.
	
Consider the partially ordered set of positive divisors of $ m $. For every two comparable divisors $ d_2 \mid d_1 \mid m $, consider the maps
	\[ C_m \to C_{d_1} \to C_{d_2}, \qquad (x,y) \mapsto (x^{{m}/{d_1}}, y) \mapsto (x^{{m}/{d_2}},y). \]
	The corresponding pull-back maps on the Jacobians $ J_{d_2}\to J_{d_1}\to J_m $ are embeddings, up to isogeny.
	Since all these maps are compatible (they fit in a commutative diagram indexed by the divisibility poset of $ m $), we may identify all these varieties with their image in $ J_m $, where they are contained in one another as in the diagram. We are interested in characterizing their sum $ Y = \sum_{i=1}^r J_{m/p_i} \subseteq J_m $. The inductive hypothesis applies to every $ J_{m/p_i} $.

\begin{lemma}{\label{lemma: decomposition into Xd}}
   The subvariety $Y$ decomposes up to isogeny as the product $\prod_{d \mid m, \, d\neq 1,2,m} X_d$ over all proper divisors $d$ of $m$ distinct from $2$.
\end{lemma}
\begin{proof}    
Notice that, for any two different prime factors $ p_1, p_2 $ of $ m $, any common divisor $ d $ of $ m/p_1 $ and $ m/p_2 $ is a divisor of $ m/p_1p_2 $. The corresponding simple factors $ X_d^{(1)} $ and $ X_d^{(2)} $ of $ J_{m/p_1} $ and $ J_{m/p_2} $ actually come from the inclusion $ J_{m/p_1p_2} \subseteq J_{m/p_1} \cap J_{m/p_2} $ and are therefore the same unique factor.
	
At the level of complex uniformizations, the inclusion $ J_{m/p_i} \subseteq J_m $ corresponds to the inclusion of complex vector spaces
\[
    \C \langle x^{ap_i-1} \, dx/y
    \mid a = 1,\, \dots,\, g(m/p_i)  \rangle \hookrightarrow
    \C\langle x^{j} \, dx/y \mid
    j = 0,\, \dots,\, g(m)-1 \rangle.
\]
The intersection of the uniformizing vector spaces for $ J_{m/p_1} $ and $ J_{m/p_2} $ is the space uniformizing $ J_{m/p_1p_2} $. Thus, the intersection $ J_{m/p_1} \cap J_{m/p_2} $ contains no extra simple factors in addition to those of $ J_{m/p_1p_2} $.
	
From these observations and the inductive hypothesis, we obtain that the sum $ Y $ is the product of a unique abelian variety $ X_d $ for every proper divisor $ d $ of $ m $ (up to isogeny). %
\end{proof}
	
We now focus on the short exact sequence
\begin{equation}{\label{eq: ses splitting J}}
    0 \to Y \to J_m \to X_m \to 0
\end{equation}
defining $X_m$ and the corresponding short exact sequence of uniformizing complex vector spaces. Notice that \Cref{lemma: decomposition into Xd} and Equation (\ref{eq: ses splitting J}) give the first part of the statement of \Cref{th:Jm-factorization}. 
We can compute the dimension of $X_m$ using the inductive hypothesis and the previous lemma:
\[ \dim X_m = \dim J_m - \dim Y = g(m) - \sum_{\substack{d \mid m,\\ d \neq 1,2,m}} \frac{1}{2}\phi(d) = \phi(m)/2. \]
We are left with investigating the simplicity of the quotient abelian variety $ X_m $; in order to do so, we show that it has complex multiplication and determine its CM-type.

\begin{lemma}\label{lemma: CM type of X_m}
    The abelian variety $X_m$ has complex multiplication by $\Q(\zeta_m)$ with CM-type
    \begin{equation}{\label{eq:CMtype-of-X_m}}
        \Phi_m = \{ \sigma_j\colon \zeta_m \mapsto \zeta_m^j \colon (j,m)=1, \, 1\leq j \leq g(m) \}.
    \end{equation}
\end{lemma}
\begin{proof}
The inclusion $Y\to J_m$ corresponds to the inclusion
\[ \sum_{i=1}^r \C\langle x^{ap_i-1} \, dx/y 
\mid a = 1,\, \dots,\, g(m/p_i) \rangle 
\hookrightarrow
\C\langle x^{j}\, dx/y \mid 
j = 0,\, \dots,\, g(m)-1 \rangle. \]
The quotient abelian variety $ X_m $ is uniformized by the quotient vector space, for which we obtain the basis $ \{ x^j \, dx/y \mid j < g(m) \text{ such that } j \not \equiv -1 \mod p_i\} $. 
The endomorphisms $\alpha_m^i \colon (x, y) \mapsto (\zeta_m^i \cdot x, y) $ of $ J_m $ are diagonal in this basis, with characters $ \chi_j \colon \zeta_m \mapsto \zeta_m^{j+1}$. The action thus respects $ Y $ and induces an action of $\Z[\zeta_m]$ on the quotient. Notice that $[\Q(\zeta_m)\colon\Q]=\phi(m)=2\dim X_m$. On the complex uniformization of $X_m$, the action is diagonal (with respect to the basis just described) and we are left with the characters $ \chi_j $ whose exponents are {less than $g(m)$ and} coprime to every prime factor $ p_i $, hence coprime to $ m $. These characters correspond to the field embeddings $$  \sigma_j\colon \zeta_m\mapsto \zeta_m^j \in  \Gal(\Q(\zeta_m)/\Q) \simeq \left({\Z/m\Z}\right)^\times $$ with $ j \leq g(m) $.
\end{proof}

This establishes (1) in the statement of \Cref{th:Jm-factorization}.

\subsection{Simplicity of the factors \texorpdfstring{$X_d$}{Xd}}
The simplicity of an abelian variety with complex multiplication is related to the simplicity of its CM type via \Cref{prop:LangSimple}.

\begin{lemma}{\label{lemma:simpleCM}}
    Let $m$ be a positive integer and let $\Phi_m$ be the CM-type of \Cref{lemma: CM type of X_m}. The stabilizer of $\Phi_m$ for the $(\Z/m\Z)^\times$-action given by left multiplication on itself is
    \begin{itemize}
        \item trivial, if $m$ is odd, or $m=4k+2$ for some positive integer $k$, or $m=4$;
        \item $\{1,\tfrac{m}{2}-1\}$, if $4\mid m$ and $m \not \in \{4, 20,24,60\}$;
        \item $\{1,3,7,9\}$, if $m=20$;
        \item $\{1,5,7,11\}$, if $m=24$;
        \item $\{1,11,19,29\}$, if $m=60$.
    \end{itemize}
\end{lemma}
\begin{proof}
    First, consider the case of odd $m$. Let $a$ be such that $a\Phi_m = \Phi_m$; we want to prove that $a=1$. Observe that $a = a \cdot 1 \in a \Phi_m = \Phi_m$, so $a<\frac{m}{2}$. Let $k \geq 0$ be the unique integer such that $2^ka < \frac{m}{2} < 2^{k+1}a$ (where we have the strict inequalities because $m/2$ is not an integer). We distinguish two cases:
\begin{enumerate}
    \item $2^{k+1} < m/2$: In this case, $2^{k+1} \in {\Phi_m}$ (since $(2,m)=1$) and we get a contradiction. Indeed, this would imply that $2^{k+1}a\in{\Phi_m}$. However, $2^{k+1}a > m/2$ by assumption, and $2^{k+1}a = 2(2^k a) < 2 \cdot \tfrac{m}{2} = m$, so the unique representative of $2^{k+1}a$ modulo $m$ lies in the interval $(m/2, m)$, and hence does not belong to $\Phi_m$.
    \item $2^{k+1} > m/2$: Since $2^ka < \frac{m}{2}$ we obtain $2^{k+1} > m/2 > 2^ka$. Hence, $2 > a$  and so $a=1$.
\end{enumerate}
The cases where $m$ is even follow from the slightly more general \cite[Lemma 5]{Gannon1996}.
\end{proof}

Combining \Cref{lemma:simpleCM} and \cite[Theorems 3.3 and 3.5]{Lang}, we recover the geometric decomposition of $X_m$ into simple abelian varieties as in the statement of \Cref{th:Jm-factorization}.

In particular, for odd $m$, $\Phi_m$ is a simple CM-type, and so $X_m$ is {geometrically} simple by \Cref{prop:LangSimple}. 
For $m=4k+2$, the CM-type is also simple and, therefore, $X_m$ is simple. Consider the involution
\[ \beta\colon C_m \to C_m, \qquad (x,y) \mapsto \left(\frac{1}{x}, \frac{y}{x^{\frac{m}{2}}}\right). \]
On the complex uniformization of the Jacobian, this involution acts by the formula
\begin{equation}{\label{eq:beta-pull-back}}
    \beta^\ast\left(x^i\, {dx}/{y} \right) = -x^{m/2-i-2} \, {dx}/{y},
\end{equation}
hence it swaps the forms with even exponent with those of odd exponent. In particular, $\beta^\ast$ sends the basis $\{x^i \, {dx}/{y} \colon (i+1,m)=1 \}$ of the complex uniformization of $X_m$ to a basis of the complex uniformization of $X_{m/2}$. In fact, if $(i+1,m)=1$, then the corresponding exponent $m/2-(i+1)$ has no odd prime factor in common with $m$, but is of the opposite parity, i.e., even. Thus, $\beta$ provides an isomorphism between $X_m$ and $X_{m/2}$ defined over $\Q$.

This proves (3) and (4) in the statement of \Cref{th:Jm-factorization}.

\subsection{Further decomposition in the case \texorpdfstring{$m \equiv 0 \pmod{4}$}{when m is divisible by 4}}
When $m$ is a proper multiple of $4$, the abelian variety $X_m$ is not geometrically simple. In the non-exceptional cases (i.e., $m\not= 20, 24, 60$), $X_m$ is geometrically isogenous to the square of some geometrically simple variety $Y_m$.
The CM-field of $Y_m$ is the subfield of $\Q(\zeta_m)$ fixed by the order-two subgroup $\{1, \tfrac{m}{2}-1\} \subseteq (\Z/m\Z)^\times \cong \Gal(\Q(\zeta_m)/\Q)$, which is $\Q(\zeta_m-\zeta_m^{-1})$ (notice that $\{1, \tfrac{m}{2}-1\}=\{1\}$ for $m=4$).
We can describe an abelian variety in the isogeny class of $Y_m$.
Equation \Cref{eq:beta-pull-back} still holds in this context, but now $\beta$ restricts to a map $X_m \to X_m$: if $(i+1)$ is coprime to $m$, then $m/2-(i+1)$ is also coprime to $m$ (note that here we use the fact that $m/2$ is even). The matrix representing $\beta$ on the complex uniformization is, with respect to our basis, the anti-diagonal matrix with all entries equal to $-1$. 

Let $X_m^{\pm} = \ker[1\pm\beta]^0$ be the subvarieties corresponding to the eigenspaces of $\beta$ (since $\beta$ is an involution, its only eigenvalues are $\pm 1$). Since $X_m^+$ and $X_m^-$ are proper subvarieties of $X_m$ {and we know that $X_m$ is geometrically the square of the simple abelian variety $Y_m$, each subvariety $X_m^{\pm}$ must be geometrically isogenous to $Y_m$}. The complex uniformizations for the two subvarieties are respectively generated by $x^i\pm x^{m/2-i-2}\, dx/y$, which together form a basis for the complex uniformization of $X_m$. Therefore, $X_m \sim X_m^{+}\times X_m^{-}$.

Consider the map
\[ \gamma\colon C_m \to C_m, \qquad (x,y) \mapsto \left(\frac{1}{\zeta_{m} x}, \frac{y}{x^{\frac{m}{2}}}\right), \]
and notice that $\beta\gamma \neq \gamma\beta$. Hence, the {action} of $\gamma$ on the uniformization does not respect the eigenspaces of $\beta$, giving a nontrivial map $X_m^{+}\to X_m^{-}$.
Since both varieties are simple, this map is an isogeny defined over $\Q(\zeta_m)$. Notice that two varieties $X_m^{+}$ and $X_m^{-}$ are defined over $\Q$, while the isogeny between them might not be.
This proves (5) in the statement of \Cref{th:Jm-factorization}.

\subsection{Exceptional cases and isogenies}\label{sect: exceptional cases 20 24 60}
For $d=20, 24, 60$, no abelian variety in the isogeny class of $Y_d$ can be defined over $\mathbb{Q}$.
We will show this after discussing the case $d=20$ in detail. The automorphisms of $C_{20}$ induce maps $C_{20} \to C' \to C''$ defined over $\Q$, where $C': y^2=x^{11}+x$ and $C'' : y^2 = x^5 - 5x^3 + 5x$ (see \cite[Theorem 4.3 and Section 5.1]{Heidi2}). Over $\overline{\Q}$ (or even over the connected monodromy field of $J_m$), the curve $C''$ admits maps to two non-isomorphic elliptic curves $E, E'$  induced by automorphisms of $C''$. Both elliptic curves $E, E'$ have CM by $\mathbb{Z}[\sqrt{-5}]$. Their $j$-invariants are the roots of the irreducible quadratic polynomial $x^2 - 1264000x - 681472000$ and lie in $\mathbb{Q}(\sqrt{5}) \setminus \Q$. Thus, one can define the elliptic curves $E, E'$ over $\Q(\sqrt{5})$, but not over $\Q$. Using the techniques of \cite{MR3882288, MR3904148, MR4280568}, one can further check that the unique minimal extension $L$ of $\Q$ over which $(X_{20})_L$ is a fourth power up to isogeny is the degree-16 extension $L=\Q(\End(J_{20}))$. For the calculations verifying these claims, see \texttt{X20.m} and the corresponding output file \texttt{X20.out}.
Finally, the fact that $\Q(\sqrt{-5})$ is the CM field of (any variety isogenous to) $Y_{20}$ can also be obtained from \Cref{lemma:simpleCM}: $\Q(\sqrt{-5})$ is the subfield of $\Q(\zeta_{20})$ fixed by the subgroup $\{1,3,7,9\}$ of $(\Z/20\Z)^\times$. 

The knowledge of the CM field can be used to show directly that no elliptic curve isogenous to $Y_{20}$ can be defined over $\Q$: indeed, according to CM theory, the field of definition of an elliptic curve with CM by (any order in) the field $\Q(\sqrt{-5})$ has degree over $\Q$ at least $h(\Q(\sqrt{-5}))=2$.

For $m=24$, a similar calculation with CM types shows that the CM field of the elliptic curve $Y_{24}$ is $\Q(\sqrt{-6})$, which shows that no elliptic curve isogenous to $Y_{24}$ can be defined over $\Q$, since $h(\Q(\sqrt{-6}))=2$. For $m=60$, the CM field of $Y_{60}$ is the splitting field of $x^4 + 15x^2 + 45$ %
and has class number 4, so no abelian surface isogenous to $Y_{60}$ can be defined over $\Q$. This proves (6) in the statement of \Cref{th:Jm-factorization}. 


To conclude the proof of \Cref{th:Jm-factorization}, we only need to show the final statement about isogenies. 
We first note that no two $X_d$ with odd $d$ are isogenous (because the CM fields of $X_d, X_{d'}$  with $d \neq d'$ are $\Q(\zeta_d) \neq \Q(\zeta_{d'})$), and similarly no two $Y_d$ are isogenous: again, their CM fields are pairwise distinct. This follows from the fact that the conductor of $\Q(\zeta_{4k}-\zeta_{4k}^{-1})$ is $4k$, except for $k=3$, in which case the conductor of $\Q(\zeta_{12}-\zeta_{12}^{-1})$ is $4$. Note that, by direct inspection, the conductors of the CM fields corresponding to $Y_{20}, Y_{24}$ and $Y_{60}$ are $20, 24$, and $60$, respectively.
It remains to exclude the existence of isogenies between an $X_d$ with $d$ odd and some $Y_e$. As above, it suffices to show that the CM field $\Q(\zeta_d)$ of $X_d$ is not isomorphic to the CM field of $Y_e$. Indeed, the field $\Q(\zeta_d)$ is unramified at $2$, whereas the CM field of $Y_e$ is ramified at $2$, including the exceptional cases $e \in \{20, 24, 60\}$.

Finally, the CM field of $X_4$ is $\Q(i)$, of conductor $4$, so the previous analysis shows that $X_4$ is only isogenous to $Y_{12}$ (over $\Qbar$: both $X_{4, \Qbar}$ and $Y_{12, \Qbar}$ are elliptic curves with CM by $\Q(i)$, hence are isogenous).

\subsection{The endomorphism ring of Fermat Jacobians and their simple factors}{\label{section: endomorphism field and ring of J_m}}
In this section we determine the endomorphism field $\Q(\End(J_m))$.
\Cref{prop: endomorphism field easy containment} proves that, when $m$ is odd, the endomorphism field coincides with the cyclotomic field $\Q(\zeta_m)$ and that, for $m$ even, one always has the containment $\Q(\End(J_m)) \supseteq \Q(\zeta_m)$. \Cref{prop: endomorphism field easy containment} is all we will need from this section in the rest of the paper.

For the sake of completeness, in \Cref{th: endomorphism field of Jm complete version} we also determine the endomorphism field of every Jacobian variety $J_m$ (both for odd and even $m$). Note that, in order to accomplish this, we use some results from \Cref{sec:deligne}.

\begin{proposition}{\label{prop: endomorphism field easy containment}}
    The endomorphism field $\Q(\End(J_m))$ is a finite extension of the cyclotomic field $\Q(\zeta_m)$.
    If $m$ is odd, the endomorphism field coincides with the $m$-th cyclotomic field: $\Q(\End(J_m)) = \Q(\zeta_m).$
\end{proposition}
\begin{proof}
    Since the automorphism $\alpha_m$ of \Cref{eq: definition of alpha_m} of $C_m$ is defined over $\Q(\zeta_m)$, but not over any smaller subfield, we immediately obtain that the corresponding automorphism $\alpha_m^*$ of $J_m$ is minimally defined over $\Q(\zeta_m)$, hence that
    \begin{equation}{\label{eq: inclusion1}}
        \Q(\End(J_m)) \supseteq 
        \Q(\zeta_m).
    \end{equation}
    Consider the decomposition up to isogeny 
    \[ J_m \sim \prod_{\substack{d \mid m,\\ d \neq 1,2}} X_d \]
    given in \Cref{th:Jm-factorization}.
    Assume $m$ is odd. Every component $X_d$ is simple, with CM type described by \Cref{lemma: CM type of X_m}. Note that the CM type is simple (by \Cref{prop:LangSimple}) and hence the reflex field is $\Q(\zeta_d)$. We can apply \cite[Chapter 3, Theorem 1.1.ii]{Lang} and obtain $\Q(\End(X_d)) = \Q(\zeta_d)$. When $d$ and $d'$ are two distinct positive divisors of $m$, the corresponding endomorphism fields are different and $X_{d}$ and $X_{d'}$ are not geometrically isogenous. The endomorphism algebra of $J_{m,\overline{\Q}}$ thus decomposes into the product of the endomorphism algebras of $X_d$, all of which are defined over $\Q(\zeta_m)$. This observation gives the opposite inclusion to (\ref{eq: inclusion1}) and finishes the proof.
\end{proof}

The endomorphism algebra of any factor $X_m$ of the Fermat Jacobian can be deduced from its CM-type. %

\begin{lemma}{\label{lemma: ring of endomorphism of Xd}}
Let $m \geq 3$ be an integer.
\begin{enumerate}
    \item If $m$ is odd or $m \equiv 2 \pmod{4}$ or $m=4$, the geometric endomorphism algebra of $X_m$ is $\Q(\zeta_m)$.
    \item If $m \equiv 0 \pmod 4$ and $m \not \in \{4,  20, 24, 60\}$, the geometric endomorphism algebra of $X_m$ is $\operatorname{Mat}_{2 \times 2}(\Q(\zeta_m-\zeta_m^{-1}))$. 
    \item If $m \in \{20, 24, 60\}$, the geometric endomorphism algebra of $X_m$ is $\operatorname{Mat}_{4 \times 4}(F_m)$, where $F_{20} = \Q(\sqrt{-5})$, $F_{24}=\Q(\sqrt{-6})$, and $F_{60}$ is the field generated over $\Q$ by a root of $x^4+15x^2+45$ (this field is a degree-4 normal subextension of $\Q(\zeta_{60})$). 
\end{enumerate}
\end{lemma}
\begin{proof}
    In the first case, \Cref{th:Jm-factorization} shows that $X_m$ is geometrically simple and has CM by $\Q(\zeta_m)$. This implies that the geometric endomorphism algebra is precisely $\Q(\zeta_m)$. The third case follows from the discussion in \Cref{sect: exceptional cases 20 24 60}. Finally, if $m \equiv 0 \pmod{4}$ with $m \not \in \{4, 20, 24, 60\}$, by \Cref{th:Jm-factorization} the variety $X_m$ is geometrically isogenous to $Y_m^2$, where $Y_m$ is geometrically simple. This implies that the geometric endomorphism algebra of $X_m$ is $\operatorname{Mat}_{2 \times 2}(\operatorname{End}(Y_{m, \overline{\Q}}))$, which by \Cref{th:Jm-factorization} again is $\operatorname{Mat}_{2 \times 2}(\Q(\zeta_m-\zeta_m^{-1}))$. %
\end{proof}
Note that this lemma, combined with the last statement in \Cref{th:Jm-factorization}, yields the geometric endomorphism algebra of $J_m$ for every $m$.

Our next objective is to determine the field of definition of the endomorphisms of $J_m$. This is particularly tricky when $m$ is in the set $\{12, 20, 24, 60\}$. In these cases, we reduce the computation of the endomorphism field of $J_m$ to the (usually harder) problem of computing its connected monodromy field. For $m \in \{12, 20, 24\}$, the following lemma shows $\Q(\End(J_m))=\Q(\varepsilon_{J_m})$, so that we can use \Cref{th:Kconn-even} to compute the endomorphism field. For $m=60$, one still has the inclusion $\Q(\End(J_{60})) \subseteq \Q(\varepsilon_{J_{60}})$. We use this, together with the analysis of the splitting behavior of certain primes in $\Q(\End(J_{60}))$, to precisely pin down $\Q(\End(J_{60}))$ as a subfield of $\Q(\varepsilon_{J_{60}})$: see \Cref{cor: endomorphism field for m = 20 24 60} for details.


\begin{lemma}\label{lemma: Kconn equals KEnd if all factors have dimension at most 2}
    Let $A$ be an abelian variety over a number field $K$. If $A$ is geometrically isogenous to a product of abelian varieties of dimension at most $2$, the fields $K(\End(A))$ and $K(\varepsilon_A)$ coincide.
\end{lemma}
\begin{proof}
    Replacing $K$ with $F \colonequals K(\End(A))$, it suffices to show that the Galois representations attached to $A$ over $F$ are connected. Over $F$ we have an isogeny $A \sim \prod A_i^{e_i}$ where the $A_i$ are geometrically simple, pairwise geometrically non-isogenous, and have all their endomorphisms defined over $F$. The Galois action on each factor $A_i^{e_i}$ is diagonal; this implies easily that the projection $A \to \prod_i A_i$ induces an isomorphism on $\ell$-adic monodromy groups. Hence, it suffices to prove that the $\ell$-adic monodromy group $\mathcal{G}_{\ell, B}$, where $B=\prod_i A_i$, is connected.  
    
    Note that each group $\mathcal{G}_{\ell, A_i}$ is connected: this follows from \cite[Theorem 6.10]{MR3320526} since $\dim A_i \leq 2$ for all $i$ by assumption (this also uses the fact that the (decomposable) twisted Lefschetz group, denoted by $\operatorname{DL}_F(A_i) = \operatorname{DL}_F^{\operatorname{id}}(A_i)$ in \cite{MR3320526}, is connected for any abelian variety of dimension at most $2$). 
    For an abelian variety $X$, denote by $\mathcal{G}_{\ell, X, 1}$ the subgroup of $\mathcal{G}_{\ell, X}$ given by the intersection $\mathcal{G}_{\ell, X} \cap \operatorname{Sp}_{V_\ell(X)}$. By \cite[Theorem 3.3]{MR3320526}, $\mathcal{G}_{\ell, X}$ is connected if and only if $\mathcal{G}_{\ell, X, 1}$ is. We now show that this applies for $X=B$.
    
    It is clear from the definitions that $\mathcal{G}_{\ell, B, 1}$ is a subgroup of the connected (hence irreducible) group $\prod_i \mathcal{G}_{\ell, A_i, 1}$. %
    By \cite[Corollary 4.5]{MR3494170} and the assumption $\dim A_i \leq 2$ for all $i$ we know that the dimension of the group $\mathcal{G}_{\ell, B, 1}$ is equal to $\sum_i \dim \mathcal{G}_{\ell, A_i, 1}$. Hence, $\prod_i \mathcal{G}_{\ell, A_i, 1}$ is an irreducible variety that has the same dimension as its subvariety $\mathcal{G}_{\ell, B, 1}$. It follows that $\mathcal{G}_{\ell, B, 1}=\prod_i \mathcal{G}_{\ell, A_i, 1}$, and in particular, $\mathcal{G}_{\ell, B, 1}$ is connected, as desired. %
\end{proof}

\begin{corollary}\label{cor: endomorphism field for m = 20 24 60}
    The equality $\Q(\varepsilon_{J_m})=\Q(\End J_m)$ holds for $m=12, 20, 24$. For $m \in \{12, 20, 24\}$, the field $\Q(\varepsilon_{J_m})$ is generated by a root of the polynomial $f_m(x)$, where
    \[ f_{12}(x) = x^8 - 2x^7 + 2x^6 - 2x^5 + 7x^4 - 
    10x^3 + 8x^2 - 4x + 1, \]
    \[
    \begin{aligned}
    f_{20}(x) & = & x^{16} + 8x^{14} + 4x^{13} + 27x^{12} + 14x^{11}
    + 46x^{10} + 12x^9 + 20x^8 \\
    & &  - 18x^7 - 2x^6 - 8x^5 + 2x^4 - 6x^3 - 2x^2
    + 2x + 1, \\
    \end{aligned}
    \]
    \[
    f_{24}(x) = x^{16} + 40x^{12} + 432x^8 -
    128x^4 + 256.
    \]
    For $m=60$, the field $\Q(\End J_{60})$ is the degree-4 extension of $\Q(\zeta_{60})$ generated by
    \[
    \alpha = \sqrt{3\zeta_{60}^{12} + 2\zeta_{60}^8 - 2\zeta_{60}^6 - \zeta_{60}^4 - 2\zeta_{60}^2 - 1}, \quad
    \beta = \sqrt{4\zeta_{60}^{13} - 2\zeta_{60}^3}
    \]
    and $\Q(\varepsilon_{J_{60}})$ is the quadratic extension of $\Q(\End J_{60})$ generated by
    \[
    \gamma = \sqrt{2\zeta_{60}^{14} - 2\zeta_{60}^9 - \zeta_{60}^6 - 2\zeta_{60}^4 + 2\zeta_{60}^3 + 2}.
    \]
\end{corollary}
\begin{proof}
    For $m \in \{12, 20, 24\}$, the equality $\Q(\varepsilon_{J_m})=\Q(\End J_m)$ follows from \Cref{lemma: Kconn equals KEnd if all factors have dimension at most 2}, whose hypothesis is satisfied by \Cref{th:Jm-factorization}. We can then compute $\Q(\varepsilon_{J_m})$ using \Cref{th:Kconn-even} (which we will prove independently of this lemma) and the implementation of our algorithms in Magma \cite{OurScripts}, which gives the stated result.

    For $m=60$, the hypothesis of \Cref{lemma: Kconn equals KEnd if all factors have dimension at most 2} is not satisfied, because the simple factor $X_{15}$ has dimension $4$. To simplify the notation, let $A = J_{60}$.
    We first compute $\Q(\varepsilon_A)$ using \Cref{th:Kconn-even}: it is the degree-8 extension of $\Q(\zeta_{60})$ generated by $\alpha,\, \beta,\, \gamma$.
    We then determine the subextension $\Q(\End A) \subseteq \Q(\varepsilon_A)$ by exploiting the splitting data of finitely many primes of $\Q$ that split completely in $\Q(\zeta_{60})$, as follows.

    Since $\Q(\End A) \subseteq \Q(\varepsilon_{A})$, there is an action of $\Gal(\Q(\varepsilon_A)/\Q)$ on the geometric endomorphism algebra $\End^0(A_{\overline \Q})$, and the fixed field of the kernel of this action is $\Q(\End A)$. Equivalently, for every rational prime $p$ unramified in $\Q(\varepsilon_A)$, the Frobenius at $p$ acts trivially on $\End^0(A_{\overline{\Q}})$ if and only if $p$ is totally split in $\Q(\End A)$.
    
    Let $p$ be a prime of good reduction of $A/\Q$ (in particular, a prime unramified in $\Q(\varepsilon_A)$) and let $A_p/\F_p$ be the reduction of $A$. Reduction modulo $p$ induces an injective homomorphism between the geometric endomorphism algebra of $A$ and the geometric endomorphism algebra of $A_p$. Assume that the reduction map on geometric endomorphism algebras is also surjective (hence an isomorphism), or equivalently that
    \begin{equation}
    \label{eq: fat prime condition}
         \dim_\Q \End^0(A_p \times_{\F_p}\overline{\F_p})=
         \dim_\Q \End^0(A_{\overline \Q}).
    \end{equation}
    The reduction map is equivariant with respect to the action of a decomposition group at $p$, so -- if \eqref{eq: fat prime condition} holds -- $\operatorname{Frob}_p \in \operatorname{Gal}(\overline{\Q}/\Q)$ acts trivially on $\End^0(A_{\overline \Q})$ if and only if $\operatorname{Frob}_p \in \operatorname{Gal}(\overline{\F}_p/\F_p)$ acts trivially on $\End^0(A_p \times_{\F_p}\overline{\F_p})$.
    
    By the argument above, a prime $p$ satisfying \eqref{eq: fat prime condition} is totally split in $\Q(\End A)/\Q$ if and only if the action of the corresponding Frobenius element on the geometric endomorphism algebra of $A_p$ is trivial, hence if and only if
    \begin{equation}
    \label{eq: fat prime splitting condition}
        \dim_\Q \End^0(A_p) =
        \dim_\Q \End^0(A_p \times_{\F_p}\overline{\F_p}).
    \end{equation}

    Using \cite[Lemma 7.2.7]{MR3904148}, we can test directly whether the equalities \Cref{eq: fat prime condition} and \Cref{eq: fat prime splitting condition} hold for a given rational prime $p$; note that by \Cref{th:Jm-factorization} and \Cref{lemma: ring of endomorphism of Xd} we know that the endomorphism algebra of $A_{\Qbar} \sim Y_{60}^4 \times X_{15}^2 \times Y_{20}^4 \times X_4^3 \times X_5^2 \times X_3^2$ is 
    \[
    \begin{aligned}
    \operatorname{Mat}_{4 \times 4}(F_{60}) & \times \operatorname{Mat}_{2 \times 2}(\Q(\zeta_{15})) \times \operatorname{Mat}_{4 \times 4}(F_{20})  \\ & \times \operatorname{Mat}_{3 \times 3}(\Q(i)) \times \operatorname{Mat}_{2 \times 2}(\Q(\zeta_5)) \times \operatorname{Mat}_{2 \times 2}(\Q(\zeta_3)),
    \end{aligned}
    \]
    of $\Q$-dimension $\dim_\Q \End^0(A_{\overline \Q}) = 4^2 \cdot 4 + 2^2 \cdot 8 + 4^2 \cdot 2 + 3^2 \cdot 2 + 2^2 \cdot 4 + 2^2 \cdot 2 = 170$. Since we have $\Q(\zeta_{60}) \subseteq \Q(\End A)$ by \Cref{prop: endomorphism field easy containment}, we may focus on primes that are totally split in $\Q(\zeta_{60})$.
    Testing all primes $p \leq 2000$ such that $p\equiv 1 \bmod 60$, we get the following information: 61 and 181 are not totally split in $\Q(\End A)/\Q$, while 1741 is.
    Enumerating the $8$ intermediate extensions between $\Q(\zeta_{60})$ and $\Q(\zeta_{60}, \alpha, \beta, \gamma)$, we find that the only one that satisfies these three splitting conditions is $\Q(\zeta_{60}, \alpha, \beta)$.
\end{proof}

\begin{remark}
    For $m=12, 20, 24$, the method of \cite{MR3904148} is able to handle the calculation of $\Q(\End J_m)$. We have checked that the results of the two methods agree. On the other hand, the algorithm in \cite{MR3904148} timed out on our computers for $m=60$. Even with our method, it is not entirely straightforward to compute $\Q(\varepsilon_{J_{60}})$, because any defining polynomial will be of degree $128$ (and in practice, quite unwieldy). However, since our algorithms give a set of generators for $\Q(\varepsilon_{J_{60}})$ over $\Q(\zeta_{60})$, we were able to write down the minimal polynomial of each generator over $\Q(\zeta_{60})$ and obtain the description given in \Cref{cor: endomorphism field for m = 20 24 60}.
\end{remark}

We summarize the results of this section in the following theorem, which characterizes the endomorphism field of any Fermat Jacobian $J_m$.

\begin{theorem}{\label{th: endomorphism field of Jm complete version}}
    Let $m > 2$ be a positive integer and let $J_m/\Q$ be the Jacobian variety of the hyperelliptic curve defined by the affine equation $y^2=x^m+1$.
    If $m$ is not divisible by 12, 20, 24, or 60, the endomorphism field of $J_m$ is the $m$-th cyclotomic field $\Q(\zeta_m)$.
    If $m$ is divisible by 12, 20, 24, or 60, its endomorphism field is the compositum 
    \[ \Q(\End(J_m)) = \Q(\zeta_m) \cdot \prod_{\substack{d \in \{12,\, 20, \, 24, \, 60\}, \\ d \mid m}} \Q(\End(J_d)), \]
    where we denote by $\prod$ the compositum of fields.
    The endomorphism fields on the right-hand side of the equation are described in \Cref{cor: endomorphism field for m = 20 24 60}.
\end{theorem}

\begin{proof}
Recall that, if an abelian variety $A/\Q$ decomposes over $\overline{\Q}$ as the product (up to isogeny) of two factors $B$ and $C$ with $\operatorname{Hom}(B_{\overline{\Q}}, C_{\overline{\Q}})=(0)$, the geometric endomorphism algebra of $A$ is the direct product of the geometric endomorphism algebras of $B$ and $C$, and the endomorphism field of $A$ is the compositum of the endomorphism fields of $B$ and $C$.
Using this, the theorem follows from the decomposition of $J_m$ into simple factors (\Cref{th:Jm-factorization}), the computation of the geometric endomorphism algebra of the various abelian varieties $X_d$ (\Cref{lemma: ring of endomorphism of Xd}), and the computation of the endomorphism field in the exceptional cases (\Cref{cor: endomorphism field for m = 20 24 60}).
\end{proof} %
\section{Turning Mumford-Tate equations into Tate classes} {\label{section:MT}}
As in the previous sections, let $J_m/\Q$ be the Jacobian of the hyperelliptic curve defined by the affine equation $y^2=x^m+1$.
The present section aims at identifying the `exceptional algebraic cycles' on $J_m$ responsible for the (possible) non-triviality of the extension $\Q(\varepsilon_{J_m})/\Q(\End(J_m))$.

Since it is quite challenging to produce algebraic cycles, guided by the Tate conjecture we focus instead on identifying Tate classes, specifically in the $\ell$-adic étale cohomology of the self-products $(J_m)^k$ for $k \geq 1$. These are classes that are invariant under the action of any open subgroup of the absolute Galois group $\Gamma_\Q$ or, equivalently, under the action of the connected component of the identity of the $\ell$-adic monodromy group.
{This is relevant to the determination of $\Q(\varepsilon_{J_m})$ since we know that this field coincides with the minimal field of definition of all Tate classes on all powers of $J_m$, see \cite[Proof of Corollary 2.4.4]{cantoralfarfan2023monodromy}.}

Our first step is to compute equations for $\MT(J_m)$, in a suitable sense. We describe the strategy in \Cref{section:MT-1} and apply the results to $J_m$ in \Cref{subsec:MTcomp}.
In \Cref{subsec:MTEintoTC}, we associate with every equation $f$ for $\MT(J_m)$ a Tate class $v_f$ in étale cohomology.
The main result of the section is \Cref{prop:fixMT}, which reduces the problem of computing the connected monodromy field $\Q(\varepsilon_{J_m})$ to the task of finding the field of definition of the Tate classes $v_f$ we just defined.

\subsection{Computing equations for the Mumford-Tate group}{\label{section:MT-1}}
We describe an algorithm that allows us to explicitly compute $\MT(A)$ for abelian varieties $A$ with complex multiplication. This algorithm was introduced in \cite{MR4557876}; we reproduce the main result.

\medskip

Let $A$ be an abelian variety over a number field $K$ that decomposes (up to isogeny over $\overline{K}$) as a product of (not necessarily distinct) geometrically simple factors $A_{\overline{K}} \sim A_1 \times \dots \times A_r$. Suppose that each simple factor $A_i$ has complex multiplication by the CM field $E_i$, with CM type $\Phi_i$. Since each $A_i$ is geometrically simple, each %
CM-type $\Phi_i$ is primitive (see \Cref{prop:LangSimple}). %
Recall the discussion of \Cref{sub:CM}. %

Fix one of the fields $E_i$ %
and let $T_{E_i} = \Res_{E_i/\Q}(\mathbb{G}_{m,E_i})$ be the corresponding algebraic torus over the rationals. By definition, this is the group scheme that sends every $\Q$-algebra $A$ to the group $\mathbb{G}_{m,{E_i}}(A\otimes_{\Q} E_i) $.
We can also define $T_{E_i}$ as the unique $\Q$-torus with character group $\Z^{\Hom(E_i, \C)}$ with its natural $\Gamma_\Q$-action. Consider the map
${ \phi_i} \colon T_{E_i^*} \to T_{E_i}$
that induces on points the reflex-norm map
$N_{\Phi_i^\ast} \colon T_{E_i^*}(\Q)=(E_i^*)^\times \to E_i^\times = T_{E_i}(\Q)$,
{where $\Phi_i^\ast$ is the dual CM type of $\Phi_i.$}
On character groups, ${\phi_i}$ corresponds to the map
\[ \phi_i^\ast \colon \widehat{T}_{E_i} \to \widehat{T}_{E_i^\ast}, \qquad [\varphi]\mapsto \sum_{\sigma \in \Phi_i^\ast} [\sigma \varphi]. \]

Notice {that} this map is $\Gamma_\Q$-equivariant and, hence {$\phi_i$} is defined over $\Q$.
In the same way, we lift the norm map for any Galois extension $E/E_i$ to a map, defined over $\Q$, between the corresponding tori:
\[ {N_i^\ast} \colon \widehat{T}_{E_i} \mapsto \widehat{T}_{E}, \quad [\varphi] \mapsto \sum_{\sigma \in \Gal(E/E_i)}[\sigma \varphi]. \]
Note that, via the inclusion $\iota\colon \prod_i E_i \hookrightarrow \End^0(A_{\overline{K}})$, we get an inclusion $\prod_i T_{E_i} \hookrightarrow \GL_V$, where, as in \Cref{sub:MTdef}, $V$ is $\HH^1(A(\C), \Q)$.

\begin{proposition}[{\cite[Lemma 4.1]{MR4557876}}{\label{prop:computeMT}}]
    Let $A = \prod_i^r A_i$ be a product of geometrically simple abelian varieties, where $A_i$ has CM type $(E_i,\Phi_i)$. Assume that all the extensions $E_i/\Q$ are Galois and denote by $L$ the compositum of all these extensions. {The norm maps $L \to E_i$ induce algebraic maps $T_L\to T_{E_i}$. We denote by $N: T_L \to \prod_i T_{E_i}$ their product and by ${\phi_i}$ the reflex norm corresponding to $\Phi_i^*$.} The image of the map
    \begin{equation}\label{eq: Mumford-Tate as the image of the CM torus}
    \begin{tikzcd}
        T_L \rar["N"] &
        \prod_i T_{E_i} \rar["\phi_i"] &
        \prod_i T_{E_i} \subseteq \GL_V
    \end{tikzcd}
    \end{equation}
    is $\MT(A)$, the Mumford-Tate group of $A$.
\end{proposition}

\subsection{Computing equations for \texorpdfstring{$\MT(J_m)$}{MT(Jm)}}{\label{subsec:MTcomp}}
Let $m=2g+1>1$ be an odd positive integer and $J_m/\Q$ be the Jacobian of the smooth projective curve with affine equation $y^2 = x^m + 1$. As a consequence of \Cref{th:Jm-factorization} and \Cref{prop:computeMT}, we can explicitly compute the Mumford-Tate group of the Fermat Jacobian $J_m$.

\medskip

Let $V = \HH^1(J_m(\C),\Q)$ and $K = \Q(\End(J_m)) = \Q(\zeta_m)$ (as stated in \Cref{prop: endomorphism field easy containment}). Since $J_m$ is the Jacobian of a hyperelliptic curve, we have a privileged basis for $V_\C \cong \HH_{\operatorname{dR}}^1(J_m(\C), \C)$, namely, that given by the 1-forms $ \omega_i = x^{i-1}\,{dx}/{y}$ for $ i = 1,\, \dots,\, 2g$ (see \Cref{sub:bases}). 

We describe a slight modification of this basis that makes sense in Betti cohomology with coefficients in a number field.
{Recall the automorphism $\alpha_m$ from \cref{eq: definition of alpha_m}.}
The base-change to $K$ of the {dual of the} rational representation, that is 
\begin{equation}\label{eq: definition of VK}
    V_K \colonequals V \otimes_\Q K = \HH^1(J_m(\C), K),
    \end{equation}
decomposes into one-dimensional irreducible factors for the action of {$\alpha_m^*$.}
We fix a basis $\{v_i\}$ for $V_K$ by picking each $v_i$ in the eigenspace {$\alpha_m^*$} corresponding to the character $\chi_i \colon \zeta_m \mapsto \zeta_m^i$.
Since the one-dimensional complex vector space $\operatorname{Span}_\C(v_i)$ contains the 1-form {$x^{i-1}\, dx/y$}, there exists a complex number $\delta_i$ such that $v_i = \delta_i  \cdot x^{i-1}\, dx/y$. 

Fix a prime number $\ell$ such that $\mu_m \subseteq \Q_\ell$. For such a prime $\ell$, we can fix an embedding of $K$ into $\Q_\ell$. Through the singular-to-étale isomorphism
\[ \HH^1(J_m(\C), K) \otimes_K \Q_\ell \simeq \HH^1_{\text{ét}}(J_{m,\overline{\Q}}, \Q_\ell),  \]
the basis $\{v_i\}$ just defined gives a basis for the étale cohomology group on the right.

Having fixed a basis $v_1,\, \dots,\, v_{2g}$ for $V_K$, there is a natural choice of coordinates of the corresponding algebraic group $\GL_{V_K}=\GL_{V, K}$.
Let $x_1, \dots, x_{2g}$ be the coordinates corresponding to the diagonal elements in $\GL_{V,K}$ in the given basis.
We define the diagonal torus $ T \subseteq \GL_{V,K} $ as the subgroup
$$T = \Spec K[x_1^{\pm 1}, \dots, x_{2g}^{\pm 1}] \subseteq \GL_{V,K}.$$
The character group of $ T $ can be naturally identified with the free abelian group generated by these variables, i.e.,
\begin{equation}\label{eq: character group of T}
    \widehat{T}=\bigoplus_{i=1}^{2g}\Z x_i.
\end{equation}
Every character $\sum_i e_i x_i \in \bigoplus_{i=1}^{2g}\Z x_i = \widehat{T}$ gives an algebraic function $$ \textstyle \prod_i x_i^{e_i} \in K[x_1^{\pm 1}, \dots, x_{2g}^{\pm 1}] = \mathcal{O}_T.$$

\begin{lemma}{\label{lemma:diagonal}}
    The base-change of the Mumford-Tate group $\MT(J_m)_K \subseteq \GL_{V,K}$ lies in the diagonal torus $T \subseteq \GL_{V,K}$.
\end{lemma}
\begin{proof}
    Any endomorphism of $J_m$ defined over $K$ lifts to a linear transformation of $V_K$. The action of the Mumford-Tate group $\MT(J_m)_K$ on the $K$-Hodge structure $\End(V_K)=V_K\otimes V_K^\vee$ fixes all endomorphisms coming from the abelian variety \cite[5.2]{moonen}.
    In particular, a matrix $A$ in $\MT(J_m)_K(\overline{K})$ commutes with $\alpha_m^\ast \in \End(V_K)$.
    Here $\alpha_m$ is the endomorphism defined in (\ref{eq: definition of alpha_m}), and the matrix corresponding to $\alpha_m^\ast$ in our basis is diagonal with different eigenvalues. This forces $A$ itself to be diagonal. Since geometric points of reduced schemes are dense, this concludes the proof.%
\end{proof}

\Cref{lemma:diagonal} shows that $\MT(J_m)_K$ is a subgroup of the diagonal torus $T$.
The subgroups of this torus are easy to describe: they are all cut out by equations of the form $\prod x_i^{e_i}=1$. We now set out to find equations of this form that describe our specific group $\MT(J_m)_K$.

\begin{definition}{\label{def: equation for MT}}
{
An element $f$ of the ring $\mathcal{O}_T = K[x_1^{\pm 1}, \dots, x_{2g}^{\pm 1}]$ is an \emph{equation for $\MT(J_m)$} if it is of the form $f=\prod_{j=1}^{2g} x_j^{d_j}$ for some integers $d_j$ and $f-1$ belongs to the ideal of $\mathcal{O}_T$ defining the inclusion $\MT(J_m)_K \subseteq T$. In other words, an equation for $\MT(J_m)$ is a monomial rational function on $\mathbb{G}_{m,K}^{2g}$ whose restriction to $\MT(J_m)$ is constantly equal to $1$. We will often write the equation $f$ as $f(x)=1$.
}
Let
    \[ q = \sum_{j \text{ s.t.~} d_j < 0} |d_j| \]
denote the sum of the absolute values of all the negative exponents of the equation $f$. We often multiply the monomial $f(x)$ by $\det^{2q}$ so that their product
    \[ f \cdot {\det}^{2q} = x_1^{e_1} \cdots x_{2g}^{e_{2g}} \]
has non-negative exponents $e_j$ only. Finally, we denote by $n = \sum_{j=1}^{2g} e_j$ the degree of the product.
In what follows, a \emph{set of equations $f_1,\ldots,f_r$ defining $\MT(J_m)$} will mean a finite set of equations for $\MT(J_m)$ such that $f_1-1, \ldots, f_r-1$ generate the ideal corresponding to the inclusion $\MT(J_m)_K \subseteq T$.
\end{definition}

\begin{corollary}{\label{cor:MT}}
    Let $m>1$ be an odd positive integer. We can explicitly compute a finite set of equations $f_1,\, f_2, \, \dots, f_r$ defining $\MT(J_m)$.
\end{corollary}
\begin{proof}
    We apply the discussion in \Cref{section:MT-1} to $A=J_m$, for odd $m$. We proved in \Cref{th:Jm-factorization} that $J_m$ is geometrically isogenous to the product $\prod_{d\mid m} X_d$, where $X_d$ is an abelian variety over $\Q$. Each factor $X_d$ is geometrically simple and has complex multiplication over the CM-field $E_d=\Q(\zeta_d)$, as shown in \Cref{lemma: CM type of X_m}. Notice that all extensions $\Q(\zeta_d)/\Q$ are Galois and their compositum is $K = \Q(\zeta_m)$. Letting {$\phi_d$} be the reflex norm relative to the CM type of $X_d$, \Cref{prop:computeMT} states that the image of the composition
    \begin{equation}\label{eq: Mumford-Tate as the image of the CM torus, applied version}
    \begin{tikzcd}
        T_K \rar["N"] &
        \prod_d T_{E_d} \rar["{\phi_{d}}"] &
        \prod_d T_{E_d} \subseteq \GL_{V}
    \end{tikzcd}
    \end{equation}
    is the Mumford-Tate group of $A$.
    The base change to $K$ of \eqref{eq: Mumford-Tate as the image of the CM torus, applied version} yields a map between algebraic tori over $K$. The image of the base-changed map corresponds to $\MT(J_m)_K$ and is contained in the diagonal torus $T$, because of \Cref{lemma:diagonal}.
    
    Passing to character groups, we get a map
    \begin{equation}{\label{eq: MT as the kernel of a map}}
    \begin{tikzcd}
    \mathfrak{E} \colon \bigoplus_{i=1}^{2g}\Z x_i \rar &
    \prod_d \widehat{T}_{E_d} \rar["\phi_d^\ast"] &
    \prod_d\widehat{T}_{E_d} \rar["N^\ast"] &
    \widehat{T}_{K},
    \end{tikzcd}
    \end{equation}
    the kernel of which is the character group of the torus corresponding to (base change of) the Mumford-Tate group.
    Using the explicit description we gave of the norm map and the reflex-norm map, as better described in \cite[Section 4]{MR4557876}, we can express both maps in the above diagram as matrices with integral coefficients (since they are maps between free abelian groups).
    We can thus compute a basis for the finitely generated free abelian group $\ker\mathfrak{E}$.
    Each character of this basis corresponds to an algebraic function on $T$ of the form $x_1^{e_1} \cdots x_{2g}^{e_{2g}} =1 $, for some integers $e_i$, thus proving that each function is an equation for $\MT(J_m)$ as per \Cref{def: equation for MT}.
    These functions generate the ideal of $K[x_1^{\pm 1}, \dots, x_{2g}^{\pm 1}]$ which defines the Mumford-Tate group as a closed subvariety of the diagonal torus $T$ \cite[Proposition 1.1.9]{tori}. This proves that the given set of equations defines $\MT(J_m)$, as in \Cref{def: equation for MT}. 
\end{proof}

\begin{remark}{\label{remark:bases}}
    Notice that the (possibly transcendental) change of basis between our basis $\{v_i\}$ and the more standard $\{x^i\,dx/y\}$ is diagonal, and preserves diagonal matrices coefficient-by-coefficient. Hence, in every practical computation, we can use the latter.
\end{remark}

The use of \Cref{cor:MT} is better explained with a practical example.

\begin{example}{\label{ex:m15-equations}}
    Let $m = 15$ and consider the Fermat Jacobian $J_{15}$. From \Cref{th:Jm-factorization}, we know {that} $J_{15}$ decomposes up to isogeny as the product {of} three simple factors
    \[ J_{15} \simeq J_3 \times J_5 \times X_{15}. \]
    The elliptic curve $J_3\colon y^2 = x^3+1$ has complex multiplication by $\Q(\zeta_3)$ with CM-type $\{ 1 \} \subseteq (\Z/3\Z)^\times$. The Jacobian $J_5$ of the genus 2 curve $y^2=x^5+1$ has complex multiplication by $\Q(\zeta_5)$ with CM-type $\{1,2\}\subseteq (\Z/5\Z)^\times$. The largest simple factor $X_{15}$ has dimension 4 and has complex multiplication by $\Q(\zeta_{15})$ with CM-type $\{ 1,2,4 \}\subseteq (\Z/15\Z)^\times$.
    The compositum of the various CM fields is $E = \Q(\zeta_{15})$. The character group $\widehat{T}_E$ has rank 8, while $\prod_i \widehat{T}_{E_i}$ has rank 14.
    The matrix corresponding to the map $\mathfrak{E}\colon \Z[x_1, \dots, x_{2g}] \to \widehat{T}_E$ has been computed with this method in \cite[Section 5.2.1]{Heidi}. Computing its kernel gives us the equations
        \begin{align*}
            x_1x_{14}/x_7x_8=1, \\
            x_2x_{13}/x_7x_8=1, \\
            x_3x_{12}/x_7x_8=1, \\
            x_4x_{11}/x_7x_8=1, \\
            x_5x_{10}/x_7x_8=1, \\
            x_6x_{9}/x_7x_8=1, \\
            x_{9}x_{12}/x_{8}x_{13}=1, \\
            x_{11}x_{12}/x_{9}x_{14}=1, \\
            x_{10}x_{12}/x_{8}x_{14}=1.
        \end{align*}
        The first six express the condition that an element of the Mumford-Tate group must be symplectic. Indeed, the Mumford-Tate group preserves the symplectic form on $V$ coming from the polarization on $J_m$ and is thus contained in the general symplectic group $\operatorname{GSp}_V$ (see \cite[5.2]{moonen} and \cite[Remark 2.4.3]{part2} for the special case of $J_m$). 
        
        In the next section, we define Tate classes starting from these equations. We expect $\Q(\varepsilon_{J_m})/\Q(\End(J_m))$ to be non-trivial in the presence of `exceptional algebraic cycles', i.e.~as a consequence of the Tate classes arising from the last three equations (see \Cref{rmrk:expected}).

\end{example}

\begin{lemma}\label{lemma:degf0}
    Let $x_1^{d_1} \cdots x_{2g}^{d_{2g}} = 1$ be an equation for the Mumford-Tate group $\MT(J_m)$. The total degree $\sum_{i=1}^{2g} d_i $ is $ 0$.
\end{lemma}
\begin{proof}
    We prove that the sum of positive exponents $r=\sum_{d_i \text{ s.t.~}d_i>0}d_i$ is equal to the sum of the absolute values of the negative exponents $q=\sum_{d_j \text{ s.t.~}d_j<0}|d_j|$.  
    The Mumford-Tate group contains the sub-group of homotheties $\mathbb{G}_{m} \subseteq \GL_V$ \cite[Proposition 3.6]{Deligne}. In particular, every equation for $\MT(A)$ must be satisfied by every scalar matrix $x \in \mathbb{G}_m(\Q) \simeq \Q^\times$, thus imposing $f(x) = x^{r-q} = 1$  for every $x \in \Q^\times$. This is only possible if $r=q$, that is, $\sum_{i=1}^{2g} d_i=0$.
\end{proof}

Not much in this section changes when $m$ is an even positive integer greater than 2. The only significant difference lies in the very first step: picking a basis for $\HH^1(J_m(\C), K)$.
    Notice that the usual basis of holomorphic 1-forms is $x^i\,dx/y$ for $i=0, \dots, g-1$, while the anti-holomorphic 1-forms are generated by $x^{m-1-i}\,dx/y$ for $i=0, \dots, g-1$, thus skipping $x^{m/2-1}\, dx/y$ entirely. Correspondingly, %
    we also skip the character $\chi_{m/2}\colon \zeta_m\mapsto \zeta_m^{m/2}$. Thus, our basis indices range from $1$ to $g$ and then from $g+2$ to $m-1=2g+1$, skipping $i=m/2$. This %
    numbering is compatible with the fact that $v_i$ is an eigenvector with character $\chi_i$. Everything else then follows by the very same arguments, replacing every indexing $i = 1, \dots, 2g$ with $i =1,\dots, m/2-1, m/2+1,\dots, m-1$.

\begin{example}{\label{ex:m10-equations}}
    Following the same procedure as in \Cref{ex:m15-equations}, we can compute a set of equations defining $\MT(J_{10})$ in the diagonal torus of $\GL_{8, \Q(\zeta_{10})}$ (notice that, since $m$ is even, we skip $x_5$ to keep the notation consistent):
    \begin{align*}
        x_{1}/x_{4} &= 1, \\
        x_{2}/x_{3} &= 1, \\
        x_{1}x_{9}/x_{4}x_{6}  &= 1, \\
        x_{2}x_{8}/x_{4}x_{6}  &= 1, \\
        x_{3}x_{7}/x_{4}x_{6}  &= 1.
    \end{align*}
    Here the first two relations are justified by the isogeny $J_{10} \sim J_5^2$ defined over $\Q$ discussed in \cref{eq:beta-pull-back}. The other three equations come from the existence of the polarization.
\end{example}

\subsection{Turning Mumford-Tate equations into Tate classes}{\label{subsec:MTEintoTC}}
{Let $f_1,\, \ldots,\, f_r$ be equations defining $\MT(J_m)$, which we recall means that the functions $f_i-1$ cut out $\MT(J_m)_K$ inside the diagonal torus $T$ (see \Cref{def: equation for MT}). By \Cref{cor:MT}, such equations are explicitly computable and can be assumed to be of the form $f_i=\prod_j x_j^{d_{ij}}$ for suitable exponents $d_{ij} \in \Z$.}

{Let $f=f_i$ be one of these equations.}
In this section, we define a Tate class $v_f$ associated with $f$ (Equation \cref{eq:def-vf}) and prove that the field of definition of all such $v_f$ over $K$ is the connected monodromy field $K(\varepsilon_{J_m})$ (\Cref{prop:fixMT}).

Write $f=\prod x_j^{d_j}$ and let $q$ be the sum of {the absolute values of} all negative exponents in $f$ {(that is, $q = \sum_{j : d_j < 0} |d_j|$)}. Multiplying $f$ by {$\det^{2q} = (x_1 \cdots x_{2g})^{2q}$}, we obtain the {polynomial} $ \det^{2q} \cdot {f} = x_1^{e_1} \cdots x_{2g}^{e_{2g}} $, which has non-negative exponents and total degree $n = \sum_i e_i=2q\cdot 2g$ (see \Cref{lemma:degf0}). Write $n = 2p$ (so that $p=2gq$), and consider the tensor element
\begin{equation}{\label{eq:def-vf}}
    v_f \colonequals v_1^{\otimes e_1} \otimes  \dots \otimes v_{2g}^{\otimes e_{2g}} \otimes (2\pi i)^{p} \; \in T^{n}V_K \otimes_\Q \Q(p),
\end{equation}
where $V_K$ is defined in \eqref{eq: definition of VK} and $\Q(p)$ is the Tate structure {of weight $-2p=-n$}, on which {$\MT(J_m)$} acts via $\det^{-q}$ (see \Cref{sub:MTdef}: notice that $\MT(J_m)$ acts via $\det^{-1}$ on $\Q(g)$, hence it acts as $\det^{-2q}$ on $\Q(g)^{\otimes 2q}=\Q(2gq)=\Q(p)$).
Consider a diagonal matrix $z \in \GL_{V_K}(\overline{K})$,
with non-zero diagonal entries $(z_1, \dots, z_{2g})$. The action of $z$ on $v_f$ is given by
\begin{equation}\label{eqn:ActionOnVf}
\begin{aligned}
    z \cdot v_f & = z(v_1)^{\otimes e_1} \otimes  \dots \otimes z(v_{2g})^{\otimes e_{2g}} \otimes \det(z)^{-2q} (2\pi i)^{p}  & \text{by definition,} \\
    &= (z_1v_1)^{\otimes e_1} \otimes  \dots \otimes (z_{2g}v_{2g})^{\otimes e_{2g}} \otimes \det(z)^{-2q} (2\pi i)^{p} &   \\
    & = (z_1^{e_1} \cdots z_{2g}^{e_{2g}}) \cdot \det(z)^{-2q} \cdot v_f   & \text{grouping the scalars,}  \\
    & = {f}(z)\cdot  v_f & \text{by our choice of the $e_i$.}
\end{aligned}
\end{equation}

Notice that $T^nV_K = T^n\HH^1(J_m(\C),K)$ identifies with a subgroup of $\HH^n(J_m^n(\C), K)$ via Künneth's theorem. Consider the image of $v_f$ in
\[\HH^{2p}(J^n_m(\C),K)(p) \otimes_K \Q_\ell \simeq \HH_{\text{ét}}^{2p}(J^n_{m, \overline{K}}, \Q_\ell(p))\]
through the singular-to-étale cohomology comparison isomorphism \cite[Chapter III, Theorem 3.12]{MilneEtaleCo}. The action of the absolute Galois group $\Gamma_{K(\varepsilon_{J_m})}$ on the étale cohomology group on the right can be derived from the Galois representation attached to the abelian variety $J_m$. Thus, it factors via $\Gl^0(\Q_\ell) \subseteq \MT(\Q_\ell) = (\MT_K)_{\Q_\ell}(\Q_\ell)$.

\begin{proposition}{\label{prop:fixMT}}
    Let $m \geq 3$ be an integer and $K$ be a number field containing $\Q(\zeta_m)$. Fix a finite set of equations ${f}_i$ defining the Mumford-Tate group $\MT(J_m)$ (see definition \Cref{def: equation for MT}). The corresponding $v_{f_i}$ are Tate classes and the {minimal {extension of $K$} over which all} $v_{f_i}$ {are simultaneously defined} is the connected monodromy field $K(\varepsilon_{J_m})$.
\end{proposition}
\begin{proof}
    The action of the Galois group $\Gamma_{K}$ on the Tate module $V_\ell(J_m)$ (via the $\ell$-adic representation $\rho$) is compatible with the action of the endomorphism algebra, which is defined over $K$. Every $\sigma \in \Gamma_K$ commutes with the endomorphism $\alpha_m$ defined in \Cref{eq: definition of alpha_m}, thus respecting the one-dimensional eigenspaces of $\alpha_m^\ast$. Since $\HH_{\text{ét}}^1(J_{m, \overline{K}}, \Q_\ell)$ is dual to $V_\ell(J_m)$ and our basis $\{v_i\}$ consists of eigenvectors of $\alpha_m^\ast$, the Galois action is diagonal with respect to our working basis.
    
    An element $\sigma \in \Gamma_{K} $ acts on $V_\ell(J_m)^\vee \cong \HH_{\text{ét}}^1(J_{m, \overline{K}}, \Q_\ell)$ as a diagonal matrix $\rho(\sigma)$. Thus, transferring \eqref{eqn:ActionOnVf} to étale cohomology via the comparison isomorphism, we have
    \[ \rho(\sigma)\cdot v_{f_i} = v_{f_i} \iff {f}_i(\rho(\sigma)) = 1. \]
    This implies that $\sigma$ fixes all $v_{f_i}$ if and only if the corresponding matrix in $\Gl{(\Q_\ell)}$ satisfies all the equations defining the Mumford-Tate group, which means $\rho(\sigma) \in \MT(J_m)_{\Q_\ell}(\Q_\ell)  = \Gl^0(\Q_\ell)$.
    The equality uses the Mumford-Tate conjecture for abelian varieties of CM-type, which was proven in \cite{Pohlmann}. Equivalently, the absolute Galois group of the {minimal field over which all} $v_{f_i}$ {are simultaneously defined} is the inverse image $\rho^{-1}(\Gl^0(\Q_\ell))$, which is $\Gamma_{\KconnJ}$ by definition. On the other hand, every $v_{f_i}$ is fixed by $ \Gl^0(\Q_\ell) \supseteq \rho(\Gamma_{\KconnJ})$. In particular, the open subgroup $\Gamma_{\KconnJ} \subseteq \Gamma_K$ acts trivially on $v_{f_i}$ which is therefore a Tate class by definition.
\end{proof}

We record in the next lemma an interesting property of the equations defining $\MT(J_m)$. The first statement of the following lemma is the counterpart of \Cref{lemma:diagonal} in étale cohomology.

\begin{lemma}\label{lemma: equations are constant on connected components}
    Let $m>2$ be an integer and $K$ be a number field containing $\Q(\zeta_m)$. Let $\Gl$ be the $\ell$-adic monodromy group attached to $J_m$ over $K$. Let $\ell$ be a prime congruent to $1$ modulo $m$ and fix an embedding of $\Q(\zeta_m)$ in $\Q_\ell$.
    \begin{enumerate}
        \item The monodromy group $\Gl$ is diagonal in the basis of $\HH^1_{\operatorname{\acute{e}t}}(J_{m, \overline{\Q}}, \Q_\ell)^\vee$ corresponding to the basis $\{v_i\}$ through \Cref{eq: H1et is dual to Tate module};
        \item Identify $\Gl$ to a subgroup of $\mathbb{G}_{m, \Q_\ell}^{2g} = \Spec \Q_\ell[x_1^{\pm 1}, \ldots, x_{2g}^{\pm 1}]$. An equation $f$ for $\MT(J_m)$ is an element of the ring $\Q_\ell[x_1^{\pm 1}, \ldots, x_{2g}^{\pm 1}]$, hence it gives a regular function on $\Gl$ and a homomorphism $\Gl \to \mathbb{G}_{m}$ that we still denote by $f$. The kernel of $f$ contains $\Gl^0$, hence $f$ is constant on each connected component of $\Gl$ (in particular, the image of $f$ is finite).
    \end{enumerate}
\end{lemma}
\begin{proof}
We already discussed part (1) in the proof of \Cref{prop:fixMT}.
Part (2) is immediate: any Laurent monomial gives a homomorphism $\mathbb{G}_{m,\Q_\ell}^{2g} \to \mathbb{G}_{m,\Q_\ell}$, hence, by restriction, a homomorphism $\Gl \to \mathbb{G}_m$. The Mumford-Tate conjecture implies that $\Gl^0 = \MT(J_m)_K \times_K \Q_\ell$, and so, by the very definition of a defining equation for $\MT(J_m)$, we have that $\Gl^0$ is contained in the kernel of $f$. This implies that $f$ factors via $\Gl/\Gl^0$ and, hence, that it is constant on connected components.
\end{proof}

\begin{example}\label{ex: using Kedlaya}
   Computing equations for the Mumford-Tate group is enough to prove that $\Q(\varepsilon_A)/\Q(\End(A))$ is a non-trivial extension for $A=J_{15}$. In fact, it {is} enough to find an equation $f$ for $\MT(A)$ and element $\sigma$ in the Galois group $\Gamma_{\Q(\zeta_{15})}$ acting non-trivially on the tensor element $v_f$; by \eqref{eqn:ActionOnVf} (and the analogue of this formula in étale cohomology), this is the same as $f(\rho_{A, \ell}(\sigma)) \neq 1$.
    
    Let $K=\Q(\zeta_{15})$. Fix a prime $p$ congruent to $1$ modulo $m=15$ and let $\mathfrak{p}$ be a prime of $K$ of residue characteristic $p$. Note that the completion $K_{\mathfrak{p}}$ is isomorphic to $\Q_p$. An algorithm due to Kedlaya \cite[Section 4]{KedlayaAlg} allows us to numerically compute the matrix $M$ of Frobenius acting on the first crystalline cohomology group $\HH^1_{\text{cris}}(A, W(\mathcal{O}_K/\mathfrak{p}))$. Note that this cohomology has coefficients in the Witt vectors of $\mathcal{O}_K/\mathfrak{p} \cong \mathbb{F}_p$, namely $\Z_p$, so we can regard $M$ as having coefficients in $\Q_p$.
    The basis used in the computation consists of the differential forms $x^i \, dx/y$ \cite[Section 2]{KedlayaAlg}, which are eigenvectors for the action of the automorphism $\alpha_m$. Katz and Messing \cite{Katz1974} proved that the characteristic polynomial of Frobenius acting on $\HH^1_{\text{cris}}(A, W(\mathcal{O}_K/\mathfrak{p}))$ (that is, the characteristic polynomial of $M$) is the same as the characteristic polynomial of Frobenius on $\ell$-adic étale cohomology for $\ell \neq p$. In particular, the eigenvalues of $M$ are the roots of the characteristic polynomial of Frobenius. Note that the \textit{geometric} Frobenius automorphism of $C_{m, \mathcal{O}_K/\mathfrak{p}}$ (or of $A_{\mathcal{O}_K/\mathfrak{p}}$), raising coordinates to the $p$-th power, commutes with $\alpha_m : (x,y) \mapsto (\zeta_m x, y)$ since $p \equiv 1 \pmod{m}$. Thus, the classes $x^i \, dx/y$ -- which are eigenvectors for $\alpha_m^*$ -- are also eigenvectors for the crystalline Frobenius, and $M$ is diagonal. Similarly, $\rho_{A, \ell}(\Frob_{\mathfrak{p}})$ is diagonal in our basis $\{v_i\}$.     
    The various comparison isomorphisms are compatible with the action of automorphisms, so the eigenvalue of the $\ell$-adic Frobenius on the eigenvector $v_i$ can be identified with the eigenvalue of the crystalline Frobenius on $x^i \, dx/y$ (up to embedding both $\Q_\ell$ and $\Q_p$ in $\C$). In other words, again up to embedding both $\Q_\ell$ and $\Q_p$ into $\C$, the matrix $M$ represents $\rho_{A, \ell}(\Frob_{\mathfrak{p}})$.

We can then test explicitly if $M$ satisfies the equations defining the Mumford-Tate group. {By the Mumford-Tate conjecture, this is equivalent $\rho_{A,\ell}(\Frob_{\mathfrak{p}})$ belonging to $\Gl^0(\Q_\ell)$, or equivalently, to $\Frob_{\mathfrak{p}}$ belonging to the absolute Galois group $\Gamma_{\KconnA}$ of $\KconnA$.}
    Notice {that} we recover the matrix $M$ only up to a certain $p$-adic precision, but any equation $f$ only takes a finite number of (algebraic) values {on} $\Gl$, since it is constant on connected components (\Cref{lemma: equations are constant on connected components} (2)). We finally observe that the condition $f(M)=1$ is independent of the embedding of $\Q_p$ into $\C$, because $1$ is sent to $1$ under any embedding.
    
    Concretely, we choose the prime $p=31$, which is totally split in $K$. {Fix any prime $\mathfrak{p}$ of $K$ above $p$. Any representative for $\Frob_{31} \in \Gamma_{\Q}$ lies in $\Gamma_{K}$ and is a representative for $\Frob_{\mathfrak{p}}$.}
    Given the matrix $M$ obtained via Kedlaya's algorithm, a quick computation yields
    \[ (x_{10}x_{12}/x_8x_{14})(M) \equiv -1 \pmod{p}, \]
    proving that $K(\varepsilon_A)/K$ is non-trivial. Notice that knowing the matrix $M$ even just modulo $p$ is enough to decide that the number $(x_{10}x_{12}/x_8x_{14})(M)$ is not 1 (modulo $p$, and therefore also as an algebraic number){, and so the fact that Kedlaya's algorithm only yields approximate results is not an obstacle}.  We have implemented this calculation for general $m$ in the scripts accompanying this paper \cite{OurScripts}. %
\end{example}

\section{Fermat varieties}{\label{section:FermatVarieties}}
{In the previous section we have defined certain Tate classes $v_f$,  one for each monomial equation $f$ defining the Mumford-Tate group, as in \Cref{cor:MT}. By \Cref{prop:fixMT}, the field of definition of these classes over $K=\Q(\zeta_m)$ is the connected monodromy field $K(\varepsilon_{J_m})$. We are thus interested in computing the Galois action on the $v_f$. We will take up this task in \Cref{sec:deligne}.}
Our strategy relies on Deligne's work \cite{Deligne}, which explicitly determines the Galois action on Tate classes of particularly nice algebraic varieties: Fermat hypersurfaces.
In the present section, we reduce our problem to one solved by Deligne{, by showing that the cohomology of the powers of $J_m$ embeds inside the cohomology of a suitable Fermat hypersurface. We begin by setting notation for the Fermat varieties and recalling known facts about their cohomology.}

Let $m \geq 1,\, n\geq 1$ be integers. We denote by $X_m^n$ the Fermat variety of dimension $n$ and degree $m$, that is, the non-singular hypersurface in $\mathbb{P}^{n+1}_\Q$ defined by the equation
\[ x_0^m + x_1^m + \dots + x_{n+1}^m = 0. \]
There is a finite abelian group $G_m^n$ acting on $X_m^n$, decomposing its cohomology groups into {one-dimensional generalized eigenspaces}. 
{We describe the recursive structure of this action (\Cref{th:ShiodaII}) and recall the decomposition of the cohomology groups that it induces (\Cref{th:ShiodaI}).}

\medskip

We then map the Tate classes $v_f$ to the cohomology of a Fermat variety $X_m^n$ (for a suitable $n$), via an injective and Galois equivariant map (\Cref{prop:ShiodaIII}). Defining this Galois equivariant map involves a subtle detail: we employ a twist of the curve $C_m$. In Section \ref{sub:twist}, we explore the implications of transitioning to this twist.

\subsection{Cohomology of Fermat varieties}{\label{section:ShiodaI}}
In this section, we describe the action of a finite abelian group $G_m^n$ on the Fermat variety $X_m^n$ and its cohomology groups. These results were developed among others by Shioda and Katsura \cite{Shioda1,Shioda2,Shioda3} in their investigations of the Hodge conjecture for Fermat varieties.

As for any hypersurface of $\mathbb{P}^{n+1}_\C$, the cohomology of $X_m^n$, in any degree $i\neq n$, can be computed from the cohomology of the projective space via the Hyperplane Lefschetz theorem and the Hard Lefschetz theorem.  In particular, the primitive part $\HH_{\prim}^i(X_m^n)$ of $\HH^i(X_m^n(\C),\C) $, defined as the kernel of the cup-product with the hyperplane section, is trivial in every degree $i$ except $i=n$. For even values of $n$, the primitive part $\HH_{\prim}^n(X_m^n)$ is a codimension-one subspace of the whole cohomology group $\HH^n(X_m^n)$, while for odd $n$ it coincides with $\HH^n(X_m^n)$ itself.

Let $\mu_m$ be the group of $m$-th roots of unity and denote by $G_m^n $ the $(n+2)$-fold product of $\mu_m$ with itself modulo the diagonal embedding. Identify the character group of $G_m^n$ with the set
\begin{equation}{\label{eq:Ghat}}
    \widehat{G}_m^n \simeq \{ \alpha = (a_0, \dots, a_{n+1}) \in (\Z/m\Z)^{n+2} \colon a_0 + \dots + a_{n+1} = 0 \},
\end{equation}
mapping a character $[\zeta_0:\dots:\zeta_{n+1}] \mapsto \zeta_0^{a_0}\cdots\zeta_{n+1}^{a_{n+1}}$ to the tuple of its exponents $\alpha = (a_0,\dots,a_{n+1})$.

The group $G_m^n$ acts component-wise on the Fermat variety $X_m^n$ {and the action, seen as an action of a $\Q$-algebraic group on a $\Q$-variety, is defined over $\Q$}:
\[ [\zeta_0:\dots:\zeta_{n+1}]\cdot [x_0:\dots:x_{n+1}] = [\zeta_0x_0:\dots : \zeta_{n+1}x_{n+1}]. \]
Through this action, all cohomology groups of $X_m^n$ gain a structure of $G_m^n$-representations. Whenever a cohomology group $\HH$ is a vector space over a field containing the $m$-th roots of unity $\mu_m$, one can define the eigenspaces 
\begin{equation}\label{eq: eigenspaces in the cohomology of Fermat varieties}
\HH_{\alpha} = \{ v \in \HH \colon g\cdot v = \alpha(g) \cdot v,  \text{ for all } g \in G_m^n \}
\end{equation}
and decompose the cohomology group into one-dimensional irreducible representations.
This {applies} to singular cohomology groups with coefficients in the $m$-th cyclotomic field $K=\Q(\zeta_m)$, i.e.~$\HH^i(X_m^n(\C), \Q)\otimes K$, and with coefficients in the complex numbers.
Also, for a prime number $\ell$ such that $m \mid \ell -1$, {the field $\Q_\ell$ contains the $m$-th roots of unity (by Hensel's lemma)}  and we can thus fix an embedding $K \subseteq \Q_\ell$. %
The argument therefore also {applies} to {the} étale cohomology groups $\HH^n_{\text{ét}}(X_{m,\overline{\Q}}^n, \Q_\ell)$.
The singular-to-étale comparison isomorphism respects the $G_m^n$-action and therefore the decomposition.
We further introduce the sets
\begin{equation}\label{eq: Amn and Bmn}
\begin{aligned}
\mathfrak{A}_m^n & = \{ \alpha \in \hat{G}_m^n \colon a_i \neq 0  \text{ for all } i \},
\\ 
\mathfrak{B}_m^n & = \{ \alpha \in \hat{G}_m^n \colon \langle t \alpha\rangle=n/2+1  \text{ for all }t \in (\Z/m\Z)^\times \},
\end{aligned}
\end{equation}
where 
\begin{equation}\label{eq: definition of angular brackets}
\langle t\alpha\rangle = \sum_i [t a_i] /m
\end{equation}
and $[a]$ is the only representative for $a$ in $\Z$ lying in the interval $[0,m-1]$.

\begin{proposition}\label{th:ShiodaI}
The primitive part of the cohomology of $X_m^n$ decomposes as
\[ \textstyle \HHpn(X_m^n(\C), \C) = \bigoplus_{\alpha \in \hat{G}_m^n} \HH(X_m^n(\C), \C)_{\alpha}, \]
where each eigenspace $\HH(X_m^n(\C), \C)_{\alpha}$ has dimension 1 if $\alpha \in \mathfrak{A}_m^n$ and 0 otherwise.
The subspace $\HH(X_m^n(\C), \C)_{\alpha}$ is of type $(p,q)$ in the Hodge {decomposition}, where $p = \langle\alpha\rangle -1 $ and $q=n-p$.
\end{proposition}
\begin{proof}
    This is \cite[Proposition 3.5.1]{Shioda2} or \cite[Propositions 7.4 and 7.5]{Deligne}.
\end{proof}

In order to recover the decomposition of the full cohomology group $\HH^n(X_m^n(\C), \C)$ we must add back the imprimitive part (if $n$ is even), which is the one-dimensional eigenspace corresponding to the character $\alpha = 0$ \cite[Remark 7.5]{Deligne}.

\begin{remark}{\label{rmrk:ShiodaI}}
    
    Since a similar decomposition into eigenspaces $\HH_{\alpha}$ also holds in étale cohomology, and the two decompositions correspond to each other through the comparison isomorphisms, %
    in what follows we will sometimes write $\HH_{\prim}^n(X_m^n)$ %
    for the subspace of $\HH^n_{\text{ét}}(X_{m, \overline{K}}^n, \Q_\ell)$ given by the direct sum of the eigenspaces $\HH^n_{\text{ét}}(X_{m, \overline{K}}^n, \Q_\ell)_{\alpha}$ for $\alpha \in \mathfrak{A}_m^n$. A more intrinsic way to define the primitive part would be to use the hard Lefschetz theorem in $\ell$-adic étale cohomology, proven by Deligne in \cite[IV]{MR601520}. 
\end{remark}

\begin{notation}
We use $\HH^n(X_m^n)$ and $\HH^n_{\operatorname{prim}}(X_m^n)$ to denote interchangeably de Rham or étale cohomology in those instances where statements hold for both cohomology theories and proofs apply verbatim to both.
\end{notation}

We now describe the inductive structure of the decomposition of \Cref{th:ShiodaI}. Let $r,\, s$ {be} positive integers {and} set $ n=r+s+2 $. Through the homomorphism
\begin{equation}
    \label{eq: Grs to Gr times Gs}
    G_m^n \to G_m^r \times G_m^s, \qquad [\zeta_0:\dots:\zeta_{n+1}] \mapsto \left( [\zeta_0:\dots:\zeta_{r+1}],[\zeta_{r+2}:\dots:\zeta_{n+1}] \right),
\end{equation}
we can lift the action of $G_m^r \times G_m^s$ on the product $X_m^r\times X_m^s$ to a $G_m^n$-action.

\begin{proposition}%
{\label{th:ShiodaII}}
    Let $r,\, s$ {be} positive integers {and} set $ n=r+s+2 $. There is an inclusion
    \begin{equation}{\label{eq:ShiodaII}}
        \HH_{\prim}^r(X_m^r) \otimes \HH_{\prim}^s(X_m^s) \hookrightarrow \HH_{\prim}^{n}(X_m^{n})(1)
    \end{equation}
    which is equivariant with respect to the $G_m^n$-action. If $\HH$ denotes étale cohomology, the inclusion is also equivariant for the $\Gamma_{\Q}$-action.
\end{proposition}

\begin{remark}{\label{rmrk:ShiodaII}} 
    The $G_m^n$-equivariance of \Cref{eq:ShiodaII} describes how the eigenspace decomposition can be transferred between the cohomology groups of $X_m^r\times X_m^s$ and $X_m^n$: $\HH_{\beta_r}\otimes \HH_{\beta_s}$ is sent to $\HH_{\beta_r \ast \beta_s}$, where $\beta_r \in \mathfrak{A}_m^r, \, \beta_s \in \mathfrak{A}_m^s$ and $\beta_r \ast \beta_s \in \mathfrak{A}^{r+s+2}$ is the concatenation of tuples \cite[(2.14)]{Shioda2}.
\end{remark}

\begin{proof}
    Everything in the statement follows from \cite[Theorem II]{Shioda2} except the equivariance for the Galois action. We start by recalling the main construction: fix a $2m$-th root of unity $\varepsilon$ such that $\varepsilon^m=-1$ and consider the commutative diagram 
    \begin{equation}{\label{pic:ShiodaII}}
    \begin{tikzcd}
        \beta^{-1}(Y) \rar[hook] \dar["\beta"] &
        Z_m^{r+1,s+1}\dar["\beta"] \arrow[dr, bend left = 30,  "\psi"] \\
        Y = X_m^r \times X_m^s \rar[hook, "j"] &
        X_m^{r+1}\times X_m^{s+1} \rar[dashed, "\varphi"] &
        X_m^n,
    \end{tikzcd}
    \end{equation}
    where $\varphi$ is the rational map defined by
    $$[x_0:\dots :x_{r+2} ], [y_0:\dots :y_{s+2} ] \to [x_0y_{s+2}:\dots: x_{r+1}y_{s+2}: \varepsilon x_{r+2}y_0 : \dots : \varepsilon x_{r+2}y_{s+1}],$$
    $j$ is the inclusion of the subvariety $Y$ where $\varphi$ is not defined {(namely, the subvariety $y_{s+2}=x_{r+2}=0$)}, and $\beta$ is the blow-up along $Y$. The diagram is described in \cite{Shioda1}, where all maps are given explicitly in affine coordinate systems. Notice, in particular, that the blow-up $\beta \colon Z_m^{r+1, s+1} \to X_m^{r+1}\times X_m^{s+1}$ is defined over $\Q$.
    The composition $\psi = \varphi \circ \beta$ is a morphism \cite[Lemma 1.2]{Shioda1}. The desired inclusion is the restriction of the push-forward $\psi_\ast$ to a certain subgroup of $\HH^n(Z_m^{r+1,s+1})$ which can be identified naturally with $\HH^{n-2}(Y)(1)$ \cite[Lemma 2.1]{Shioda1}.

We now prove that the inclusion is $\Gamma_\Q$-equivariant. Fix any automorphism $\sigma$ in the Galois group $\Gamma_\Q$. Notice that $\sigma(\varepsilon)/\varepsilon$ is an $m$-th root of unity (because both numerator and denominator are primitive $2m$-th roots of unity). Consider the commutative diagram
\begin{equation*} 
\begin{tikzcd}
   {Z_m^{r+1, s+1}} \arrow[d, "\psi"] \arrow[dr, "\sigma(\psi)"] & \\
   X_m^n \arrow[r, "g"] & X_m^n 
\end{tikzcd}
\end{equation*}
where $g$ is the map fixing the first $(r+2)$ coordinates and multiplying the $(s+2)$ last ones by $\zeta=\sigma(\varepsilon)/\varepsilon$. Notice that the bottom map $g$ is the action of an element $g=(g_0, \dots, g_{n+1})$ of $G_m^n$, where $g_i=1$ for $i = 0, \dots, r+2$ and $g_i = \zeta$ for $i = r+3, \dots, n+1$. For any character $\alpha $, the push forward $g_\ast$ restricts to an isomorphism $\HH^n(X_m^n)_{\alpha}\to \HH^n(X_m^n)_{\alpha}$, which is the multiplication by $\alpha(g) = \zeta^{a_{r+3}+\dots +a_{n+1}}$.

Notice that the image of the inclusion $(\ref{eq:ShiodaII})$ decomposes as the sum of eigenspaces corresponding to characters $\alpha$ which are the concatenation of two characters $\beta_r\in\mathfrak{A}_m^r, \,\beta_s\in\mathfrak{A}_m^s$. We prove that the automorphism we get by restricting $g_\ast$ to the eigenspace $\HH^n(X_m^n)_{\alpha}$ for any such character is the identity.

When $\alpha$ is the concatenation of two characters $\beta_r\in\mathfrak{A}_m^r$ and $ \,\beta_s\in\mathfrak{A}_m^s$, the sum of the last $(s+2)$ coefficients of $\alpha$ is the sum of the coefficients of $\beta_s$, which is divisible by $m$, and therefore $\alpha(g)=1$. 
The action of an element $\sigma \in \Gamma_\Q$ on a class $\psi_\ast\omega \in \HH^n(X_m^n)_{\alpha}$ is given by
\begin{equation*}
\begin{aligned}
    \sigma \cdot {\psi}_\ast(\omega) 
    &=  [\sigma({\psi}_\ast)] (\sigma\cdot\omega) \\
    &= [g_\ast {\psi}_\ast](\sigma\cdot \omega)\\
    &= \alpha(g) \cdot  {\psi}_\ast(\sigma \cdot \omega)  \\
    &= {\psi}_\ast(\sigma \cdot \omega),
\end{aligned}
\end{equation*}
which shows that the inclusion is $\Gamma_\Q$-equivariant.
\end{proof}

\begin{corollary}\label{cor:ShiodaII}
    Let $d\geq 1$ be a positive integer and set $n = 3d-2$. There is a $G_m^n$-equivariant and $\Gamma_{\Q}$-equivariant inclusion of the $d$-fold self product of $\HH_{{\prim}}^1(X_m^1)$ into $\HH_{{\prim}}^{n}(X_m^n)(d-1)$:
    \[
    \underbrace{\HH_{{\prim}}^1(X_m^1) \otimes \dots \otimes \HH_{{\prim}}^1(X_m^1)}_{d \text{ times}} \hookrightarrow \HH_{{\prim}}^{n}(X_m^n)(d-1).
    \]
\end{corollary}
\begin{proof}
    We argue by induction. In the base case $d=1$ there is nothing to prove. Now suppose the statement holds for $d$. In the case $d+1$, apply \Cref{th:ShiodaII} with $r=1$ and $s=3d-2$ to obtain the inclusion
    \[ \HH_{\text{prim}}^1(X_m^1) \otimes \HH_{\text{prim}}^{3d-2}(X_m^{3d-2})(d-1) \hookrightarrow \HH_{\text{prim}}^{3(d+1)-2}(X_m^{3(d+1)-2})(d), \]
    then apply the inductive hypothesis.
\end{proof}

\subsection{Pull-back of Tate classes to Fermat varieties}{\label{section: Shioda pullback}}
We show that, starting with an equation $f$ for the Mumford-Tate group $\MT(J_m)$, the corresponding Tate class $v_f$ (see \Cref{eq:def-vf}) can be pulled back to a Tate class on a Fermat hypersurface, with the pullback having the same field of definition. We start by proving the result for a suitable twist $\Tilde{J}_m$ of $J_m$ (\Cref{prop:ShiodaIII}); in \Cref{sub:twist} we then reduce the case of $J_m$ to that of $\Tilde{J}_m$.

\medskip

Let $m>2$ be a positive integer. Denote by $\Tilde{C}_m$ the (smooth projective) hyperelliptic curve defined over $\Q$ given by the completion of the curve with affine equation $y^2=-4x^m+1$. Denote by $\Tilde{J}_m/\Q$ its Jacobian. Note that $\Tilde{C}_m$ is a twist of $C_m$ (see \Cref{sub:twist}).
{Consider the $\Q$-morphism $\psi : X_m^1 \to \Tilde{C}_m$ %
induced by the map of affine curves}
\begin{equation}
{\label{eq: map X1m to CmTilde}}
\begin{array}{cccc}
\psi \colon &  u^m+v^m+1=0 & \to & y^2 = -4x^m+1 \\
& (u,v) & \mapsto & (uv, 2u^m+1).
\end{array}
\end{equation}
This map allows us to relate the cohomology of $\Tilde{C}_m$, which we are interested in because it is closely related to that of $C_m$, to the cohomology of the Fermat curve $X_m^1$, which we understand well
(see \Cref{section:ShiodaI}). {The need to consider the twist $\Tilde{C}_m$ instead of our original curve $C_m$ arises from the fact that we need a map $X_m^1 \to \Tilde{C}_m$ defined over $\Q$ (and not a larger number field) in order for our subsequent calculations in cohomology to be equivariant for the full Galois group $\Gamma_\Q$.}

The inclusion $\Tilde{C}_m\hookrightarrow\Tilde{J}_m$ given by the rational point $(0, 1)$ induces a $\Gamma_\Q$ and $G_m^1$-equivariant identification $\HH^1(\Tilde{J}_m) = \HH^1(\Tilde{C}_m)$.
We (improperly) call $\{v_i\}$ the basis of $\HH^1(\Tilde{C}_m, K)$ that we obtain by pulling back the basis $\{v_i\}$ of $\HH^1(C_m, K)$ along the twist $t:\Tilde{C}_m \to C_m$ (see \Cref{remark:same-equations}). %
For $i=1,\ldots, m-1$ (skipping $i=m/2$ when $m$ is even) we define the element
\begin{equation}{\label{eq: definition of gammai}}
    \gamma_i =(i,i,-2i) \text{ in } \mathfrak{A}_m^1.
\end{equation}
These characters are pairwise distinct.

\begin{lemma}{\label{lemma:HJm-to-HJXm}}
    The pullback $\psi^\ast$ induces an isomorphism
    \[ \psi^\ast\colon \HH^1(\Tilde{C}_{m}) \longrightarrow \bigoplus^{2g}_{i = 1} \HH^1(X_{m}^1)_{\gamma_i} \subseteq \HH^1_{\prim}(X_m^1). \]
    This isomorphism sends $v_i$ to an element $\psi^\ast(v_i)\in \HH^1_{\prim}(X_m^1)_{\gamma_i}$. If $\HH$ denotes étale cohomology, then this isomorphism is also $\Gamma_{\Q}$-equivariant.
\end{lemma}
\begin{proof}
    {The étale-to-singular comparison isomorphism \Cref{eq: H1et is dual to Tate module} is $G_m^n$-equivariant. Therefore, it is enough to prove the statement in its de Rham version.}
    Let $y^2 = P(x)$ {be} the affine equation defining $\Tilde{C}_m$. 
    Given that ${dx}/{y} = {{2}dy}/{P'(x)}$, the pullback of $x^{i-1}\, dx/y$ {through $\psi$} is
    \begin{equation}
    {\label{eq: image of diff form from CmTilde to X1m}}
    \begin{split}
        \psi^\ast(x^{i-1}\, dx/y) &= \psi^\ast\left( x^{i-1} \,\frac{{2}dy}{-4m\cdot x^{m-1}} \right)\\
        & = - u^{i-1}v^{i-m}\, du,
    \end{split}
    \end{equation}
    where the last equality holds by direct substitution. %
    The group $G_m^1$ acts on this form via the character $\gamma_i = (i,i,-2i) \in \mathfrak{A}_m^1$. Since the basis $\{v_i\}$ is sent to a linearly independent set, the map is injective, hence an isomorphism onto its image.
\end{proof}

Let $h=3n-2$ and $n = 2p$, and recall (\Cref{sub:MTdef}) that we denote by $T^n W$ the $n$-th tensor power of $W$. Composing the map of Corollary \ref{cor:ShiodaII} and the self-product $T^n\psi^\ast = \left( \psi^\ast \right)^{\otimes n}$, we get a $\Gamma_{\Q}$-equivariant inclusion%
\begin{equation}{\label{eq: theta}}
    \Theta \colon T^n\HH^1(\Tilde{J}_m)(p) = T^n\HH^1(\Tilde{C}_m)(p) \hookrightarrow T^n \HH^1_{\prim}(X_m^1)(p) \hookrightarrow \HH_{{\prim}}^{h}(X_m^h)(h/2).
\end{equation}
The twist in the last group is $p+(n-1)=3p-1=h/2$.

{As} $\Tilde{C}_m$ {is} a twist of $C_m$ (see \Cref{sub:twist}), both \Cref{th:Jm-factorization} and \Cref{prop:fixMT} apply identically to $\Tilde{J}_m$, since the Mumford-Tate group only depends on the complex geometry of the abelian variety, and $J_m, \Tilde{J}_m$ are isomorphic over $\C$.

\begin{proposition}{\label{prop:ShiodaIII}}
    Fix a prime $\ell$ such that $m \mid \ell-1$, so that we can embed $\Q(\zeta_m)$ into $\Q_\ell$. Let $f$ be an equation for $\MT(\Tilde{J}_m)$. Denote by $q$ the sum of its negative exponents, let $g$ be the genus of $\Tilde{C}_m$, and set $2p = n = 4qg$. Let
    $$v_f = v_1^{\otimes e_1} \otimes\cdots\otimes  v_{2g}^{\otimes e_{2g}} \otimes (2\pi i)^p
    \in T^{n}\HH^1_{\operatorname{\acute{e}t}}(\Tilde{J}_{m, \overline{\Q}}, \Q_\ell) \otimes_\Q \Q(p)$$ be the corresponding Tate class, as defined in \Cref{subsec:MTEintoTC}. 
    The following hold:
    \begin{enumerate}
        \item The inclusion $\Theta$, as defined in (\ref{eq: theta}), is $\Gamma_{\Q}$-equivariant.
        \item $\Theta(v_f)$ is a Tate class, belonging to the eigenspace {with character $\gamma_f$}, where
        \[\gamma_f = \gamma_1^{\ast e_1}\ast \dots \ast \gamma_{2g}^{\ast e_{2g}} \in \mathfrak{B}_m^h\]
        and $h = 3n-2.$ In the notation of \Cref{section:ShiodaI}, $\langle \gamma_f \rangle = 2q \cdot 3g $.
        \item The field of definition of $\Theta(v_f)$ is the same as that of $v_f$.
    \end{enumerate}
\end{proposition} 
\begin{proof}
    As usual, $ \det^{2q} \cdot {f} = x_1^{e_1} \cdots x_{2g}^{e_{2g}} $ has non-negative exponents only. The total degree of $\det^{2q} \cdot {f}$ is $n = \sum_i e_i$.
    \begin{enumerate}
        \item We defined the map $\Theta$ as the composition of three injective and $\Gamma_\Q$-equivariant morphisms:
        \begin{itemize}
            \item the canonical isomorphism
            \[ \HH_{\text{ét}}^1(\Tilde{J}_{m, \overline{\Q}},\Q_\ell) \simeq \HH_{\text{ét}}^1(\Tilde{C}_{m, \overline{\Q}},\Q_\ell); \]
            \item the $n$-th tensor power of
            $$ \psi^\ast\colon\HH_{\text{ét}}^1(\Tilde{C}_{m, \overline{\Q}},\Q_\ell) \to \HH_{\operatorname{\acute{e}t}}^1(X^1_{m, \overline{\Q}},\Q_\ell),$$
            which is injective and Galois equivariant by \Cref{lemma:HJm-to-HJXm};
            \item the inclusion
            \[ T^n \HH_{\text{prim}}^1(\Tilde{C}_{m, \overline{\Q}},\Q_\ell) \to \HH_{\text{prim}}^h(X_m^h,\Q_\ell)(n-1), \]
            coming from the inductive structure of \Cref{cor:ShiodaII}; here $h = 3n-2$.
        \end{itemize}

        \item That $\Theta(v_f)$ is a Tate class follows from (1): $v_f$ is a Tate class (\Cref{prop:fixMT}) and $\Theta$ is a Galois-equivariant injection. {That the character of $\Theta(v_f)$ is $\gamma_1^{\ast e_1}\ast \dots \ast \gamma_{2g}^{\ast e_{2g}}$ follows from \Cref{th:ShiodaII} and \Cref{rmrk:ShiodaII}, as well as the fact that $\psi^*(v_i)$ lies in the eigenspace with character $\gamma_i$ (\Cref{lemma:HJm-to-HJXm}).}
        {We prove that $\gamma_f$ belongs to $\mathfrak{B}_m^h$. Consider the commutative diagram} %
        \begin{equation}
        \begin{tikzcd}
            T^{n}\HH^1_{\text{ét}}(\Tilde{J}_{m, \overline{\Q}}, \Q_\ell) \otimes_\Q \Q(p) \rar\dar["\cong"] & \HH_{ \operatorname{\acute{e}t}}^h(X_m^h,\Q_\ell)(h/2)\dar["\cong"]\\
            T^{n}\HH^1(\Tilde{J}_{m}(\C), \Q)\otimes_\Q \Q_\ell(p)  \rar & \HH^h(X_m^h(\C), \Q) \otimes_\Q \Q_\ell(h/2).  \\
        \end{tikzcd}
        \end{equation}
        {The vertical isomorphism on the left sends Tate classes to {the subspace generated by} Hodge classes because the Mumford-Tate conjecture holds for the CM abelian variety $\Tilde{J}_m$ (see \Cref{rmk: MT iff Tate and Hodge classes biject}). The lower horizontal map is built from pullbacks and pushforwards and has degree 0. It therefore preserves the subspace spanned by Hodge classes. Given a character $\alpha$, the $\alpha$-eigenspace $\HH^h(X_m^h(\C), \C)_{\alpha}$ is of type $(0,0)$ in the Hodge decomposition if and only if $\alpha \in \mathfrak{B}_m^h$ (see \Cref{th:ShiodaI}). Since $v_f$ is a Tate class and every map in the diagram is $G_m^h$-equivariant, we conclude that $\gamma_f$ is a character in $\mathfrak{B}_m^h$.}
        Finally, the characters in $\mathfrak{B}_m^h$ are characterized by the property of having constant weight $h/2+1$. Unwinding the definitions of $h$ and $n$, we get
        \begin{equation*}
            \langle \gamma_f \rangle = \frac{h}{2} +1 = \frac{3n-2}{2}+1 = \frac{3\cdot 4gq-2}{2}+1 = 2q \cdot 3g.
        \end{equation*}

        \item Follows from (1); note in particular that $\Theta$ is injective.
    \end{enumerate}
\end{proof}

\subsection{Invariance under twists}{\label{sub:twist}}
In this section we show that, for odd values of $m$, the connected monodromy field $\KconnA/K$ is the same as the corresponding field $K(\varepsilon_B)$ associated with the Jacobian $B$ of any twist of the form $y^2=x^m+a$ (for any $a \in \Q^\times$, see \Cref{prop: twists}). This applies in particular to the specific twist $\Tilde{J}_m$ we used in \Cref{section: Shioda pullback} (see \Cref{prop: twist-invariance}). For even $m$, we point out what changes are required to retrieve $K(\varepsilon_A)$ from $K(\varepsilon_{B})$.

\medskip

Let $f$ be an equation for $\MT(J_m)$. As usual, let $q$ denote the sum of the absolute values of its negative exponents, so that the exponents of $\det^{2q}\cdot f = x_1^{e_1}\cdots x_{m-1}^{e_{m-1}}$ are all positive. %
Define $e(f)$ as the weighted sum of these exponents:
\begin{equation}{\label{eq: define ef}}
        e(f) \colonequals \sum_i i \cdot e_i.
\end{equation}

\begin{proposition}\label{prop: twists}
Let ${C_{m,a}} \colon y^2 = x^m+a / \Q$, for $a \in \Q^\times$, be a twist of $C_m \colon y^2 = x^m+1$. Let $t \colon (C_m)_{\overline{\Q}} \to {(C_{m,a})}_{\overline{\Q}}$ be the isomorphism
\[
t(x,y) = (a^{1/m} \cdot x, a^{1/2} \cdot y).
\]
The following hold:
\begin{enumerate}
    \item for any equation $f$ for $\MT(A)$ and any $\sigma \in \Gamma_\Q$ we have
    \begin{equation}
        {\label{eq:even-twist-gen}}
        \sigma(t_* v_f) = \left( \frac{\sigma(a^{1/m})}{a^{1/m}} \right)^{u^{-1}\cdot e(f)} t_*(\sigma v_f),
    \end{equation}
    where $u \in (\Z/m\Z)^\times$ is defined by the condition $\sigma(\zeta_m)=\zeta_m^u$, the symbol $e(f)$ denotes the weighted sum in \Cref{eq: define ef}, and $v_f$ is the Tate class of \Cref{eq:def-vf}.
    \item if $m$ is odd, we have $e(f) \equiv 0 \pmod m$, hence $t_*$ identifies the Tate classes $v_f$ on powers of $J_{m, \overline{\Q}}$ with the Tate classes $t_*(v_f)$ on powers of $J_{m, \overline{\Q}}$ in a Galois-equivariant way.
\end{enumerate}
\end{proposition}

\begin{remark}
    In part (1), note that $\sigma(a^{1/m})/a^{1/m}$ is an $m$-th root of unity, so we only need $u^{-1} \cdot e(f)$ to be well-defined modulo $m$.
\end{remark}

\begin{proof}
    The following is essentially the same argument we used to prove \Cref{th:ShiodaII}. Let $c$ be the cocycle corresponding to the twist $t$ (in the sense of \cite[Theorem X.2.2]{Silverman1}), defined by
    \begin{equation}
    c \colon \Gamma_{\Q} \to \Aut(C_{m, \overline{\Q}}), \qquad \sigma \mapsto
    \left[ (x,y) \mapsto
    \left( \frac{\sigma(a^{1/m})}{a^{1/m}} \cdot x , \, \frac{\sigma(a^{1/2})}{a^{1/2}} \cdot   y \right) \right]. 
    \end{equation}
    Fix an element $\sigma \in \Gamma_\Q$. Notice that ${\sigma(a^{1/m})}/{a^{1/m}}$ is an $m$-th root of unity and ${\sigma(a^{1/2})}/{a^{1/2}}$ is $\pm 1$. There exists an integer $r$ and a sign, both depending only on $\sigma$, such that $c(\sigma) = \pm \alpha_m^r$, where $\alpha_m$ is the automorphism of $J_m$ defined in \Cref{eq: definition of alpha_m}. Note that the hyperelliptic involution of a hyperelliptic curve $C$ (which we are denoting by $-1$) acts on $\Jac(C)$ and on $\HH^1_{\operatorname{\acute{e}t}}(\Jac(C)_{\overline{\Q}}, \Q_\ell)$ as $[-1]$. The push-forward of the action of $\sigma$ along the twist is
    \begin{equation}{\label{eq:pushforward1}}
    \begin{aligned}
        \sigma (t_\ast v_i)
        &= [\sigma(t)t^{-1}]_\ast \circ t_\ast(\sigma v_i) \\
        &= c(\sigma)_\ast \circ t_\ast(\sigma v_i) \\
        &= \pm\left(\alpha_m^r\right)_\ast t_\ast(\sigma v_i), \\
        &= \pm\left(\sigma(a^{1/m})/a^{1/m}\right)^{i\cdot u^{-1}} t_\ast(\sigma v_i).
    \end{aligned}
    \end{equation}
    The last equality follows from the fact that $\sigma v_i$ lies in the eigenspace of $(\alpha_m^r)_\ast$ with character $\zeta \mapsto \zeta^{u^{-1}i}$ (see \cite[Lemma 2.4.4]{part2}) and that $t_\ast$ preserves all $\alpha_m$-eigenspaces.
    
    Now consider an equation $f=\prod_{j=1}^{2g} x_j^{d_j}$ for $\MT(A)$ and write $\det^{2q} f = x_1^{e_1} \cdots x_{2g}^{e_{2g}}$ as in \Cref{def: equation for MT}. From \cref{eq:pushforward1}, the definition of $v_f$ (see \eqref{eq:def-vf}), and the fact that $(-1)^{\sum e_i} = 1$ by \Cref{lemma:degf0} we deduce
    \begin{equation}{\label{eq: twisted action}}
    \begin{aligned}
        \sigma(t_\ast v_f) = \left(\sigma(a^{1/m})/a^{1/m}\right)^{u^{-1}\cdot e(f)} \cdot t_\ast(\sigma v_f).
    \end{aligned}
    \end{equation}
    This proves the first part of the result.
    
    For $m$ odd, we claim that the weighted sum $e(f) = \sum i \cdot e_i$ is divisible by $m$. By \Cref{prop:ShiodaIII}, we know that $\Theta(v_f)$ is a Tate class in $\HH^n_{\operatorname{\acute{e}t}}(X_{m, \overline{\Q}}^n, \Q_\ell)(n/2)$, for some integer $n$, and the character $\gamma_f $ belongs to $\mathfrak{B}_m^n$. We can then apply \cite[Lemma 5.1]{Shioda3} to conclude that the character $\ast (i)^{\ast e_i} $ belong to the set $\mathfrak{B}_m^d$, where $n=3d+4$.
    By definition of $\mathfrak{B}_m^d$, the character $\ast (i)^{\ast e_i} $ then satisfies \cref{eq:Ghat}, that is, $\sum_i i \cdot e_i$ is divisible by $m$.
    The Galois action on $t_\ast v_f$ is thus trivially twisted:
    \begin{equation}{\label{eq: trivially twisted action}}
    \begin{aligned}
        \sigma(t_\ast v_f) = t_\ast(\sigma v_f).
    \end{aligned}
    \end{equation}
    Since this holds for any equation for $\MT(J_m)$, it holds for a set of generating equations. The final statement then follows from \Cref{prop:fixMT}.
\end{proof}

\begin{remark}{\label{rmrk:shioda-even-counter}}
    For even values of $m$, {there can be characters $\gamma = \ast \gamma_i^{\ast e_i} \in \mathfrak{B}_m^n$ such that} $\sum i\cdot e_i$ is not divisible by $m$, because \cite[Lemma 5.1]{Shioda3} does not hold in this case. An explicit example is given by $m=18$ and $\gamma = \gamma_1 \ast \gamma_1 \ast \gamma_{10} \ast \gamma_{10}$. Restricting our attention to characters coming from equations for the Mumford-Tate group does not help either: for $m=18$, one of the possible equations is $x_{2}/x_{7}=1${, which leads to a character that does not satisfy $\sum i\cdot e_i \equiv 0 \pmod{m}$.}
    
    In fact, in the even case, different twists do have different fields $\KconnA \neq K(\varepsilon_B)$ (see \Cref{th:Kconn-even}), and possibly even different fields $K(\End(A)) \neq K(\End(B))$.

    {In all cases, Equation \eqref{eq:even-twist-gen} allows us to describe the Galois action on the classes $t_*v_f$ in terms of the action on the classes $v_f$, or conversely. }
\end{remark}

{
\begin{remark}\label{rmk: non-diagonal twists}
    Notice that when $m$ is even the curve $C_m$ possesses the automorphism $(x,y) \mapsto \left(\frac{1}{x}, \frac{y}{x^{m/2}}\right)$. The existence of this extra automorphism implies that twists of $C_m$ are not necessarily of the form $y^2 = x^m+a$ for $a \in \Q^\times$. For example, the curve with equation
    \[
    y^2 = -2(x^{10} - 45x^8 + 210x^6 - 210x^4 + 45x^2 - 1)
    \]
    is a twist of $y^2=x^{10}+1$ trivialised by the field $\Q(i)$, and cannot be written in the form $y^2=x^m+a$ for any $a \in \Q^\times$. We can compute all twists using the techniques of \cite{MR3906177}, which also give a $\overline{\Q}$-isomorphism between $C_m$ and any given twist. Once the twisting isomorphism is known, one can proceed as in \Cref{prop: twists} to compute the Galois action on the Tate classes on the (powers of the) Jacobian of any twist of $C_m$ purely in terms of the geometric action of the automorphism group of $C_m$ on its cohomology.
\end{remark}
}

The curve $\Tilde{C}_m\colon y^2 = -4x^m+1$, which we used in \Cref{section: Shioda pullback}, is a twist of $C_m \colon y^2 = x^m+1$, via the isomorphism
\begin{equation}
{\label{eq: definition of t twisting isomorphism between C and C tilde}}
    t\colon{C}_{m, \overline{\Q}} \to \Tilde{C}_{m, \overline{\Q}},
    \qquad (x,y) \mapsto ((-4)^{-1/m} \cdot x,y) .
\end{equation}
We are going to show that in the odd case, $\tilde{C}_m$ is isomorphic over $\Q$ to a twist $C_{m,a}$ of $C_m$, for which we can apply \Cref{prop: twists}.
Since we are interested in relating the connected monodromy fields of the Jacobians $J_m$ and $\tilde{J}_m = \Jac \tilde{C}_m$, we start with a useful remark on their Mumford-Tate groups. 

\begin{remark}{\label{remark:same-equations}}
    Let $t$ be the twisting isomorphism $t\colon{C}_{m, \overline{\Q}} \to \Tilde{C}_{m, \overline{\Q}}$ of \Cref{eq: definition of t twisting isomorphism between C and C tilde}.
    There is a basis $\{w_i\}$ for $\HH^1(\Tilde{C}_{m, \overline{\Q}}, \Q_\ell)$ that pulls back to the original basis $v_i=t^\ast w_i$ in $\HH^1(C_{m, \overline{\Q}}, \Q_\ell)$, both sharing the property of \Cref{prop:eigenbase} (we thus improperly call both bases $\{v_i\}$ outside of this remark, with the understanding that they are chosen so that $t^\ast$ sends one to the other).
    Write $A \colonequals J_m$ and $B \colonequals \Tilde{J}_m$.
    The isomorphism $\GL_V \to \GL_{\Tilde{V}}$ induced by the twist sends the algebraic subgroup $\MT(A)$ to $\MT(B)$.
    Thus, given equations $f_1, \dots, f_r$ defining $\MT(A)$, we have that $\KconnA$ is the field of definition of the Tate classes ${v}_{f_i}$ (this is \Cref{prop:fixMT}), while $K(\varepsilon_{B})$ is the field of definition of the Tate classes $w_{f_i} = t_\ast v_{f_i}$. %
\end{remark}

\begin{corollary}{\label{prop: twist-invariance}}
    Let $m \geq 3$ be a positive integer and $K=\Q(\zeta_m)$. Let $A=J_m$ be the Jacobian of the hyperelliptic curve $C_m\colon y^2=x^m+1$ defined over $\Q$ and $B=\Tilde{J}_m$ the Jacobian of $\Tilde{C}_m\colon y^2 = -4x^m+1$, also defined over $\Q$.
    Let $t$ be the twisting isomorphism in \Cref{eq: definition of t twisting isomorphism between C and C tilde}. 
    The following hold:
    \begin{enumerate}
        \item for any equation $f$ for $\MT(A)$ and any $\sigma \in \Gamma_K$ we have
        \begin{equation}
            {\label{eq:even-twist}}
            \sigma(t_* v_f) = \left( \frac{\sigma(\mu)}{\mu} \right)^{e(f)} t_*(\sigma v_f),
        \end{equation}
        where $\mu$ is a root of the polynomial $x^m-(-4)^{m-1}$, $e(f)$ is the weighted sum defined in \Cref{eq: define ef}, and $v_f$ is the Tate class of \Cref{eq:def-vf}.
        
        \item if $m$ is odd, then $K(\varepsilon_A)=K(\varepsilon_B)$.
    \end{enumerate}

\end{corollary}
\begin{proof}
    We start by noticing that, if $C/\Q$ is a hyperelliptic curve, the hyperelliptic involution of $C$ acts as $[-1]$ on $\Jac(C)$, hence it acts as $[-1]$ on $\HH^1_{\operatorname{\acute{e}t}}(\Jac(C)_{\overline{\Q}}, \Q_\ell)$. Since the Tate classes $v_f$ live in even-degree tensor powers of $\HH^1_{\operatorname{\acute{e}t}}(\Jac(C_m)_{\overline{\Q}}, \Q_\ell)$, we see that the hyperelliptic involution acts trivially on all $v_f$ and all $t_\ast v_f$.

    The hyperelliptic curves $\Tilde{C}_m$ and $\Tilde{C}_m'\colon (-1)^{m-1}y^2= x^m+(-4)^{m-1}$ are isomorphic over $\Q$, via the isomorphism
    \[
    t_1\colon (x,y) \mapsto (-4\cdot x, 2^{m-1} \cdot y).
    \]
    This gives a Galois-equivariant identification of their étale cohomology groups.
    
    If $m$ is odd, $\Tilde{C}_m'$ is a twist of the type considered in \Cref{prop: twists}.
    Part (1) follows from \Cref{prop: twists}(1). In fact, when we restrict to $\sigma \in \Gamma_K \subseteq \Gamma_\Q$, we always get $u=1$ by definition; furthermore, if we denote $s\colon C_m \to C_{m, (-4)^{m-1}}$ the isomorphism of \Cref{prop: twists}, then the composition%
    \[
    t_1^{-1} \circ s \colon C_m \to C_{m, (-4)^{m-1}}  = \tilde{C}'_m \to \tilde{C}_m
    \]
    coincides with the isomorphism $t$ in the statement up to the hyperelliptic involution, which does not affect the conclusion.
    Part (2) follows by combining \Cref{prop: twists}(2), \Cref{prop:fixMT}, and \Cref{remark:same-equations}.

    If $m$ is even, the hyperelliptic curves $\Tilde{C}_m'$ and $\Tilde{C}_m''\colon y^2= x^m+(-4)^{m-1}$ are quadratic twists of one another (through the hyperelliptic involution). The twist $\Tilde{C}_m''$ is of the type considered in \Cref{prop: twists}: we obtain the conclusion combining \Cref{prop: twists}(1) and our remark about the action of the hyperelliptic involution.
\end{proof}

\section{Galois action on Tate classes}{\label{sec:deligne}}
In this section we describe a method to compute the connected monodromy field $\Q(\varepsilon_{A})$ of the Fermat Jacobian $A=J_m$. According to \Cref{prop:fixMT} and \Cref{prop:ShiodaIII}(3), this is equivalent to computing the field of definition of certain Tate classes on Fermat hypersurfaces.

\medskip

We first describe the Galois action on cohomology classes of $X_m^n$, in terms of the decomposition into one-dimensional representations described in \Cref{section:FermatVarieties} and special values of the Gamma function (see \Cref{th:DeligneI} and \Cref{th:DeligneII}).
Applying these results to the Tate classes $v_f$ defined in \Cref{section:MT} and pulled back to Fermat varieties in \Cref{section:FermatVarieties}, we derive a description of the connected monodromy field $\Q(\varepsilon_{J_m})$ (\Cref{th:Kconn}). We then describe how to perform the computation algorithmically (\Cref{sub:deligneIII}) and give some examples (\Cref{section:examples}).

\subsection{Galois action on the cohomology of Fermat varieties}{\label{section:DeligneI}}
Let $K=\Q(\zeta_m)$. Following the work of Deligne \cite{Deligne}, we describe explicitly the action of the absolute Galois group $\Gamma_K$ on the étale cohomology of the Fermat variety $X_m^n$.

Let $\mathfrak{p}$ be a prime ideal of $K$ not dividing $m$ and let $\mathbb{F}_q$ {be} its residue field. Reduction modulo $\mathfrak{p}$ induces an isomorphism between $\mu_m \subset \mathcal{O}_K^\times$ and the $m$-th roots of unity in $\mathbb{F}_q^\times$, whose inverse {we denote by} $t$. Let $\alpha = (a_0, \dots, a_{n+1}) \in \mathfrak{A}_m^n$ {be} a fixed character of $G_m^n$. We introduce the {character}
\[ \epsilon_i\colon \mathbb{F}_q^\times \to \mu_m, \qquad x \mapsto t \left(x^{ \frac{1-q}{m} }\right)^{a_i}. \]
Let $\psi\colon \mathbb{F}_q \to \C^\times $ be a non-trivial additive character. Define the Gauss sums
\[ g(\mathfrak{p}, a_i, \psi) = - \sum_{x \in \mathbb{F}_q} \epsilon_i(x)\psi(x), \qquad
g(\mathfrak{p}, \alpha) = q^{-\langle \alpha\rangle} \prod_{i=0}^{n+1} g(\mathfrak{p}, a_i, \psi),\]
where $\langle\alpha\rangle = \sum_i [a_i] /m$ and $[a]$ is the only representative for $a$ in $\Z$ lying in the interval $[0,m-1]$, as defined in \Cref{section:FermatVarieties}.

\begin{theorem}[Deligne {\cite[Corollary 7.11]{Deligne}}{\label{th:DeligneI}}]
    {Fix} $\alpha \in \mathfrak{B}_m^n$ and {let} $\Frob_\mathfrak{p}$ be a Frobenius element corresponding to a prime $\mathfrak{p}$ of $K$ not dividing $m$.
    For every $v \in\HH^n_{\operatorname{\acute{e}t}}(X^n_{m,\overline{K}}, \Q_\ell)(p)_{\alpha}$ we have
    \[ \Frob_\mathfrak{p}(v) = g(\mathfrak{p}, \alpha) \cdot v. \]
\end{theorem}

{Our strategy for describing the connected monodromy field of $J_m$ comes down to the fact that, for each character $\alpha$ as above, we can find an explicit algebraic number on which Galois acts in the same way as on $v \in \HH^n_{\operatorname{\acute{e}t}}(X^n_{m,\overline{K}}, \Q_\ell)(p)_{\alpha}$. These numbers are obtained as suitable (normalised) products of values of the $\Gamma$ function evaluated at rational arguments.}
Recall that Euler's Gamma function is defined as the analytic continuation of the integral
$$\Gamma(s)= \int_0^\infty e^{-t}\, t^s\, \frac{dt}{t}, \qquad  \text{for } \Re(s)>0.$$
It is a meromorphic function that is holomorphic on the whole complex plane except for the non-positive integers, where it has simple poles. The Gamma function satisfies the following well-known relations \cite{MR1638625}:
\begin{itemize}
    \item Translation: $\Gamma(s+1) = s\Gamma(s)$.
    \item Reflection: $\Gamma(1-s)\Gamma(s) = \pi/\sin(\pi s)$.
    \item Multiplication: for every positive integer $n\geq1$,
    \[ \prod_{k=0}^{n-1} \Gamma\left( s +\frac{k}{n} \right) = (2\pi)^{\frac{n-1}{2}}n^{\frac{1}{2}-ns} \Gamma(ns). \]
\end{itemize}

\begin{definition}{\label{definition:character-gamma-value}}
We define the Gamma-value of a character $\alpha=(a_0, \dots, a_{n+1}) \in \mathfrak{A}_m^n$ as the complex number given by the formula
\[ \Gamma(\alpha) = (2\pi i)^{-\langle \alpha \rangle } \prod_{i=0}^{n+1} \Gamma\left(\{\frac{a_i}{m}\}\right),  \]
where $\{q\} \in (0,1]$ denotes the fractional part of $q \in \mathbb{Q}$.
\end{definition}

\begin{theorem}[Deligne {\cite[Theorem 7.15]{Deligne}}{\label{th:DeligneII}}]
    {Fix} $\alpha \in \mathfrak{B}_m^n$ and {let} $\Frob_\mathfrak{p}$ be a geometric Frobenius element corresponding to a prime $\mathfrak{p}$ of $K$ not dividing $m$. {The complex number} $\Gamma(\alpha)$ is algebraic over the rationals and
    \[ \Frob_\mathfrak{p}(\Gamma(\alpha)) = g(\mathfrak{p}, \alpha) \cdot \Gamma(\alpha). \]
\end{theorem}

\subsection{The connected monodromy field of \texorpdfstring{$J_m$}{Jm}} Let $m>1$ be an odd positive integer and $J_m$ be the Jacobian of the smooth projective curve {over $\Q$} with affine equation $y^2=x^m+1$. For any equation $f$ for the Mumford-Tate group $\MT(J_m)$, denote by $q$ the sum of the absolute value of all negative exponents of $f$ as in \Cref{def: equation for MT}, so that $\det^{2q}\cdot f = x_1^{e_1} \cdots x_{2g}^{e_{2g}}$ has no negative exponent. Recall from \Cref{eq: definition of gammai} that we defined $\gamma_i$ to be the character $(i,i,-2i) \in \mathfrak{B}_m^1$ and 
\begin{equation}{\label{eq: definition of gamma_f}}
    \gamma_f = \gamma_1^{\ast e_1} \ast \dots \ast \gamma_{2g}^{\ast e_{2g}}.
\end{equation}
By \Cref{prop:ShiodaIII}, we know that $\gamma_f$ belongs to $\mathfrak{B}_m^n$ for a suitable integer $n$.

\begin{theorem}{\label{th:Kconn}}
    Let $m \geq 3$ be an odd positive integer and $J_m / {\Q}$ be the Jacobian of the smooth projective curve {over $\Q$} with affine equation $y^2=x^m+1$.
    Fix a set of equations ${f}_1, \, f_2, \,  \dots,\,  {f}_r$ defining the Mumford-Tate group $\MT({J}_m)$. The connected monodromy field $\Q(\varepsilon_{{J}_m})$ is generated over $\Q$ by the complex (algebraic) numbers $$\zeta_m, \, \Gamma(\gamma_{f_1}),\, \Gamma(\gamma_{f_2}),\, \dots, \, \Gamma(\gamma_{f_r}).$$
\end{theorem}

\begin{proof}
    The connected monodromy field of ${J}_m$ is an extension of the endomorphism field, which we determined in \Cref{prop: endomorphism field easy containment} to be $ \Q(\End(J_m)) = \Q(\zeta_m)$, and which we denote by $K$.
    We add an $m$-th root of unity $\zeta_m$ to our list of generators and proceed to determine the extension $\KconnJ/K$.
    
    By \Cref{prop: twist-invariance} we can instead compute the connected monodromy field of the twist $B=\Tilde{J}_m$ over $K$.
    The equations $f_i$ also define the Mumford-Tate group of this twist (see \Cref{remark:same-equations}).
    By \Cref{prop:fixMT}, the connected monodromy field $K(\varepsilon_B)$ is the field of definition of all tensor elements $v_{f_i}$, which coincides with the field of definition of the classes $c_{f_i}\colonequals \Theta(v_{f_i}) \in \HH^n_{\operatorname{\acute{e}t}}(X_{m, \overline{\Q}}^n, \Q_\ell)(n/2)_{ \gamma_{f_i}}$, for a suitable integer $n$, as guaranteed by \Cref{prop:ShiodaIII}(3).
    By \Cref{th:DeligneI}, given a prime ideal $\mathfrak{p}$ of $K$ not dividing $m$, the corresponding geometric Frobenius element acts on $c_{f_i}$ as multiplication by a Gauss sum:
    \begin{equation}\label{eq: Frobenius acting on Tate classes}
            \Frob_\mathfrak{p}(c_{f_i}) = g(\mathfrak{p}, \gamma_{f_i}) \cdot c_{f_i}.
    \end{equation}
    \Cref{prop:ShiodaIII} shows that $\gamma_{f_i} $ belongs to $ \mathfrak{B}_m^n$. Hence \Cref{th:DeligneII} applies and the Frobenius element $\Frob_\mathfrak{p}$ acts on the algebraic number $\Gamma(\gamma_{f_i})$ in exactly the same way as it acts on the corresponding Tate class:
    \begin{equation}\label{eq: action on gamma-value is the same as action on Tate class}
            \Frob_\mathfrak{p}(\Gamma(\gamma_{f_i})) = g(\mathfrak{p}, \gamma_{f_i}) \cdot \Gamma(\gamma_{f_i}).
    \end{equation}

    This holds for all Frobenius elements except the finitely many {corresponding to primes} dividing $m$. As Frobenius elements are dense in the absolute Galois group $\Gamma_K$ by Chebotarev's theorem, every element of  $\Gamma_K$ acts identically on the two objects, and therefore their fields of definition coincide.
\end{proof}

When $m$ is even, \Cref{th:Kconn} holds for the twist $y^2=-4x^m+1$. In order to get the connected monodromy field $\KconnJ$, we need to understand the field of definition of the push-forward of the classes $v_{f_i}$ through the twist, as discussed in \Cref{prop: twist-invariance}. 
\begin{theorem}{\label{th:Kconn-even}}
    Let $m \geq 4$ be an even positive integer and $J_m/\Q$ be the Jacobian of the smooth projective curve over $\Q$ with affine equation $y^2=x^m+1$.
    Fix a set of equations ${f}_1, \dots, {f}_r$ defining $\MT({J}_m)$. The connected monodromy field $\Q(\varepsilon_{{J}_m})$ is generated over $\Q$ by the complex (algebraic) numbers
    \[\zeta_m, \, \Gamma(\gamma_{f_1})\cdot \muu^{e(f_1)}, \,\dots, \,\Gamma(\gamma_{f_r})\cdot \muu^{e(f_r)},\]
    where $\muu \in \overline{\Q}$ is a root of the polynomial $x^m+4$ and $e(f)$ is defined in \Cref{eq: define ef}.
\end{theorem}
\begin{proof}
    The same argument in the proof of \Cref{th:Kconn} holds for the twist $y^2=-4x^m+1$: we know that the connected monodromy field is an extension of the endomorphism field, which contains $K = \Q(\zeta_m)$ by \Cref{prop: endomorphism field easy containment}. Let $t$ be the isomorphism in \Cref{eq: definition of t twisting isomorphism between C and C tilde}.
    Combining \eqref{eq: Frobenius acting on Tate classes} and \eqref{eq: action on gamma-value is the same as action on Tate class} we obtain that,
    given a prime ideal $\mathfrak{p}$ of $K$ not dividing $m$ and an equation $f$ for $\MT(J_m)$, the action of the geometric Frobenius element at $\mathfrak{p}$ on $t_*v_f$ is
    \begin{equation}{\label{eq:kconn-th-even}}
    \Fp (t_\ast v_f) = \frac{\Fp \Gamma(\gamma_f)}{\Gamma({\gamma_f})} \cdot t_\ast v_f.
    \end{equation}
    Notice that we also used \Cref{prop:ShiodaIII} to write $v_f$ instead of $c_f=\Theta(v_f)$.
    Let $\mu_1 \in \overline{\Q}$ be a root of $x^m-(-4)^{m-1}$. We apply \Cref{prop: twist-invariance}(1) to rewrite the left-hand side of \Cref{eq:kconn-th-even} and get
    \[ \left( \frac{\Fp(\mu_1)}{\mu_1} \right)^{e(f)} \cdot t_*(\Fp v_f) = \frac{\Fp \Gamma(\gamma_f)}{\Gamma({\gamma_f})} \cdot t_\ast v_f. \]
    Notice that $t_\ast$ is linear and injective, and therefore
    \[ \Fp (v_f) = \left( \frac{\Fp(\mu_1)}{\mu_1} \right)^{-e(f)} \cdot  \frac{\Fp \Gamma(\gamma_f)}{\Gamma({\gamma_f})} \cdot v_f = \frac{\Frob_{\mathfrak{p}}\left(\Gamma(f) \mu_1^{-e(f)}\right)}{\Gamma(f) \mu_1^{-e(f)}} \cdot v_f.
    \]
    Thus, the minimal field of definition of $v_f$ is generated over $\Q(\zeta_m)$ by the algebraic number $\Gamma(f) \mu_1^{-e(f)}$, or equivalently by $\Gamma(f) (-4\mu_1^{-1})^{e(f)}$. Since $\mu \colonequals -4 \mu_1^{-1}$ is a root of the polynomial $x^m+4$,
    we conclude that the field of definition of all the classes $v_{f_i}$ is the one in the statement.
    As discussed in \Cref{remark:same-equations}, this field coincides with $\KconnJ$.
\end{proof}

\subsection{Exact computations with the Gamma function}{\label{sub:deligneIII}} {\Cref{th:Kconn} gives a set of generators for the field $\KconnJ$. We know that these are algebraic numbers, but their definition in terms of the $\Gamma$ function means that, a priori, we can only compute them numerically. In this section, we show that these numbers can also be computed exactly (that is, we can determine their minimal polynomials) via a purely algebraic procedure (\Cref{th:exactly-gamma}). In particular, this makes the computation of $\KconnJ$ suggested by \Cref{th:Kconn} exact, and gives a fully algorithmic procedure to determine $\KconnJ$ for any given $m$.}
To do so, we need to introduce the theory of distributions, which we borrow from \cite{MR0545172}, \cite{MR1638625}, and \cite{MR1871965}. 

\medskip

Fix a positive integer $m>2$ and consider the set $A_m = \mmZ$. Let $\mathbb{A}$ be the free abelian group over $A_m$, that is, the free abelian group with basis indexed by symbols of the form $[a]$ for $a \in A_m$. Note that we can consider the elements of $\mathbb{A}$ either as (finite) formal combinations $\sum_i c_i [a_i]$ for certain $c_i \in \mathbb{Z}$ and $a_i \in A_m$, or as functions $A_m \to \mathbb{Z}$. A homomorphism $f$ defined on $\mathbb{A}$ (with values in an abelian group $B$) is a \textit{distribution} if the equality $f([a]) = \sum_{b \colon nb = a}f([b])$ holds for all positive integers $n$ dividing $m$ and all $a \in nA_m$. 
We call 
\begin{equation}\label{eq: distribution relation}
    [a]- \sum_{b \colon nb=a} [b]
\end{equation}
a \textit{distribution relation} and
$\mathbb{D} = \langle [a]- \sum_{nb=a} [b] \rangle_{a \in nA_m, n \in \mathbb{Z}^+, n \mid m}$ the submodule of distribution relations. It is clear that any distribution factors via the quotient $\mathbb{U} \colonequals \mathbb{A}/\mathbb{D}$. 
\begin{remark}\label{rmk: U is torsion-free}
    Note that our group $\mathbb{U}$ injects into the universal distribution group $U_{\Q/\Z}$ of \cite[Section 1]{MR0545172}: to see this, notice that our distribution relations are exactly the intersection of the distribution relations in $\Z[\Q/\Z]$ with $\mathbb{A} = \Z[A_m]$ (see \cite[pp.182-183]{MR0545172}). In particular, $\mathbb{U}$ is torsion-free.
\end{remark}

A distribution $f: \mathbb{A} \to B$ is said to be \textit{odd} if furthermore $f([-a]) = -f([a])$ holds for all $a \in A_m$. Every odd distribution factors via the quotient
\[
\mathbb{U}^- \colonequals \frac{\mathbb{A}}{\langle [a]- \sum_{nb=a} [b], [a]+[-a] \rangle_{a \in nA_m, n \in \mathbb{Z}^+, n \mid m} }.
\]
We note that $\mathbb{U}$ is a $(\mathbb{\Z}/m\mathbb{Z})^\times$-module, where $t \in (\mathbb{\Z}/m\mathbb{Z})^\times$ acts via the formula $t \cdot \left( \sum c_i [a_i] \right) \colonequals \sum c_i [ta_i]$. Indeed, $(\mathbb{\Z}/m\mathbb{Z})^\times$ acts naturally on $A_m$ and the action above preserves the distribution relations: for all $t \in (\mathbb{\Z}/m\mathbb{Z})^\times$ one has $t \cdot \mathbb{D} \subseteq \mathbb{D}$. To see this, it suffices to notice that for all $t \in (\mathbb{\Z}/m\mathbb{Z})^\times$, all positive integers $n$, and all $a \in A_m$ the element
\[
t \cdot \left( [a] - \sum_{nb=a} [b] \right) = [ta] - \sum_{nb=a} [tb] = [ta] - \sum_{nb=ta} [b]
\]
is again in $\mathbb{D}$. This ensures that the $(\mathbb{\Z}/m\mathbb{Z})^\times$-action descends to $\mathbb{U}$ and induces an action of the two-element group $\{\pm 1\}$ on $\mathbb{U}$.

The next statement and its proof have been adapted from \cite[Proposition 2.4]{MR1871965} to suit our specific context.

\begin{proposition}{\label{prop:torsion-cohomo-distr}}
There is a canonical isomorphism
\begin{equation*}
\mathbb{U}^-_{\operatorname{tors}} \cong \HH^2(\{\pm 1\}, \mathbb{U}),    
\end{equation*}
where $\mathbb{U}^-_{\operatorname{tors}}$ is the torsion subgroup of $\mathbb{U}^-$.
\end{proposition}
\begin{proof}
Since the group $\{\pm 1\}$ is cyclic, its cohomology is 2-periodic \cite[Theorem 2.16]{Harari}. Recalling the definition of Tate's reduced cohomology groups \cite[Definition 2.3]{Harari}, we get
\begin{equation*}
    \HH^2(\{\pm1\}, \mathbb{U}) \simeq \widehat{\HH}{}^0(\{\pm1\}, \mathbb{U}) \simeq \frac{\ker[(1-[-1])\colon \mathbb{U}\to\mathbb{U}]}{(1+[-1])\mathbb{U}}.
\end{equation*}
If $x \in \mathbb{U}$ and there exists an integer $n$ such that $nx \in (1+[-1])\mathbb{U}$ (i.e.~$x$ is torsion in $\mathbb{U}^{-}$), then $(1-[-1])nx = 0$ in $\mathbb{U}$ and therefore $(1-[-1])x=0$ since $\mathbb{U}$ is torsion free (see \Cref{rmk: U is torsion-free}). In this case, $x$ belongs to $\ker[(1-[-1])\colon \mathbb{U}\to\mathbb{U}]$, and so we have the inclusions
\[ \mathbb{U}^{-}_{\operatorname{tors}} \subseteq \HH^2(\{\pm 1\}, \mathbb{U}) \subseteq \mathbb{U}^{-}. \]
Since $\HH^2(\{\pm 1\}, \mathbb{U})$ is a torsion group, the first inclusion must be an equality. 
\end{proof}

Consider functions $f \colon A_m \to \Z$, which we think of as elements of $\mathbb{A}$. For any such function, define the weight map
\begin{equation}\label{eq: weight map}
    \langle f \rangle \colon (\Z/m\Z)^\times \to \Q, \; u \mapsto \sum_{a \in A_m} \{a\} f(ua),
\end{equation}
where $\{a\}$ is the representative of $a$ in the interval $(0,1]\cap \Q$. %

For any positive divisor $d$ of $m$ and any element $a \in A_m\setminus \{0\}$, consider the set
\begin{equation}\label{eq: support of distribution relation}
 S_{d,a}=\{ a+k/d \colon k = 0,\dots, d-1 \} \cup \{-da\}.
 \end{equation}
Its characteristic function $\varepsilon_{d,a} : A_m \to \mathbb{Z}$  has constant weight (see the Appendix by Koblitz and Ogus in \cite{KoblitzOgus}, in particular, the example on page 343).%

{Our next result, characterizing functions $f: A_m \to \Z$ with constant weight, is announced by Koblitz and Ogus in \cite[Remark on p.~345]{KoblitzOgus}, but does not seem to have been published, at least not in this form.}%

\begin{proposition}{\label{prop: a constant weight function is a linear combination of distribution relations with half-integral coefficients}}
    Suppose $f : A_m \to \mathbb{Z}$ is a function such that $\langle f \rangle$ is constant. The function $2f$ is a linear combination with integral coefficients of functions of the form $\varepsilon_{d, a}$. %
\end{proposition}
\begin{proof}
We begin by noting that the functions $\varepsilon_{d, a}${, as $a$ varies in $A_m \setminus \{0\}$,} generate the submodule $\langle \mathbb{D}, [a]+[-a]\rangle_{a \in A_m \setminus \{0\}}$. On the one hand, the function $\varepsilon_{1, a}$ is $[a]+[-a]$. 
On the other hand, 
\begin{equation}\label{eq: the epsilon functions of Koblitz-Ogus are distribution relations}
    -\varepsilon_{d, a} + \varepsilon_{1, da} = -\sum_{k=0}^{d-1} \left[a+\frac{k}{d} \right] - [-da] + ([da]+[-da]) = [da] - \sum_{k=0}^{d-1} \left[a+\frac{k}{d} \right]
\end{equation}
is exactly the distribution relation \eqref{eq: distribution relation} with parameters $(a, n)=({da}, d)$: indeed, the elements $x \in A_m$ such that $dx = da$ are precisely those for which $x-a$ is $d$-torsion in $A_m$, hence those of the form $x-{a} = \frac{k}{d}$ for $k=0,\ldots,d-1$.
As $d,\, a$ vary, we thus obtain generators for the whole submodule $\langle \mathbb{D}, [a]+[-a] \rangle_{a \in A_m\setminus \{0\}}$.

A result of Koblitz and Ogus \cite[Appendix, Proposition]{KoblitzOgus} (see also \cite[Theorem 8]{MR1638625}) shows that $f$ is a $\mathbb{Q}$-linear combination of functions of the form $\varepsilon_{d, a}$. In particular, some integral multiple $nf$ of $f$ belongs to $\langle \mathbb{D}, [a]+[-a]\rangle_{a \in A_m\setminus \{0\}} $, hence $f$ is torsion in $\mathbb{U}^{-}$. Since $\HH^2(\{\pm 1\}, \mathbb{U})$ is a 2-torsion group, \Cref{prop:torsion-cohomo-distr} now shows that $2f$ belongs to $\langle \mathbb{D}, [a]+[-a]\rangle_{a \in A_m\setminus \{0\}} = \langle \varepsilon_{d, a} \rangle_{d, a}$, as claimed.
\end{proof}

We are ready to apply the theory of distributions to our situation. We can interpret a character $\alpha = (a_0, \dots, a_{n+1}) \in \mathfrak{A}_m^n$ as a function $\alpha\colon \Z/m\Z \simeq A_m \to \Z$, sending each $a \in \Z/m\Z$ to the number of its occurrences among the coefficients $a_i$. Notice that $\alpha(0)$ is zero for any character in $\mathfrak{A}_m^n$. A character $\alpha \in \mathfrak{B}_m^n$ is such that $\langle \alpha \rangle$ is constant since the definition of weight function coincides with the one given for characters in \Cref{definition:character-gamma-value}. 

\begin{definition}{\label{def: gamma f}}
For any function $\alpha \colon A_m \to \Z$ such that $\langle \alpha \rangle$ is constant, define
$$\Gamma(\alpha) := (2\pi i)^{-\langle \alpha \rangle}\prod_{x \in A_m} \Gamma\left( {\{x\}}\right)^{\alpha(x)}. $$
\end{definition}
Notice that for a character this coincides with the previous definition (\Cref{definition:character-gamma-value}). \Cref{prop: a constant weight function is a linear combination of distribution relations with half-integral coefficients} thus applies to characters: we now show that this implies that their Gamma-values can be computed using only the translation, reflection, and multiplication formulas. Before stating the theorem, we recall that the exact computation of an algebraic number (viewed as an element of $\mathbb{C}$) amounts to determining its minimal polynomial over $\mathbb{Q}$ together with a sufficiently accurate numerical approximation to distinguish it from the other roots of that polynomial.

\begin{theorem}{\label{th:exactly-gamma}}
    There is an algorithm that computes the algebraic number $\Gamma(\alpha)$ exactly, for every $\alpha \in \mathfrak{B}_m^n$.
\end{theorem}
\begin{proof}
    Via \Cref{prop: a constant weight function is a linear combination of distribution relations with half-integral coefficients} and linear algebra, we write $2\alpha$ as a combination of some $\varepsilon_{d,a}$. Let $H$ be the $\Z$-module generated by the $\varepsilon_{d,a}$. The function
    \[
    \begin{array}{ccc}
        H & \to & \C^\times  \\
        \alpha & \mapsto & \Gamma(\alpha)
    \end{array}
    \]
    is clearly a group homomorphism, so
    we obtain an expression for the number $\Gamma(\alpha)^2$ as a product of factors of the form $\Gamma(\varepsilon_{d,a})$.

    We now show that each $\Gamma(\varepsilon_{d, a})$ can be explicitly computed using the translation, reflection, and multiplication properties of the Gamma function. Using the translation property, we can assume $\{a\}$ to be in the interval between $0$ and $1/d$. With this simplification, $\{-da\} = 1- \{da\}$ and we can therefore express $\Gamma(1-\{da\})$ in terms of $\Gamma(\{da\})^{-1}$ via the reflection property (unless $\{-da\}=1$, in which case $\Gamma(\{-da\})=1$). The new expression we get in this way is, up to some explicit constant,
    \[ \Gamma\left(d \cdot {\{a\}}\right)^{-1} \prod_{k=0}^{d-1} \Gamma\left( {\{a\}} +\frac{k}{d} \right), \]
    which can be computed with the multiplication formula. This procedure ultimately expresses $\Gamma(\alpha)^2$ as a product of fractional powers of integers and numbers of the form $\sin(\pi s)$ with $s \in \Q$, which can in turn be expressed in terms of roots of unity in the usual way.


    We have thus written $\Gamma(\alpha)^2$ in terms of known algebraic numbers and, hence, computed $\Gamma(\alpha)$ up to sign. We obtain an exact expression for $\Gamma(\alpha)$ by giving its minimal polynomial together with a sufficiently good complex approximation. This is done by first using the expression of $\Gamma(\alpha)^2$ in terms of algebraic numbers to determine its minimal polynomial $q(t) \in \mathbb{Q}[t]$. The polynomial $q(t^2) \in \mathbb{Q}[t]$ then vanishes at $\Gamma(\alpha)$, and this polynomial can then be factored over $\Q$ to extract its irreducible factors. 
    We then approximate $\Gamma(\alpha)$ numerically to sufficiently high precision and determine which of these irreducible factors has $\Gamma(\alpha)$ as a root by ruling out the others numerically. In this way, we obtain the minimal polynomial of $\Gamma(\alpha)$ and hence an exact representation.
\end{proof}

\subsection{Explicit calculations}{\label{section:examples}}
In this section, we give examples of the computation of $\Gamma(f)$ and $\Q(\varepsilon_{J_m})$ for certain equations $f$ and certain values of $m$.
We start by associating each equation for the Mumford-Tate group with a complex algebraic number via the Gamma function. 

\begin{definition}{\label{def: gamma value associated to equation for MT}}
    For a positive integer $m>2$, let $J_m/\Q$ be the Jacobian of the hyperelliptic curve $y^2=x^m+1$. Let $f$ be an equation for $\MT(J_m)$ in the sense of \Cref{def: equation for MT}. We associate with $f = x_1^{e_1} \cdots x_{m-1}^{e_{m-1}}$ its Gamma-value
    \begin{equation}{\label{eq:Gammaf}}
    \Gamma(f) \colonequals \prod_{j=1}^{m-1} \left[
    \Gamma\left(\tfrac{j}{m}\right) \cdot \Gamma\left(\tfrac{j}{m}\right) \cdot 
    \Gamma\left(\tfrac{[-2j]}{m}\right) \right]^{e_j},    
    \end{equation}
    where $[-2j]$ is the integer between $1$ and $m$ which is congruent to $-2j$ modulo $m$.
\end{definition}

We now have three ways of attaching a complex algebraic number to an equation $f$ for the Mumford-Tate group via the Gamma function. Two of them (\Cref{definition:character-gamma-value} and \Cref{def: gamma f}) are denoted $\Gamma(\gamma_f)$ and are equivalent to each other, see the comment after \Cref{def: gamma f}.
On the other hand, the number $\Gamma(f)$ of \Cref{def: gamma value associated to equation for MT} is not necessarily equal to $\Gamma(\gamma_f)$. However, the next proposition shows that $\Gamma(f)$ and $\Gamma(\gamma_f)$ carry essentially the same arithmetic information.

\begin{proposition}{\label{prop: comparison between gamma values}}
    Let $f$ be an equation for $\MT(J_m)$. %
    The quotient $\Gamma(f)/\Gamma(\gamma_f)$ is rational.
\end{proposition}

\begin{proof}
    Suppose first that $m$ is odd and write $m= 2g+1$. Recall that if $f$ is the Laurent monomial $x_1^{d_1} \dots x_{2g}^{d_{2g}}$, we denote by $q$ the sum of (the absolute values of) all the negative exponents, normalize $f \cdot \det^{2q} = x_1^{e_1} \dots x_{2g}^{e_{2g}}$, and consider the character associated with this normalized equation, namely $\gamma_f = \gamma_1^{e_1} \ast \dots \ast \gamma_{2g}^{e_{2g}}$ (see \Cref{eq: definition of gamma_f}). Note that $e_j = d_j+2q $ for every index $j$.

    Consider the character $\gamma= \gamma_1\ast \dots\ast\gamma_{2g}$ corresponding to the monomial  $\det(x)$. Recall from \Cref{eq: definition of gammai} that $\gamma_i = (i, i, -2i)$ and notice that
    \[\langle \gamma \rangle = \frac{1}{m}\left(\sum_{k=1}^{2g} 2\cdot k + [-2k]\right). \]
    Since multiplication by $-2$ is an automorphism of $\Z/m\Z$ and $2g=m-1$, the sum on the right-hand side evaluates to $3(m-1)/2 = 3g$. We repeat the same argument to write the Gamma-value of the character $\gamma$ (as per \Cref{definition:character-gamma-value}) as
    \[
        \Gamma(\gamma) = (2\pi i)^{-\langle \gamma \rangle } \cdot \prod_{k=1}^{m} \Gamma\left(\frac{k}{m}\right)^3.
    \]
    We apply the multiplication formula with $(n,s)=(m,1/m)$ and get
    \[
        \Gamma(\gamma) = (2\pi i)^{-3g} \cdot \left[ (2\pi)^{\frac{m-1}{2}} \cdot m^{\frac{1}{2}-m\cdot \frac{1}{m}} \cdot \Gamma(1) \right]^3 = i^{-3g} \cdot m^{-3/2}.
    \]
    Notice that all the factors $2\pi$ cancel out because $-3g + 3(m-1)/2 = 0$.
    
    Consider now the quotient
    \begin{equation}\label{eq: ratio two definitions gamma value}
        \frac{\Gamma(\gamma_f)}{\Gamma(\gamma)^{2q}} =
        \frac{(2 \pi i)^{-\langle \gamma_f \rangle} \cdot \prod_{j = 1}^{2g}\left[ \Gamma\left(\frac{j}{m}\right)^2 \cdot \Gamma\left(\frac{[-2j]}{m}\right) \right]^{e_j} }
        {(2 \pi i)^{-2q\cdot\langle \gamma \rangle} \cdot \prod_{j = 1}^{2g}\left[ \Gamma\left(\frac{j}{m}\right)^2 \cdot \Gamma\left(\frac{[-2j]}{m}\right) \right]^{2q} }.
    \end{equation}
    Recall from \Cref{prop:ShiodaIII} (2) that $\langle \gamma_f \rangle = 2q \cdot 3g = 2q \cdot \langle \gamma \rangle $. The factors $(2\pi i)$ thus cancel out. Comparing what's left on the right-hand side of \Cref{eq: ratio two definitions gamma value} with \Cref{def: gamma f}, we find the identity
    \[ \frac{\Gamma(\gamma_f)}{\Gamma(\gamma)^{2q}} = \Gamma(f). \]
    Since $\Gamma(\gamma)^{2q} = m^{-3q} i^{-3(m-1)q} = m^{-3q}(-1)^{3gq}$ is a rational number, $\Gamma(f)$ is a rational multiple of $\Gamma(\gamma_f)$, as desired.

    If $m$ is even, essentially the same computation (but skipping $k= m/2$) gives $\langle \gamma \rangle = 3(m-2)/2=3g $. Computing $\Gamma(\gamma)$ is more challenging in this case: we use the multiplication formula with $(n,s)=(m,1/m)$ and $(n,s)=(m/2,2/m)$ to get
    $$\prod_{k=1}^{m} \Gam{k}{m} = (2\pi)^{\tfrac{m-1}{2}} \cdot m^{-\tfrac{1}{2}}\cdot \Gamma(1), \qquad
    \prod_{k=1}^{m/2} \Gam{2k}{m} = (2\pi)^{\tfrac{(m/2)-1}{2}} \cdot (\tfrac{m}{2})^{-\tfrac{1}{2}}\cdot \Gamma(1), $$
    which allows us to compute
    \[ \Gamma(\gamma) = (2\pi i)^{-3g}\cdot \prod_{k=1}^m\Gam{k}{m}^2 \cdot \frac{1}{\Gamma(1/2)^2}  \cdot \prod_{k=1}^{m/2}\Gam{2k}{m}^2  = i^{-3g} \cdot 4 \cdot m^{-2}. \]
    Since $2q$ is even, $\Gamma(\gamma)^{2q}$ is rational and we conclude as in the case of odd $m$.
\end{proof}

We state a corollary of \Cref{th:Kconn} which makes our description of the connected monodromy field $\Q(\varepsilon_{J_m})$ more intelligible and explicit computations easier.

\maincorollary

\begin{proof}
    The statement follows from \Cref{th:Kconn} and \Cref{prop: comparison between gamma values}.
\end{proof}

\begin{remark}
One could also apply \Cref{prop: comparison between gamma values} to \Cref{th:Kconn-even}, and get a similar statement for the case of even $m$.
\end{remark}

We now have the tools to compute in a completely explicit way the Gamma-value associated with any equation $f$ for $\MT(J_m)$. The next example shows that the Gamma-values associated with equations `coming from the polarization' are contained in the endomorphism field.

\begin{example}\label{rmrk:expected}
    When $f\colon  x_kx_{m-k}/x_jx_{m-j}=1$ is one of the equations for $\MT(J_m)$ expressing the expected inclusion $\MT(J_m)\subseteq\GSp_V$ (see \cite[Remark 2.4.3]{part2} for more details), the corresponding Gamma-value is fairly easy to compute. Using the reflection property we get
    \begin{equation}\label{eq: reflection and sine}
    \vartheta_k \colonequals {\textstyle {\Gam{k}{m}\Gam{m-k}{m}}} = {\sin(k\pi/m)} = \frac{e^{ki\pi/m}-e^{-k i\pi/m}}{2i} = \frac{\zeta_{2m}^k-\zeta_{2m}^{-k}}{2i}.
    \end{equation}
By definition, $\Gamma(f)$ is given by $\frac{\vartheta_k^2 \vartheta_{[-2k]}}{\vartheta_j^2 \vartheta_{[-2j]}}$. Expanding and using \Cref{eq: reflection and sine}, it is easy to see that this number lies in $\Q(\zeta_m)\subseteq \Q(\End (J_m))$.
\end{example}

Finally, we compute the connected monodromy field of a specific Fermat Jacobian and show that it is a non-trivial extension of the endomorphism field.

\begin{example}\label{ex: connected monodromy field for m=15}
    Let $m=15$. We computed the equations defining the Mumford-Tate group of $J_{15}$ in \Cref{ex:m15-equations}. We then compute $\Gamma(f)$ for every equation $f$ using \Cref{eq:Gammaf}. Notice that the first six equations are in the form of \Cref{rmrk:expected}, so their corresponding Gamma-values are in $K=\Q(\zeta_{15})$. The computation of the values corresponding to the last three equations can still be carried out by hand, and gives
    \begin{equation*}
    \begin{aligned}{}
    \Gamma\left(\frac{x_9x_{12}}{x_8x_{13}}\right)
    &= \frac{1}{2}\cdot  \sqrt{\frac{3}{5}} \cdot \frac{1}{\sin(3\pi/15)}\\
    \Gamma\left(\frac{x_{11}x_{12}}{x_9x_{14}}\right)
    &= \frac{\sqrt{3}}{8} \cdot \frac{1}{\sin(4 \pi /15)^2} \cdot \frac{1}{\sin(7\pi/15)}.
    \end{aligned}
    \end{equation*}
    One can check that both of these values are in $K$. Only the last equation gives a value in a quadratic extension of $K$: 
    \begin{equation*}
    \Gamma\left(\frac{x_{10}x_{12}}{x_{8}x_{14}}\right)^2 =
    \frac{3}{16} \cdot \sqrt{5} \cdot 
    \frac{\sin(\pi/15)^3\sin(6\pi/15)}{\sin(2\pi/15)\sin(3\pi/15)^4\sin(5\pi/15)^3}.
    \end{equation*}
In particular, $\Q(\varepsilon_{J_{15}})=K(\varepsilon_{J_{15}})$ is the quadratic extension of $K= \Q(\zeta_{15})$ which we obtain by adjoining the square root of this horrendous number.
With the assistance of a computer, we find that $ \Q(\varepsilon_{J_{15}}) $ is the number field
\begin{center}
    {\texttt{ \cite[\href{https://www.lmfdb.org/NumberField/16.0.3243658447265625.1}{number field 16.0.3243658447265625.1}]{lmfdb} }}
\end{center}
This result agrees with the prediction based on moments statistics in \cite[5.2]{Heidi}.
\end{example}

\begin{example}\label{example: odd numbers up to 105}
We implemented the algorithm suggested by \Cref{cor: connected monodromy field for odd m in terms of equations} and ran it for all odd values of $m$ in the interval $3 \leq m \leq 105$. We found that the Galois group of $\Q(\varepsilon_{J_m})$ over $\Q(\zeta_m)$ is $(\Z/2\Z)^r$, with $r$ as in the following table, where we denote by $\omega(m)$ the number of distinct prime factors of $m$.

\begin{center}
\begin{tabular}{c|c}
     $r$ & Values of $m$ \\ \hline
     $0$ & Primes and prime powers ($\omega(m)=1$) \\
     $1$ & $\{ 3 \leq m \leq 105 : m \text{ odd},\, \omega(m)=2 \} $ \\
     $3$ & $m=105$ (the unique value with $\omega(m)=3$)
\end{tabular}
\end{center}
We will show below that the Galois group $\Gal(\Q(\varepsilon_{J_m})/\Q(\zeta_m))$ is of the form $(\Z/2\Z)^r$ for every positive integer $m$ (\Cref{thm: Kconn is multiquadratic}) and explain the first row of the table in \Cref{thm: prime power case Kconn computation}.

Note that we were able to compute $\Q(\varepsilon_{J_m})$ for abelian varieties of dimension up to $52 = \frac{105-1}{2}$; many of these varieties have a very rich endomorphism structure. This shows that the algorithm is quite practical.

Finally, the different behaviour of $m=105=3 \cdot 5 \cdot 7$ strongly suggests that the rank $r$ should grow with the number of prime factors of $m$. In particular, we see no reason why the rank $r$ should be uniformly bounded, but it seems difficult to test this conjecture computationally.
\end{example}

\section{Further properties of the connected monodromy field of Fermat Jacobians}{\label{section:further-properties}}
In this section we use the results established so far to prove some general properties of the field $\Q(\varepsilon_{J_m})$. Specifically, in \Cref{sec: multiquadratic} we prove that the extension $\Q(\varepsilon_{J_m}) / \Q(\End(J_m))$ is multiquadratic for every integer $m > 2$. Moreover, in \Cref{sec:nondegeneratefactor} we show that, for every odd prime $p$ and positive integer $m$, each simple factor of $J_{p^m}$ is a non-degenerate CM abelian variety (that is, the rank of its Hodge group is equal to its dimension). This is already known to be true for $p=2$ due to nondegeneracy results of \cite{EmoryGoodson2024}. We also deduce that the field $\Q(\varepsilon_{J_m})$ coincides with $\Q(\zeta_m)$ for all odd prime powers $m=p^n$, see \Cref{thm: prime power case Kconn computation}.

\subsection{\texorpdfstring{$\Q(\varepsilon_{J_m}) / \Q(\End(J_m))$}{Qconn(Jm)/Q(End(Jm))} is multiquadratic}{\label{sec: multiquadratic}}
Numerous examples {(see in particular \Cref{example: odd numbers up to 105})} led us to suspect that the extension $\KconnJ/K$ is always multiquadratic, that is, the Galois group of the extension $\KconnJ/K$ is isomorphic to the product of $r\geq 0$ copies of $\Z/2\Z$.
We can prove this fact in general.

\TheoremMultiquadratic
\begin{proof}
Write for simplicity $J=J_m$. We already know that $\Q(\varepsilon_{J})$ contains $\Q(\End(J))$, by \cite[Proposition 2.10]{MR1355128}, and that $\Q(\End(J))$ contains $K=\Q(\zeta_m)$, by \Cref{prop: endomorphism field easy containment}. Thus, we can equivalently work with $\Q(\varepsilon_J)$ or $K(\varepsilon_J)$. 
The proof relies on two main observations:
\begin{enumerate}
    \item The connected monodromy field $K(\varepsilon_J)$ is contained in $K(J[N])$, the field of definition of the $N$-torsion, for every {$N \geq 3$}. This holds for every abelian variety \cite[Proposition 3.6]{Chi}.
    \item The field extension $K(J[4])/K(J[2])$ is multiquadratic for every abelian variety. The Galois group of this extension is a subgroup of $$\ker\left[ \GL_{2g}(\Z/4\Z) \to \GL_{2g}(\Z/2\Z) \right],$$ which is an abelian group of exponent 2. Indeed, any element in the kernel can be written as $\operatorname{Id} + 2M$ for a suitable matrix $M$ with coefficients in $\Z/4\Z$; a direct computation shows immediately that its square is the identity modulo 4. For the case where $J$ is the Jacobian of a hyperelliptic curve, Yelton gives explicit generators for the extension $K(J[4])/K(J[2])$ in \cite[Proposition 3.1, Remark 4.2.b]{yelton2014images}
\end{enumerate}
The field of definition of the 2-torsion $\Q(J[2])$ of {the Jacobian $J$ of} a hyperelliptic curve $y^2=f(x)$ {defined over $\Q$} is the splitting field of the polynomial $f$.
If $m$ is odd {and $J/\Q$ is} the Jacobian of the curve $C_m\colon y^2=x^m+1$, {the field} $\Q(J[2])$ is $\Q(\zeta_{2m})=\Q(\zeta_{m})$. From $(2)$, we get that the extension $\Q(J[4])/\Q(\zeta_{m})$ is multiquadratic, and $(1)$ gives that $\Q(\varepsilon_J)/\Q(\zeta_m)$ is a subextension of $\Q(J[4])/\Q(J[2])$, from which the claim follows.

The case of even $m$ is a bit more involved. {If $m \equiv 2 \pmod 4$, \Cref{th:Jm-factorization} implies that $J_m$ is isogenous over $\Q$ to the square of $J_{m/2}$. The connected monodromy fields of $J_m$ and $J_{m/2}$ are then equal, and the claim follows easily. Suppose therefore $m \equiv 0 \pmod{4}$, so that $i \in \Q(\zeta_m)$.} We consider the  Jacobian $J'$ of the twist $y^2 = x^m-1$, for which the claim follows {as above}. Notice that the twist $y^2=x^m+1$ (which is isomorphic over $\Q(\zeta_m)$ to $y^2=-x^m-1$, since we can replace $y$ with $iy$) of $y^2=x^m-1$ corresponds to the cocycle $$\sigma \mapsto \left[(x,y) \mapsto \left(\frac{\sigma(\zeta_{2m})}{\zeta_{2m}}\cdot x, y \right)\right].$$
Since we aim to understand the extension $K(\varepsilon_J)/K$, we consider the cocycle above as a function on the absolute Galois group $\Gamma_K$ of $K=\Q(\zeta_m)$. For any element $\sigma \in \Gamma_K$,
since $\zeta_{2m}$ has degree 2 over $K$ when $m$ is even, the ratio $\sigma(\zeta_{2m})/\zeta_{2m}$ is $\pm 1$.
By the same argument as in \Cref{th:Kconn-even}, if $K(\varepsilon_J)$ is generated over $K$ by some Gamma-values $A_i := \Gamma(\gamma_{f_i})\cdot \mu^{e(f_i)}$ for $i=1, \dots, r$, then $K(\varepsilon_{J'})$ is generated by twists $B_i$ of the same values, namely $B_i = A_i \cdot \zeta_{2m}^{e(f_i)}$. Here $\mu$, $e(f)$, and $r$ are as in \Cref{th:Kconn-even}.
The proof of \cite[Theorem 7.15]{Deligne} shows that $A_i^{2m} \in K$. Indeed, Deligne shows that $\Gamma(\gamma_{f_i})$ generates an abelian extension of $K$ and that, for every $\sigma \in \Gal(\overline{K}/K)$, there is a root of unity $\zeta \in K$ such that $\sigma \Gamma(\gamma_{f_i}) = \zeta \cdot \Gamma(\gamma_{f_i})$ (Deligne proves this for Frobenius elements, which are dense). This implies that $\Gamma(\gamma_{f_i})^m \in K$ and hence $A_i^{2m} \in K$.
It follows that $B_i^{2m}$ belongs to $K$ for every $i$. The field extension $K(\varepsilon_{J'})/K$ is determined via Galois theory by the kernel of the homomorphism $$\Gamma_K \to \mu_{2m}^r, \quad \sigma \mapsto (\sigma(B_i)/B_i)_{i=1, \dots, r}. $$Since this field extension is multiquadratic, the image of the homomorphism is actually contained in $\mu_2^r$ and the same follows for the twisted homomorphism
\[ \sigma \mapsto \frac{\sigma(A_i)}{A_i} = \frac{\sigma(B_i)}{B_i} \cdot \left(\frac{\sigma(\zeta_{2m})}{\zeta_{2m}}\right)^{e(f_i)} = \pm \frac{\sigma(B_i)}{B_i},  \]
proving that the extension $K(\varepsilon_J)/K$ is multiquadratic as well.
\end{proof}

\subsection{Nondegeneracy in the prime power case}\label{sec:nondegeneratefactor}
Let $p$ be an odd prime and $k$ be a positive integer. In this section, we prove that the factor $X_{p^k}$ of $J_{p^k}$ is a non-degenerate CM abelian variety and that the projection onto this factor induces an isomorphism between the corresponding Mumford-Tate groups.
We exploit this property to compute the connected monodromy field of $J_{p^k}$: we prove in \Cref{thm: prime power case Kconn computation} that $\Q(\varepsilon_{J_{p^k}}) = \Q(\zeta_{p^k})$.

\medskip

We begin by showing that the abelian variety $X_{p^k}$, the largest simple factor of $J_{p^k}$, is nondegenerate. Since this holds for every $k$, by \Cref{th:Jm-factorization} we obtain that every simple factor $X_{p^i}$ of $J_{p^k}$ is nondegenerate. The proof is based on results of Kubota \cite{Kubota1965} which relate the non-degeneracy of a CM abelian variety to certain (essentially combinatorial) properties of its CM type. The verification of these properties, in turn, can be turned into questions of algebraic number theory: we begin with two lemmas that solve these number-theoretic reformulations.

\begin{lemma}\label{lemma:charactersum1pk} Let $p$ be an odd prime, $k$ be a positive integer, and $\chi$ be a character of $(\Z/p^k\Z)^\times$ of conductor $p^h$, where $1\leq h \leq k$.
Write $n=\frac{p^k-1}{2}$. If $\chi(-1)=-1$ then
    $$\sum^{2n}_{a=1} a \cdot \chi(a)\not=0.$$
\end{lemma}

\begin{proof}
    To prove the result, we split the sum along residue classes modulo $p^h$ and use basic properties of characters:
    \begin{align*}
        \sum^{2n}_{a=1} a\chi(a) &=\sum^{p^h-1}_{b=1}\;\sum^{p^{k-h}-1}_{j=0} (p^hj+b)\chi(p^hj+b)\\
            &=\sum^{p^h-1}_{b=1}\chi(b)\sum^{p^{k-h}-1}_{j=0} (p^hj+b)\\
            &=(p^{k-h})\sum^{p^h-1}_{b=1}b\chi(b)+ \left(\sum^{p^{k-h}-1}_{j=0} p^hj\right)\left(\sum^{p^h-1}_{b=1}\chi(b)\right).
    \end{align*}
Since $\sum^{p^h-1}_{b=1}\chi(b)=0$, the above sum equals $p^{k-h}\sum^{p^h-1}_{b=1}b\chi(b)$.

It remains to show that the sum $\sum^{p^h-1}_{b=1}b\chi(b)$ is non-zero. This quantity is equal to $p^hB_{1,\chi}$, where $B_{1, \chi}$ is a generalized Bernoulli number, see \cite[p.~32]{Washington} and notice that the conductor of $\chi$ is exactly $p^h$. Bernoulli numbers can be computed as special values of Dirichlet $L$-functions: since $\chi(-1)=-1$, we have $B_{1, \chi} = \frac{p^h}{\pi i \tau(\overline{\chi})} L(1, \overline{\chi})$ by \cite[Theorem 4.9]{Washington}, where $\tau(\overline{\chi})$ is a Gauss sum.
This expression is non-zero since $L(1, \overline{\chi}) \neq 0$ by \cite[Corollary 4.4]{Washington}.
\end{proof}

\begin{lemma}\label{lemma:charactersum2pk}
Let $p$ be an odd prime, $k$ be a positive integer, and $\chi$ be a character of $(\Z/p^k\Z)^\times$ of conductor $p^h$, where $1\leq h \leq k$.
Write $n=\frac{p^k-1}{2}$. If $\chi(-1)=-1$ then 
    $$\sum^{n}_{a=1} \chi(a)\not=0.$$
\end{lemma}

\begin{proof}
    The proof follows that of Lemma 3 in \cite{Kubota1965}, with the main difference being that we are working with a prime power instead of simply a prime. %
    Let $\Theta=\sum_{a=1}^{2n} a\chi(a)$; this sum is nonzero by the result in Lemma \ref{lemma:charactersum1pk}. We define the following sums:
    \begin{align*}
        A&=\sum_{a=1}^{n} a\chi(a), & A'&=\sum_{a=n+1}^{2n} a\chi(a),\\
        A_1&=\sum_{a=1}^{n} \chi(2a-1)(2a-1), &  A_2&=\sum_{a=1}^{n} \chi(2a)(2a).\\
    \end{align*} 
    Thus, we have both $\Theta=A+A'$ and $\Theta=A_1+A_2$. Let $B=\sum_{a=1}^{n} \chi(a)$ be the sum in the statement of the lemma that we aim to prove {is nonzero}. We can, therefore, write
\[
\begin{aligned}
A' & =\sum_{a=1}^{n} (p^k-a)\chi(p^k-a)=p^k\sum_{a=1}^{n}\chi(p^k-a)-\sum_{a=1}^{n} a\chi(p^k-a)\\
& = p^k\sum_{a=1}^{n}\chi(-a)-\sum_{a=1}^{n} a\chi(-a) \\
& = -p^k\sum_{a=1}^{n}\chi(a)+\sum_{a=1}^{n} a\chi(a)= -p^kB+A.
\end{aligned}
\]
    Hence, $\Theta=2A-p^kB$.
    We now relate $A_1$ and $A_2$ to the other expressions. First note that $$A_2=2\chi(2)\sum_{a=1}^{n} a\chi(a)=2\chi(2)A.$$ Following the techniques of \cite{Kubota1965}, we find that $A_1=2\chi(2)A-p^k\chi(2)B.$ Hence, $\Theta=A_1+A_2=4\chi(2)A-p^k\chi(2)B$.
    Combining these results yields
    $$4\chi(2)A-p^k\chi(2)B=2A-p^kB$$
    which implies that $2A(2\chi(2)-1)=p^kB(\chi(2)-1).$ Since $\Theta=2A-p^kB$, we can rewrite this as $(\Theta +p^kB)(2\chi(2)-1)=p^kB(\chi(2)-1),$ and so
    $$\Theta(2\chi(2)-1)=-p^kB\chi(2).$$
    Since neither $\Theta$ nor $2\chi(2)-1$ equals 0, we have proved that $B\not=0$ which is the desired result.
\end{proof}

Note that for any character $\chi$ of $(\Z/p^k\Z)^\times$ we have $\chi(a)=0$ whenever $\gcd(a,p)>1$. Hence, the result in Lemma \ref{lemma:charactersum2pk} still holds when we restrict $a$ to be relatively prime to $p$:
$$\sum^{n}_{a=1,\,p\nmid a} \chi(a)\not=0.$$
This is {the sum of the values} of a Dirichlet character $\chi$ of $(\Z/p^k\Z)^\times {\cong \Gal(\Q(\zeta_{p^k})/\Q)}$ evaluated at each embedding $\sigma_j: \zeta_{p^k}\to \zeta_{p^k}^j$, where $1\leq j\leq n$ satisfies $(j,p)=1$. This fact is a key ingredient in proving the following result.

\nondeg

{
Before proving \Cref{th:nondegenarcy}, we recall the definition of \textit{rank} and \textit{defect} of a CM type, at least in the setting we are interested in. Let $A$ be an abelian variety with complex multiplication by the CM type $\Phi$. The \textit{rank} of $\Phi$ is the dimension of the Mumford-Tate group of $A$, and its \textit{defect} is the difference $\dim A + 1-\operatorname{rank}(\Phi)$. Note that our definition (superficially) differs from that given in \cite[\S2]{Kubota1965}: the equality of the two definitions follows from combining \cite[Lemma 1]{Kubota1965} with the description of the Mumford-Tate group of a CM abelian variety given in \Cref{prop:computeMT}. See also \cite[\S2]{Banaszak2003} for further details.
}

\begin{proof}
    By replacing $k$ with a smaller integer, we can assume that $X$ is the largest simple factor $X_{p^k}$ of $J_{p^k}$. It was proved in \Cref{lemma: CM type of X_m} that $X_{p^k}$ is a CM abelian variety with CM type $(\Q(\zeta_{p^k}), \, {\Phi_{p^k}})$. By Lemma \ref{lemma:charactersum2pk} we have that if $\chi(\rho)=-1$, where $\rho$ denotes complex conjugation in $\Q(\zeta_{p^k})$, then
    $$\sum_{\sigma_j} \chi(\sigma_j)\not=0.$$
    Thus, by Lemma 2 of \cite{Kubota1965}, the defect 
    of the CM type is 0, which implies that the simple abelian variety $X_{p^k}$ is nondegenerate by \cite[Theorem 2.7]{Hazama2} or
    \cite[Theorem 1.2]{Hazama89}. 
\end{proof}

We now show that, in the prime-power case $m=p^k$, the problem of determining the Mumford-Tate group of the Fermat Jacobian $J_m$ reduces to the problem of computing the Mumford-Tate group of $X_m$. This is much easier since $X_m$ is of CM-type and non-degenerate: we recall the description of the Mumford-Tate group in this case (see \cite[comments after Proposition 3.3]{MR0608640} and \cite[Remark 2.12, Section 6]{MR3766118}). If $A$ is a simple, non-degenerate CM abelian variety with CM by the field $E$, then the Mumford-Tate group of $A$ is given as a functor on $\Q$-algebras $R$ by the formula
\[
\MT(A)(R) = \{ x \in (E \otimes_{\Q} R)^\times : x\overline{x} \in R^\times \},
\]
where $x \mapsto \overline{x}$ is induced by the complex conjugation of $E$.

\MTprojection
\begin{proof}
{\Cref{th:Jm-factorization} gives the isogeny decomposition (over $\Q$)
\[
J_{p^k} \sim \prod_{1 \leq i \leq k} X_{p^i} \sim J_{p^{k-1}} \times X_{p^k},
\]
where $X_{p^i}$ is geometrically simple with endomorphism algebra $\Q(\zeta_{p^i})$. We prove the claim by induction, the case $k=1$ being trivial. The induction hypothesis shows that $J_{p^{k-1}} \to X_{p^{k-1}}$ induces an isomorphism $\MT(J_{p^{k-1}}) \xrightarrow{\sim} \MT(X_{p^{k-1}})$, so it suffices to prove that the projection $X_{p^k} \times X_{p^{k-1}} \to X_{p^k}$ induces an isomorphism
$\operatorname{MT}(X_{p^k} \times X_{p^{k-1}}) \xrightarrow{\sim} \operatorname{MT}(X_{p^k})$. This is what we show in the rest of the proof.
}

{For the purposes of this argument, we denote by $[[a,b]]$ the set of integers in the interval $[a, b]$.}
The set $\left[\left[1,\frac{p^k-1}{2}\right]\right]$ can be written as the union 
\begin{equation}\label{eq:pkinterval}
    \left[\left[1,p^{k-1}\frac{p-1}{2}\right]\right]\cup\left[\left[p^{k-1}\frac{p-1}{2}+1,p^{k-1}\frac{p-1}{2} + \frac{p^{k-1}-1}{2}\right]\right].
\end{equation}
We can express the set on the left as
$$\left[\left[1,p^{k-1}\right]\right]\cup\left[\left[p^{k-1}+1,2p^{k-1}\right]\right]\cup\left[\left[2p^{k-1}+1,3p^{k-1}\right]\right]\cup\cdots\cup\left[\left[\frac{p-3}{2}p^{k-1}+1,\frac{p-1}{2}p^{k-1}\right]\right].$$
If we now take the reduction modulo $p^{k-1}$ of the values in these intervals, we obtain $({p-1})/{2}$ copies of the interval $\left[\left[1,p^{k-1}\right]\right]$. The elements of this interval that are relatively prime to $p$ are in bijection with $\left(\mathbb Z/p^{k-1}\mathbb Z\right)^\times$. 

On the other hand, if we take the reduction modulo $p^{k-1}$ of the values in the rightmost interval in Equation \eqref{eq:pkinterval}, we obtain $\left[\left[1,({p^{k-1}-1})/{2}\right]\right]$. The elements of this interval that are relatively prime to $p$ give us the CM type of $X_{p^{k-1}}$: 
$$\Phi_{p^{k-1}}=\left\{\sigma_i\;|\; \gcd(i,p)=1 \text{ and } 1\leq i\leq \frac{p^{k-1}-1}{2} \right\},$$
{where $\sigma_j$ is the embedding of $\Q(\zeta_{p^{k-1}})$ in $\C$ that sends $\zeta_{p^{k-1}}$ to $\exp\left( 2\pi i j / p^{k-1}\right)$.}

Following \cite[\S 4]{MR4557876}, we denote by $E_1=\Q(\zeta_{p^k}), E_2=\Q(\zeta_{p^{k-1}})$ the CM fields of $X_{p^k},\, X_{p^{k-1}}$ respectively and set $L=E_1$. {We also use the notation of \Cref{section:MT-1}, so that in particular $N_i : T_{L} \to T_{E_i}$ is the norm map and $\phi_i^\ast : T_{E_i} \to T_{E_i}$ is the reflex norm.}
By \cite[Lemma 4.2]{MR4557876}, the projection $\operatorname{MT}(X_{p^k} \times X_{p^{k-1}}) \to \operatorname{MT}(X_{p^k})$ is an isomorphism if and only if
\[
\widehat{T}_{E_1} \times \{0\} + \ker\left(N_1^* \phi_1^* + N_2^*\phi_2^* \right) = \widehat{T}_{E_1} \times \widehat{T}_{E_2}.
\]
Note that $N_1$ is the identity in our case. The above equality holds if and only if, for every $v \in \widehat{T}_{E_2}$, there exists $u \in \widehat{T}_{E_1}$ such that $(u, v) \in \ker\left(N_1^* \phi_1^* + N_2^*\phi_2^* \right) = \ker\left(\phi_1^* + N_2^*\phi_2^* \right)$. Equivalently, for every $v \in \widehat{T}_{E_2}$ there should exist $u \in \widehat{T}_{E_1}$ such that $\phi_1^*(u) = -N_2^*\phi_2^*(v)$. Replacing $u$ with $-u$, this requirement is equivalent to the fact that every vector in the image of $N_2^*\phi_2^*$ is also in the image of $\phi_1^*$. This will follow easily once we prove \Cref{lemma: prime power decomposition computation}.

\begin{lemma}{\label{lemma: prime power decomposition computation}}
For every $\tau$ in $\operatorname{Hom}(\Q(\zeta_{p^{k-1}}), \overline{\Q})$ we  have
\begin{equation}{\label{eq: prime power decomposition computation}}
\sum_{\substack{\sigma \in \operatorname{Hom}(\Q(\zeta_{p^k}), \overline{\Q}) \\ \sigma|_{\Q\left(\zeta_{p^{k-1}}\right)}= \tau}} \phi_1^* [\sigma] = \frac{p-1}{2} \left( \sum_{\sigma \in \operatorname{Hom}(\Q(\zeta_{p^k}), \overline{\Q})} [\sigma] \right) + N_2^*\phi_2^* [\tau].
\end{equation}
\end{lemma}

\begin{proof}
We identify the Galois group of $\Q(\zeta_{p^k})$ with $\left(\Z/p^k\Z\right)^\times$.
Expanding the expression, the left-hand side of \eqref{eq: prime power decomposition computation} yields
\[
\sum_{i = 0}^{p-1} \sum_{a \in \Phi_{p^k}} [a^{-1}\cdot (\tau + i\cdot p^{k-1}). ]
\]
Notice that, since $E_i / \Q$ is Galois, the reflex type of $\Phi_{p^k}$ is the subset of $\Gal(\Q(\zeta_{p^k})/\Q)$ given by $\{a^{-1} : a \in \Phi_{p^k}\}$.
We now use \eqref{eq:pkinterval} to decompose the above sum into simpler ones, namely we write it as
\[
\sum_{i = 0}^{p-1} \left[ \sum_{a =1}^{p^{k-1}} + \sum_{a=p^{k-1}+1}^{2p^{k-1}} + \dots + \sum_{a=p^{k-1}\cdot \frac{p-1}{2}+1}^{p^{k-1}\cdot \frac{p-1}{2}+\frac{p^{k-1}-1}{2}} \right] [a^{-1}\cdot (\tau + i\cdot p^{k-1}) ] , 
\] 
where we omit from the notation the condition of coprimality $(a,p)=1$. The first (double) sum gives all residues modulo $p^k$ exactly once. The same argument applies to the first $(p-1)/2$ (double) sums. The last double sum is taken over all residues modulo $p^k$ that reduce modulo $p^{k-1}$ to a residue in $\phi_2^\ast[\tau]$, so the claim follows.
\end{proof}

We can now finish the proof of the proposition. We had reduced to showing that the image of $N_2^*\phi_2^*$ is contained in the image of $\phi_1^*$. 
The first sum on the right-hand side of \Cref{eq: prime power decomposition computation} is in the image of $\phi_1^*$, because it can be written as $\phi_1^*[\sigma_1]+ \phi_1^{*}[\sigma_{-1}]$. Thus, for every $\tau$ in $\operatorname{Hom}(\Q(\zeta_{p^{k-1}}), \overline{\Q})$, we have written $N_2^*\phi_2^*[\tau]$ as a linear combination of elements in the image of $\phi_1^*$. This implies the desired inclusion $\operatorname{Im} (N_2^*\phi_2^*) \subseteq \operatorname{Im} \phi_1^*$ and concludes the argument.

\end{proof}

\Cref{th: MT projection prime power case} allows us to give an explicit description of the Mumford-Tate group of $J_{p^k}$. Let $x_1, \ldots, x_{2g}$ be the coordinates on the diagonal torus $T$ of $\GL_{2g, \Q(\zeta_{p^k})}$, as in \Cref{subsec:MTcomp}.
The differential forms $x^{i-1}\, dx/y$ with $(i, p)=1$ give a natural basis for the complex uniformization of $X_{p^k}$, see the proof of \Cref{lemma: CM type of X_m}. These forms generate the eigenspaces for the action of $\alpha_{p^k}^*$ on the cohomology of $X_{p^k}$. Thus, the group $\MT(X_{p^k})_{\Q(\zeta_{p^k})}$ is contained in a natural way in the subtorus of $T$ with character group $\bigoplus_{(i,p)=1} \mathbb{Z} x_i$ (notation as in Equation \eqref{eq: character group of T}), and the projection $\pi : \MT(J_{p^k}) \to \MT(X_{p^k})$ is given on $\Qbar$-points by
\[
\begin{array}{cccc}
\pi : & \MT(J_{p^k})(\Qbar) & \to & \MT(X_{p^k})(\Qbar) \\
& \operatorname{diag}(x_1, \ldots, x_{2g}) & \mapsto & \operatorname{diag}(x_i)_{(i,p)=1}.
\end{array}
\]
By \Cref{th: MT projection prime power case}, this map is an isomorphism. Denote by $\psi$ its inverse. For every $a=1,\ldots,2g$, the composition $\psi_a := x_a \circ \psi$ is a character of $\MT(X_{p^k})$, hence a Laurent monomial in the variables $x_i$ with $(i,p)=1$, and the above discussion shows that the $x_a$-coordinate of an element $\operatorname{diag}(x_1, \ldots, x_{2g}) \in \MT(J_{p^k})(\Qbar)$ is the evaluation of $\psi_a$ at $(x_i)_{(i,p)=1}$. In other words, the coordinates $x_i$ with $p \mid i$ can be expressed as rational functions in the others.
We now try to characterize the functions $\psi_a$ more precisely.

Recall that $2g = p^k-1$. We identify the indices of the diagonal coordinates $1 \leq i \leq p^k-1$ with the non-zero elements of $\Z/p^k\Z$, which in turn we identify in a natural way with elements of the various Galois groups:
\[
    \Gal(\Q(\zeta_{p^n})/\Q) \simeq \left(\Z/p^n\Z\right)^\times \hookrightarrow p^{k-n}\Z/p^k\Z \subseteq \Z/p^k\Z \qquad \text{for } n = 0,\,\dots,\, k. 
\]
With this notation, for all $n=1, \ldots, k$ we have
\[
\widehat{T}_{\Q(\zeta_{p^n})} = \bigoplus_{\substack{a \equiv 0 \pmod{p^{k-n}} \\ a \not \equiv 0 \pmod{p^{k-n+1}}}} \Z x_a.
\]
Consider the various one-step norm maps
$N\colon \Q(\zeta_{p^{n}}) \to \Q(\zeta_{p^{n-1}})$
and the corresponding maps of tori. Write in particular
$N^\ast \colon \widehat{T}_{\Q(\zeta_{p^{n-1}})} \to \widehat{T}_{\Q(\zeta_{p^{n}})}$
for the induced map on characters. For $n \geq 2$, a basis of $\widehat{T}_{\Q(\zeta_{p^{n-1}})}$ is given by the $x_a$ whose index $a$ has $p$-adic valuation $v_p(a)=k-(n-1)$. Write such an $a$ as $a=pa'$, where $v_p(a')=k-n$. %
The map ${N}^\ast$ is given by
\begin{equation}{\label{eq: one-step norm map}}
    N^\ast(x_a) = \sum_{j=0}^{p-1} x_{a'+j \cdot p^{k-1}}.
\end{equation}
Note that the indices $a' + j \cdot p^{k-1}$ are precisely the classes $t$ in $\Z/p^k\Z$ that satisfy $t\cdot p \equiv a \pmod{p^k}$. Each of these indices has valuation $k-n$ since $k-1>k-n$, hence $\sum_{j=0}^{p-1} x_{a'+j \cdot p^{k-1}}$ is a well-defined element of $\widehat{T}_{\Q(\zeta_{p^{n}})}$.
For fixed $a$ and $a'$ as above, fix now an arbitrary partition of the set of indices $\{a'+j \cdot p^{k-1}\}$ into two subsets $B, \, C$ with $|B|=|C|-1=(p-1)/2$ and consider the expression
\begin{equation}{\label{eq: additive MT equation for last pk proof}}
    x_a + \sum_{b \in B} x_{p^k-b} - \sum_{c \in C} x_c \; \in \bigoplus_i \Z x_i.
\end{equation}
Any such expression has total degree 0 and we claim that the corresponding function%
\begin{equation}{\label{eq: additional fa equation form}}
    f_a \colon \qquad \frac{x_a \cdot \prod_{b \in B} x_{p^k-b} }{ \prod_{c \in C} x_c },
\end{equation}
is an equation for $\MT(J_{p^k})$, in the sense of \Cref{def: equation for MT}. In particular, we have $x_a = \frac{ \prod_{c \in C} x_c }{\prod_{b \in B} x_{p^k-b} }$ as functions on $\MT(J_{p^k})$. Applying inductively relations of this form, one can express any variable $x_a$ with $p \mid a$ in terms of the variables $x_i$ with $p \nmid i$ only; we will not need the precise form of the resulting equations.

\begin{remark}\label{rmk: different partitions B C}
    Consider two different partitions $B, C$ and $B', C'$ that differ by swapping an element $b \in B$ with an element $c \in C$. The corresponding functions $f_a, f_a'$ then satisfy
    \[
    \frac{f_a'}{f_a} = \frac{x_{p^k-c}x_c}{x_{p^k-b}x_b};
    \]
    this function is identically equal to $1$ on $\MT(J_m)$, as follows from the inclusion of $\MT(J_m)$ in the general symplectic group (see \cite[Remark 2.4.3]{part2} for more details).
\end{remark}

\begin{theorem}{\label{thm: prime power case Kconn computation}}
    Let $m=p^k$ be a perfect power of an odd prime number $p$, with $k > 1$. Let $J_m/\mathbb{Q}$ be the Jacobian of the smooth projective curve with affine equation $y^2=x^m+1$. Denote by $X_m$ the largest simple factor of $J_m$ as in \Cref{th:Jm-factorization}.
    \begin{enumerate}
        \item For every integer $ 1 \leq a \leq m-1$ divisible by $p$ and any choice of sets $B, C$ as above, the expression $f_a$ (defined in \Cref{eq: additional fa equation form}) is an equation for $\MT(J_m)$.
        \item A set of equations defining $\MT(J_m)$ is given by the disjoint union of a set of defining equations for $\MT(X_m)$ and a single equation $f_a$ of the form \Cref{eq: additional fa equation form}, for every integer $ 1 \leq a \leq m-1$ divisible by $p$.
        \item The connected monodromy field $\Q(\varepsilon_{J_m})$ is the $m$-th cyclotomic field $\Q(\zeta_m)$.
    \end{enumerate}
\end{theorem}
\begin{proof}
    Let $a$ be as in the statement of $(1)$ and let $n \geq 2$ be defined by the equality $v_p(a) = k-(n-1)$.
    We set
    \[
    N_j = N_{\Q(\zeta_{p^k}) / \Q(\zeta_{p^j})} \quad \text{ for } j = 0, \ldots, k, \quad N = N_{\Q(\zeta_{p^n}) / \Q(\zeta_{p^{n-1}})}
    \]
    and denote by the same names the corresponding maps of algebraic tori.
    To prove (1), we check that all characters of the form \Cref{eq: additive MT equation for last pk proof} are in the kernel of %
    \begin{equation}\label{eq: MT as kernel, case pk}
    \begin{tikzcd}
    \mathfrak{E} \colon \bigoplus_{i=1}^{2g}\Z x_i \rar["\cong"] &
    \prod_{n=1}^k \widehat{T}_{\Q(\zeta_{p^n})} \rar["\prod \phi_n^\ast"] &
    \prod_{n=1}^k \widehat{T}_{\Q(\zeta_{p^n})} \rar["\sum N_n^\ast"] &
    \widehat{T}_{\Q(\zeta_{p^k})}.
    \end{tikzcd}
    \end{equation}
    As explained in \Cref{subsec:MTcomp}, elements in the kernel of this map correspond to equations for the Mumford-Tate group, see in particular \eqref{eq: MT as the kernel of a map}.
    By \Cref{lemma: prime power decomposition computation} we can write
    \begin{equation}{\label{eq: applying commutation lemma on fa I}}
        N^\ast \phi_2^\ast(x_a) =
        \left( \sum_{j \text{ s.t. } p \cdot j = a} \phi_1^\ast(x_j) \right)
        - \frac{p-1}{2} \cdot \left( \sum_{(j,p)=1} x_j \right),
    \end{equation}
    where $N$ is the one-step norm map of \Cref{eq: one-step norm map}, while $\phi_1$ and $\phi_2$ are respectively the reflex norms given by the CM types on $\Q(\zeta_{p^{n}})$ and $\Q(\zeta_{p^{n-1}})$ induced by $X_{p^n}$ and $X_{p^{n-1}}$. Notice that the indices in the first sum are exactly the indices in $B \sqcup C$ (in the notation of \Cref{eq: additive MT equation for last pk proof}), while the last sum can be written as $\phi_1^\ast(x_j)+\phi_1^\ast(x_{p^k-j})$ for any index $j$ coprime to $p$ (as in the proof of \Cref{lemma: prime power decomposition computation}). Using the fact that the cardinality of $B$ is exactly $(p-1)/2$, we can rearrange \Cref{eq: applying commutation lemma on fa I} into
    \begin{equation}{\label{eq: applying commutation lemma on fa II}}
    \begin{aligned}
        N^\ast \phi_2^\ast(x_a) & =
        \left( \sum_{b \in B} \phi_1^\ast(x_b) + \sum_{c \in C} \phi_1^\ast(x_c) \right)
        - \left(  \sum_{b \in B} (\phi_1^\ast(x_b) + \phi_1^\ast(x_{p^k-b})) \right) \\
        & =
        \sum_{c \in C} \phi_1^\ast(x_c)
        - \sum_{b \in B} \phi_1^\ast(x_{p^k-b}).
        \end{aligned}
    \end{equation}
    Now notice that any character $\chi$ of the form \Cref{eq: additive MT equation for last pk proof} only involves variables with indices of valuation $v_p(a)$ or $v_p(a)-1$, and as such can be seen as an element of $\widehat{T}_{\Q(\zeta_{p^{n-1}})} \times \widehat{T}_{\Q(\zeta_{p^{n}})}$. To prove $\mathfrak{E}(\chi)=0$ we therefore need to show
    \[
    N_{n-1}^* \phi_2^*(x_a) + N_n^*  \left( \sum_{b \in B} \phi_1^*(x_{p^k-b}) - \sum_{c \in C} \phi_1^*(x_c) \right) = 0.
    \]
    Using the obvious equality
    \[
    N_{n-1} = N \circ N_n \Rightarrow N_{n-1}^* = N_n^* \circ N^*,
    \]
    it suffices to show
    \[
    N^* \phi_2^*(x_a) + \left( \sum_{b \in B} \phi_1^*(x_{p^k-b}) - \sum_{c \in C} \phi_1^*(x_c) \right) = 0,
    \]
    which follows immediately from \eqref{eq: applying commutation lemma on fa II}.
        
    We now prove (2). Denote by $A$ the kernel of \eqref{eq: MT as kernel, case pk}: it is a free $\Z$-submodule of %
    $\bigoplus_i \Z x_i$.
    Let $B$ be the free $\Z$-submodule of $\bigoplus_i \Z x_i$ generated by the equations in the statement.
    Part (1) gives us the inclusion $B \subseteq A$ and therefore an inclusion between the corresponding groups of multiplicative type $ \MT(J_m) = G_A \subseteq G_B$.
    Using the additional equations $f_a$, we can express every variable $x_a$, with $a$ divisible by $p$, in terms of the variables $x_i$, with $i$ coprime to $p$. Therefore, the map from $G_B$ to $\MT(X_m)$ that forgets the variables $x_a$ with $p \mid a$ is an isomorphism. It follows from \Cref{th: MT projection prime power case} that the containment $G_A\subseteq G_B$ is an equality.

    Finally, we prove (3).
    By \Cref{th:nondegenarcy}, the largest factor $X_m$ is nondegenerate, and therefore its connected monodromy field coincides with its endomorphism field, which is $\Q(\zeta_m)$ by \Cref{prop: endomorphism field easy containment}.
    We are thus reduced to computing the connected monodromy field $K(\varepsilon_{J_m})$ over the cyclotomic field $K = \Q(\zeta_m)$.
    We accomplish this by directly computing the Gamma-values $\Gamma(f_a)$ up to elements of $K$, see \Cref{def: gamma value associated to equation for MT} and \Cref{cor: connected monodromy field for odd m in terms of equations}. Note that for every equation $f$ defining $\MT(X_m)$ we have $\Gamma(f) \in \Q(\zeta_m)$, because $\Q(\zeta_m)$ is the connected monodromy field of $X_m$.  
    Fix an index $a = a' \cdot p $ such that $v_p(a)=k-(n-1) > 0$. We need to show that $\Gamma(f_a)$ lies in $K^\times$. 
    We can then work in $\C^\times / K^\times$, and in fact, since we know \textit{a priori} that $\Gamma(f_a)$ is algebraic, we can also work in the quotient $\C^\times / K^\times \langle \pi \rangle$. We will show that $\Gamma(f_a)$ is trivial in this quotient.

    Explicitly, from \Cref{def: gamma value associated to equation for MT} we have 
    \[
    \Gamma(f_a) = \Gamma(a/p^k)^2 \cdot \Gamma([-2a]/p^k) \cdot \frac{\prod_{b \in B} \Gamma([-b]/p^k)^2 \cdot \Gamma([2b]/p^k)}{\prod_{c \in C} \Gamma(c/p^k)^2 \cdot \Gamma([-2c]/p^k)}.
    \]
    We begin by remarking that the functional equation $\Gamma(z+1)=z\Gamma(z)$ implies that $\Gamma(r/p^k) \bmod K^\times$ only depends on $r$ via its residue class modulo $p^k$, provided that this class is non-zero. To compute $\Gamma(f_a) \bmod K^\times \langle \pi \rangle$, we can therefore replace the representatives $[-2a], [-b]$, etc., with $-2a, b$, etc.
        The reflection formula, together with the fact that $\sin(\pi b/p^k) = \frac{\zeta_{2m}^b-\zeta_{2m}^{-b}}{2i} = \frac{(-\zeta_{m})^b-(-\zeta_{m})^{-b}}{2i}  = \frac{1}{i} \bmod K^\times$, gives
    \[
    \Gamma\left( -b/p^k \right) = \frac{i}{\Gamma(b/p^k)} \bmod K^\times \langle \pi \rangle \quad\text{ and }\quad \Gamma\left( 2b/p^k \right) = \frac{i}{\Gamma(-2b/p^k)} \bmod K^\times \langle \pi \rangle.
    \]
    Using this, we can replace each factor $\Gamma(-b/p^k),\, \Gamma(2b/p^k)$ in the numerator of $\Gamma(f_a)$ with a corresponding factor $\Gamma(b/p^k),\, \Gamma(-2b/p^k)$ in the denominator, up to a factor of $i$. Since the total number of factors that we replace in this way is $3|B|= 3\frac{p-1}{2}$, and using the fact that $C \sqcup \{p^k-b : b \in B\}= \{a/p + j \cdot p^{k-1}\}$, this gives
    \[
    \Gamma(f_a) = i^{3(p-1)/2} \cdot \theta_a^2 \cdot \theta_{-2a} \quad\bmod K^\times \langle \pi \rangle,
    \]
    where
    \[
    \theta_a = \frac{\Gamma(\frac{a}{p^k})}{\prod_{j=0}^{p-1} \Gamma\left( \frac{a/p+j \cdot p^{k-1} }{p^k} \right)}  \quad \text{ and } \quad \theta_{-2a} = \frac{\Gamma(\frac{-2a}{p^k})}{\prod_{j=0}^{p-1} \Gamma\left( \frac{(-2a)/p+j \cdot p^{k-1} }{p^k} \right)}.
    \]
    Note that $\theta_{-2a}$ is obtained from $\theta_a$ by replacing $a$ with $-2a$.
    We then use the multiplication formula to rewrite the denominator of $\theta_a$ as
    \[
    \prod_{j=0}^{p-1} \Gamma\left( \frac{a'+j \cdot p^{k-1} }{p^k} \right) =
    \prod_{j=0}^{p-1} \Gamma\left( \frac{a'}{p^{k}}+\frac{j}{p} \right) = p^{\frac{1}{2}-\frac{a'}{p^k}\cdot p} \cdot \Gamma\left(\frac{a}{p^k}\right) \quad\bmod K^\times \langle \pi \rangle,
    \]
    which gives $\theta_a = p^{-\frac{1}{2} + \frac{a}{p^k} }$ and similarly $\theta_{-2a} = p^{-\frac{1}{2} - \frac{2a}{p^k} }$ (both equalities modulo $K^\times \langle \pi \rangle$). Thus,
    \[
    \Gamma(f_a) = i^{3\frac{p-1}{2}} \cdot \theta_a^2 \cdot \theta_{-2a} = i^{\frac{p-1}{2}} p^{-1 + \frac{2a}{p^k} } p^{-\frac{1}{2} - \frac{2a}{p^k} } = \sqrt{(-1)^{\tfrac{p-1}{2}} p} \quad\bmod K^\times \langle \pi \rangle.
    \]
    It is an elementary fact in number theory that the prime cyclotomic field $\Q(\zeta_p) \subseteq K$ contains $\sqrt{(-1)^{(p-1)/2} p}$, which concludes the proof.

\end{proof}

\bibliographystyle{alpha}
\bibliography{biblio}

@article {Barnet2011,
    AUTHOR = {Barnet-Lamb, Tom and Geraghty, David and Harris, Michael and
              Taylor, Richard},
     TITLE = {A family of {C}alabi-{Y}au varieties and potential automorphy
              {II}},
   JOURNAL = {Publ. Res. Inst. Math. Sci.},
  FJOURNAL = {Publications of the Research Institute for Mathematical
              Sciences},
    VOLUME = {47},
      YEAR = {2011},
    NUMBER = {1},
     PAGES = {29--98},
      ISSN = {0034-5318},
   MRCLASS = {11F80 (11F11 11G18 14J32)},
  MRNUMBER = {2827723},
MRREVIEWER = {Neil P. Dummigan},
       DOI = {10.2977/PRIMS/31},
       URL = {http://dx.doi.org/10.2977/PRIMS/31},
}

@book {BirkenhakeLange2004,
    AUTHOR = {Birkenhake, Christina and Lange, Herbert},
     TITLE = {Complex abelian varieties},
    SERIES = {Grundlehren der mathematischen Wissenschaften [Fundamental
              Principles of Mathematical Sciences]},
    VOLUME = {302},
   EDITION = {Second},
 PUBLISHER = {Springer-Verlag, Berlin},
      YEAR = {2004},
     PAGES = {xii+635},
      ISBN = {3-540-20488-1},
   MRCLASS = {14-02 (14H37 14Kxx 32G20)},
  MRNUMBER = {2062673},
MRREVIEWER = {Fumio Hazama},
       DOI = {10.1007/978-3-662-06307-1},
       URL = {https://doi.org/10.1007/978-3-662-06307-1},
}

@article {Clozel2008,
    AUTHOR = {Clozel, Laurent and Harris, Michael and Taylor, Richard},
     TITLE = {Automorphy for some {$l$}-adic lifts of automorphic mod {$l$}
              {G}alois representations},
      NOTE = {With Appendix A, summarizing unpublished work of Russ Mann,
              and Appendix B by Marie-France Vign\'{e}ras},
   JOURNAL = {Publ. Math. Inst. Hautes \'{E}tudes Sci.},
  FJOURNAL = {Publications Math\'{e}matiques. Institut de Hautes \'{E}tudes
              Scientifiques},
    NUMBER = {108},
      YEAR = {2008},
     PAGES = {1--181},
      ISSN = {0073-8301},
   MRCLASS = {11F80 (11G18 11R34)},
  MRNUMBER = {2470687},
MRREVIEWER = {Mark Kisin},
       DOI = {10.1007/s10240-008-0016-1},
       URL = {https://doi-org.ez-proxy.brooklyn.cuny.edu/10.1007/s10240-008-0016-1},
}

@article {Harris2010,
    AUTHOR = {Harris, Michael and Shepherd-Barron, Nick and Taylor, Richard},
     TITLE = {A family of {C}alabi-{Y}au varieties and potential automorphy},
   JOURNAL = {Ann. of Math. (2)},
  FJOURNAL = {Annals of Mathematics. Second Series},
    VOLUME = {171},
      YEAR = {2010},
    NUMBER = {2},
     PAGES = {779--813},
      ISSN = {0003-486X},
     CODEN = {ANMAAH},
   MRCLASS = {11F80 (11G18 11G35 11R34 14J32)},
  MRNUMBER = {2630056},
MRREVIEWER = {Neil P. Dummigan},
       DOI = {10.4007/annals.2010.171.779},
       URL = {http://dx.doi.org/10.4007/annals.2010.171.779},
}

@article {Taylor2008,
    AUTHOR = {Taylor, Richard},
     TITLE = {Automorphy for some {$l$}-adic lifts of automorphic mod {$l$}
              {G}alois representations. {II}},
   JOURNAL = {Publ. Math. Inst. Hautes \'{E}tudes Sci.},
  FJOURNAL = {Publications Math\'{e}matiques. Institut de Hautes \'{E}tudes
              Scientifiques},
    NUMBER = {108},
      YEAR = {2008},
     PAGES = {183--239},
      ISSN = {0073-8301},
   MRCLASS = {11F80 (11G18 11R34)},
  MRNUMBER = {2470688},
MRREVIEWER = {Mark Kisin},
       DOI = {10.1007/s10240-008-0015-2},
       URL = {https://doi-org.ez-proxy.brooklyn.cuny.edu/10.1007/s10240-008-0015-2},
}

@article {Joh17,
    AUTHOR = {Johansson, Christian},
     TITLE = {On the {S}ato-{T}ate conjecture for non-generic abelian
              surfaces},
      NOTE = {With an appendix by Francesc Fit\'{e}},
   JOURNAL = {Trans. Amer. Math. Soc.},
  FJOURNAL = {Transactions of the American Mathematical Society},
    VOLUME = {369},
      YEAR = {2017},
    NUMBER = {9},
     PAGES = {6303--6325},
      ISSN = {0002-9947},
   MRCLASS = {11F80 (11G10)},
  MRNUMBER = {3660222},
MRREVIEWER = {Cameron Franc},
       DOI = {10.1090/tran/6847},
       URL = {https://doi.org/10.1090/tran/6847},
}

@ARTICLE{Arora2016,
    AUTHOR = {Arora, Sonny and Cantoral-Farf\'{a}n, Victoria and Landesman,
              Aaron and Lombardo, Davide and Morrow, Jackson S.},
     TITLE = {The twisting {S}ato-{T}ate group of the curve
              {$y^2=x^8-14x^4+1$}},
   JOURNAL = {Math. Z.},
  FJOURNAL = {Mathematische Zeitschrift},
    VOLUME = {290},
      YEAR = {2018},
    NUMBER = {3-4},
     PAGES = {991--1022},
      ISSN = {0025-5874},
   MRCLASS = {11G05 (14H25)},
  MRNUMBER = {3856841},
       DOI = {10.1007/s00209-018-2049-6},
       URL = {https://doi.org/10.1007/s00209-018-2049-6},
}

@article{Heidi,
  doi = {10.48550/ARXIV.2211.03909},
  url = {https://arxiv.org/abs/2211.03909},
  author = {Goodson, Heidi},
  title = {{An Exploration of Degeneracy in Abelian Varieties of Fermat Type}},
   JOURNAL = {Experimental Mathematics},
  FJOURNAL = {Experimental Mathematics},
    MONTH = {June},
      YEAR = {2024},
     PAGES = {1--17}
}

@ARTICLE{GoodsonCatalan,
    AUTHOR = {Goodson, Heidi},
     TITLE = {Sato-{T}ate distributions of {C}atalan curves},
   JOURNAL = {J. Th\'{e}or. Nombres Bordeaux},
  FJOURNAL = {Journal de Th\'{e}orie des Nombres de Bordeaux},
    VOLUME = {35},
      YEAR = {2023},
    NUMBER = {1},
     PAGES = {87--113},
      ISSN = {1246-7405,2118-8572},
   MRCLASS = {11G20 (11G10 11M50 14G10)},
  MRNUMBER = {4596524},
}

@ARTICLE{GoodsonHoque2024,
    AUTHOR = {Goodson, Heidi and Hoque, Rezwan},
     TITLE = {Sato--{T}ate groups and distributions of {$y^\ell =x(x^\ell
              -1)$}},
   JOURNAL = {Acta Arith.},
  FJOURNAL = {Acta Arithmetica},
    VOLUME = {221},
      YEAR = {2025},
    NUMBER = {4},
     PAGES = {355--370},
      ISSN = {0065-1036,1730-6264},
   MRCLASS = {11G10 (11G20 14G10)},
  MRNUMBER = {5078389},
       DOI = {10.4064/aa241212-22-8},
       URL = {https://doi.org/10.4064/aa241212-22-8},
}

@article {FGL2016,
    AUTHOR = {Fit\'e, Francesc and Gonz\'alez, Josep and Lario, Joan-Carles},
     TITLE = {Frobenius distribution for quotients of {F}ermat curves of
              prime exponent},
   JOURNAL = {Canad. J. Math.},
  FJOURNAL = {Canadian Journal of Mathematics. Journal Canadien de
              Math\'ematiques},
    VOLUME = {68},
      YEAR = {2016},
    NUMBER = {2},
     PAGES = {361--394},
      ISSN = {0008-414X},
   MRCLASS = {11D41 (11G10 11M50 14G10)},
  MRNUMBER = {3484371},
MRREVIEWER = {A. Peth\H o},
       URL = {https://doi.org/10.4153/CJM-2015-028-x},
}

@article {FiteLorenzo2018,
    AUTHOR = {Fit\'{e}, Francesc and Lorenzo Garc\'{\i}a, Elisa and Sutherland,
              Andrew V.},
     TITLE = {Sato-{T}ate distributions of twists of the {F}ermat and the
              {K}lein quartics},
   JOURNAL = {Res. Math. Sci.},
  FJOURNAL = {Research in the Mathematical Sciences},
    VOLUME = {5},
      YEAR = {2018},
    NUMBER = {4},
     PAGES = {Paper No. 41, 40},
      ISSN = {2522-0144},
   MRCLASS = {14G10 (11G10 11G40 14J28 14K15)},
  MRNUMBER = {3864839},
MRREVIEWER = {John M. Voight},
       DOI = {10.1007/s40687-018-0162-0},
       URL = {https://doi.org/10.1007/s40687-018-0162-0},
}

@article {Fite2012,
    AUTHOR = {Fit{\'e}, Francesc and Kedlaya, Kiran S. and Rotger,
              V{\'{\i}}ctor and Sutherland, Andrew V.},
     TITLE = {Sato-{T}ate distributions and {G}alois endomorphism modules in
              genus 2},
   JOURNAL = {Compos. Math.},
  FJOURNAL = {Compositio Mathematica},
    VOLUME = {148},
      YEAR = {2012},
    NUMBER = {5},
     PAGES = {1390--1442},
      ISSN = {0010-437X},
   MRCLASS = {11M50 (11G10 11G20 14G10 14K15)},
  MRNUMBER = {2982436},
MRREVIEWER = {Imin Chen},
       DOI = {10.1112/S0010437X12000279},
       URL = {http://dx.doi.org/10.1112/S0010437X12000279},
}

@book {Shimura,
	AUTHOR = {Shimura, Goro and Taniyama, Yutaka},
	TITLE = {Complex multiplication of abelian varieties and its
	applications to number theory},
	SERIES = {Publications of the Mathematical Society of Japan},
	VOLUME = {6},
	PUBLISHER = {Mathematical Society of Japan, Tokyo},
	YEAR = {1961},
	PAGES = {xi+159},
	MRCLASS = {14.40 (10.68)},
	MRNUMBER = {0125113},
	MRREVIEWER = {I. Barsotti},
}

@book {Lang,
	AUTHOR = {Lang, Serge},
	TITLE = {Complex multiplication},
	SERIES = {Grundlehren der mathematischen Wissenschaften [Fundamental
	Principles of Mathematical Sciences]},
	VOLUME = {255},
	PUBLISHER = {Springer-Verlag, New York},
	YEAR = {1983},
	PAGES = {viii+184},
	ISBN = {0-387-90786-6},
	MRCLASS = {11G15 (11F41 11G18 14K22)},
	MRNUMBER = {713612},
	MRREVIEWER = {James Milne},
	DOI = {10.1007/978-1-4612-5485-0},
	URL = {https://doi.org/10.1007/978-1-4612-5485-0},
}

@article {Gannon1996,
    AUTHOR = {Gannon, Terry},
     TITLE = {The classification of {${\rm SU}(3)$} modular invariants
              revisited},
   JOURNAL = {Ann. Inst. H. Poincar\'{e} Phys. Th\'{e}or.},
  FJOURNAL = {Annales de l'Institut Henri Poincar\'{e}. Physique Th\'{e}orique},
    VOLUME = {65},
      YEAR = {1996},
    NUMBER = {1},
     PAGES = {15--55},
      ISSN = {0246-0211},
   MRCLASS = {81T40 (11F11 11Z05 81R10)},
  MRNUMBER = {1407165},
MRREVIEWER = {Mark A. Walton},
       URL = {http://www.numdam.org/item?id=AIHPA_1996__65_1_15_0},
}

@article {Kubota1965,
    AUTHOR = {Kubota, Tomio},
     TITLE = {On the field extension by complex multiplication},
   JOURNAL = {Trans. Amer. Math. Soc.},
  FJOURNAL = {Transactions of the American Mathematical Society},
    VOLUME = {118},
      YEAR = {1965},
     PAGES = {113--122},
      ISSN = {0002-9947},
   MRCLASS = {10.68 (14.51)},
  MRNUMBER = {190144},
       DOI = {10.2307/1993947},
       URL = {https://doi.org/10.2307/1993947},
}

@article {MR1150604,
    AUTHOR = {Larsen, Michael and Pink, Richard},
     TITLE = {On {$l$}-independence of algebraic monodromy groups in
              compatible systems of representations},
   JOURNAL = {Invent. Math.},
  FJOURNAL = {Inventiones Mathematicae},
    VOLUME = {107},
      YEAR = {1992},
    NUMBER = {3},
     PAGES = {603--636},
      ISSN = {0020-9910},
   MRCLASS = {22E50 (11F80 20G25)},
  MRNUMBER = {1150604},
MRREVIEWER = {Jean-Yves \'{E}tesse},
       DOI = {10.1007/BF01231904},
       URL = {https://doi.org/10.1007/BF01231904},
}

@misc{zywina2019effective,
      title={An effective open image theorem for abelian varieties}, 
      author={Zywina, David},
      year={2019},
      eprint={1910.14171},
      archivePrefix={arXiv},
      primaryClass={math.NT}
}

@article {MR1355128,
    AUTHOR = {Silverberg, Alice and Zarhin, Yuri G.},
     TITLE = {Connectedness results for {$l$}-adic representations associated to abelian varieties},
      NOTE = {Special issue in honour of Frans Oort},
   JOURNAL = {Compositio Math.},
  FJOURNAL = {Compositio Mathematica},
    VOLUME = {97},
      YEAR = {1995},
    NUMBER = {1-2},
     PAGES = {273--284},
      ISSN = {0010-437X},
   MRCLASS = {11G10 (14K05)},
  MRNUMBER = {1355128},
       URL = {http://www.numdam.org/item?id=CM_1995__97_1-2_273_0},
}

@article {MR3320526,
    AUTHOR = {Banaszak, Grzegorz and Kedlaya, Kiran S.},
     TITLE = {An algebraic {S}ato-{T}ate group and {S}ato-{T}ate conjecture},
   JOURNAL = {Indiana Univ. Math. J.},
  FJOURNAL = {Indiana University Mathematics Journal},
    VOLUME = {64},
      YEAR = {2015},
    NUMBER = {1},
     PAGES = {245--274},
      ISSN = {0022-2518},
   MRCLASS = {11G10 (14G25 14K15)},
  MRNUMBER = {3320526},
MRREVIEWER = {Ivica Gusi\'{c}},
       DOI = {10.1512/iumj.2015.64.5438},
       URL = {https://doi.org/10.1512/iumj.2015.64.5438},
}

@article {MR1630512,
    AUTHOR = {Silverberg, Alice and Zarhin, Yuri G.},
     TITLE = {Connectedness extensions for abelian varieties},
   JOURNAL = {Math. Z.},
  FJOURNAL = {Mathematische Zeitschrift},
    VOLUME = {228},
      YEAR = {1998},
    NUMBER = {2},
     PAGES = {387--403},
      ISSN = {0025-5874},
   MRCLASS = {11G10 (14K15)},
  MRNUMBER = {1630512},
MRREVIEWER = {Ben Moonen},
       DOI = {10.1007/PL00004624},
       URL = {https://doi.org/10.1007/PL00004624},
}

@misc{cantoralfarfan2023monodromy,
      title={{Monodromy groups of {J}acobians with definite quaternionic multiplication}}, 
      author={Victoria Cantoral-Farfán and Davide Lombardo and John Voight},
      year={2023},
      eprint={2303.00804},
      archivePrefix={arXiv},
      primaryClass={math.NT}
}

@article {MR4557876,
    AUTHOR = {Lombardo, Davide},
     TITLE = {Non-isogenous abelian varieties sharing the same division
              fields},
   JOURNAL = {Trans. Amer. Math. Soc.},
  FJOURNAL = {Transactions of the American Mathematical Society},
    VOLUME = {376},
      YEAR = {2023},
    NUMBER = {4},
     PAGES = {2615--2640},
      ISSN = {0002-9947},
   MRCLASS = {14K22 (11F80 11G10 14K15)},
  MRNUMBER = {4557876},
       DOI = {10.1090/tran/8767},
       URL = {https://doi.org/10.1090/tran/8767},
}

@article {Shioda3,
    AUTHOR = {Shioda, Tetsuji},
     TITLE = {Algebraic cycles on abelian varieties of {F}ermat type},
   JOURNAL = {Math. Ann.},
  FJOURNAL = {Mathematische Annalen},
    VOLUME = {258},
      YEAR = {1981/82},
    NUMBER = {1},
     PAGES = {65--80},
      ISSN = {0025-5831},
   MRCLASS = {14K20 (14C30)},
  MRNUMBER = {641669},
MRREVIEWER = {Yu. G. Zarkhin},
       DOI = {10.1007/BF01450347},
       URL = {https://doi.org/10.1007/BF01450347},
}

@article {Shioda2,
    AUTHOR = {Shioda, Tetsuji},
     TITLE = {The {H}odge conjecture for {F}ermat varieties},
   JOURNAL = {Math. Ann.},
  FJOURNAL = {Mathematische Annalen},
    VOLUME = {245},
      YEAR = {1979},
    NUMBER = {2},
     PAGES = {175--184},
      ISSN = {0025-5831},
   MRCLASS = {14G13},
  MRNUMBER = {552586},
MRREVIEWER = {Loren D. Olson},
       DOI = {10.1007/BF01428804},
       URL = {https://doi.org/10.1007/BF01428804},
}

@article {Shioda1,
    AUTHOR = {Shioda, Tetsuji and Katsura, Toshiyuki},
     TITLE = {On {F}ermat varieties},
   JOURNAL = {Tohoku Math. J. (2)},
  FJOURNAL = {The Tohoku Mathematical Journal. Second Series},
    VOLUME = {31},
      YEAR = {1979},
    NUMBER = {1},
     PAGES = {97--115},
      ISSN = {0040-8735},
   MRCLASS = {14J25 (14G05)},
  MRNUMBER = {526513},
MRREVIEWER = {V. V. Shokurov},
       DOI = {10.2748/tmj/1178229881},
       URL = {https://doi.org/10.2748/tmj/1178229881},
}

@article {Pohlmann,
    AUTHOR = {Pohlmann, Henry},
     TITLE = {Algebraic cycles on abelian varieties of complex
              multiplication type},
   JOURNAL = {Ann. of Math. (2)},
  FJOURNAL = {Annals of Mathematics. Second Series},
    VOLUME = {88},
      YEAR = {1968},
     PAGES = {161--180},
      ISSN = {0003-486X},
   MRCLASS = {14.48},
  MRNUMBER = {228500},
MRREVIEWER = {M. J. Greenberg},
       DOI = {10.2307/1970570},
       URL = {https://doi.org/10.2307/1970570},
}

@article {KedlayaAlg,
    AUTHOR = {Kedlaya, Kiran S.},
     TITLE = {Counting points on hyperelliptic curves using
              {M}onsky-{W}ashnitzer cohomology},
   JOURNAL = {J. Ramanujan Math. Soc.},
  FJOURNAL = {Journal of the Ramanujan Mathematical Society},
    VOLUME = {16},
      YEAR = {2001},
    NUMBER = {4},
     PAGES = {323--338},
      ISSN = {0970-1249},
   MRCLASS = {14G05 (11G25 14F43 14G15)},
  MRNUMBER = {1877805},
MRREVIEWER = {Miriam D. Abd\'{o}n},
}

@book {Silverman1,
    AUTHOR = {Silverman, Joseph H.},
     TITLE = {The arithmetic of elliptic curves},
    SERIES = {Graduate Texts in Mathematics},
    VOLUME = {106},
   EDITION = {Second},
 PUBLISHER = {Springer, Dordrecht},
      YEAR = {2009},
     PAGES = {xx+513},
      ISBN = {978-0-387-09493-9},
   MRCLASS = {11-02 (11G05 11G20 14H52 14K15)},
  MRNUMBER = {2514094},
MRREVIEWER = {Vasil\cprime  \={I}. Andr\={\i}\u{\i}chuk},
       DOI = {10.1007/978-0-387-09494-6},
       URL = {https://doi.org/10.1007/978-0-387-09494-6},
}

@Inbook{Deligne,
author="Deligne, Pierre",
title="{Hodge Cycles on Abelian Varieties}",
bookTitle="{Hodge Cycles, Motives, and Shimura Varieties}",
year="1982",
publisher="Springer Berlin Heidelberg",
address="Berlin, Heidelberg",
pages="9--100",
isbn="978-3-540-38955-2",
doi="10.1007/978-3-540-38955-2_3",
url="https://doi.org/10.1007/978-3-540-38955-2_3"
}

@book {MilneEtaleCo,
    AUTHOR = {Milne, James S.},
     TITLE = {\'{E}tale cohomology},
    SERIES = {Princeton Mathematical Series, No. 33},
 PUBLISHER = {Princeton University Press, Princeton, N.J.},
      YEAR = {1980},
     PAGES = {xiii+323},
      ISBN = {0-691-08238-3},
   MRCLASS = {14-02 (14F20 18F99)},
  MRNUMBER = {559531},
MRREVIEWER = {G. Horrocks},
}

@misc{moonen,
author = {Moonen, Ben},
year = {2004},
pages = {},
title = {{An introduction to Mumford-Tate groups}},
howpublished = {Available at \url{https://www.math.ru.nl/~bmoonen/Lecturenotes/MTGps.pdf}},
}

@misc{moonen2,
author = {Moonen, Ben},
year = {1999},
pages = {},
title = {Notes on {M}umford-{T}ate groups},
howpublished = {Available at \url{https://www.math.ru.nl/~bmoonen/Lecturenotes/CEBnotesMT.pdf}},
}

@incollection {KoblitzOgus,
    AUTHOR = {Deligne, Pierre},
     TITLE = {Valeurs de fonctions {$L$} et p\'{e}riodes d'int\'{e}grales.},
 BOOKTITLE = {Automorphic forms, representations and {$L$}-functions
              ({P}roc. {S}ympos. {P}ure {M}ath., {O}regon {S}tate {U}niv.,
              {C}orvallis, {O}re., 1977), {P}art 2},
    SERIES = {Proc. Sympos. Pure Math., XXXIII},
     PAGES = {313--346},
      NOTE = {With an appendix by N. Koblitz and A. Ogus.},
      YEAR = {1979},
   MRCLASS = {12A70 (10D15 10D24 10H10)},
  MRNUMBER = {546622},
MRREVIEWER = {J.\ S.\ Milne},
}

@book {tori,
    AUTHOR = {Cox, David A. and Little, John B. and Schenck, Henry K.},
     TITLE = {Toric varieties},
    SERIES = {Graduate Studies in Mathematics},
    VOLUME = {124},
 PUBLISHER = {American Mathematical Society, Providence, RI},
      YEAR = {2011},
     PAGES = {xxiv+841},
      ISBN = {978-0-8218-4819-7},
   MRCLASS = {14M25 (05A15 05E45 52B12)},
  MRNUMBER = {2810322},
MRREVIEWER = {Ivan Arzhantsev},
       DOI = {10.1090/gsm/124},
       URL = {https://doi.org/10.1090/gsm/124},
}

@article {MR601520,
    AUTHOR = {Deligne, Pierre},
     TITLE = {La conjecture de {W}eil. {II}},
   JOURNAL = {Inst. Hautes \'{E}tudes Sci. Publ. Math.},
  FJOURNAL = {Institut des Hautes \'{E}tudes Scientifiques. Publications
              Math\'{e}matiques},
    NUMBER = {52},
      YEAR = {1980},
     PAGES = {137--252},
      ISSN = {0073-8301},
   MRCLASS = {14G13 (10H10)},
  MRNUMBER = {601520},
MRREVIEWER = {Spencer J. Bloch},
       URL = {http://www.numdam.org/item?id=PMIHES_1980__52__137_0},
}

@article {MR0545172,
    AUTHOR = {Kubert, Daniel S.},
     TITLE = {The {${\bf Z}/2{\bf Z}$} cohomology of the universal ordinary
              distribution},
   JOURNAL = {Bull. Soc. Math. France},
  FJOURNAL = {Bulletin de la Soci\'{e}t\'{e} Math\'{e}matique de France},
    VOLUME = {107},
      YEAR = {1979},
    NUMBER = {2},
     PAGES = {203--224},
      ISSN = {0037-9484},
   MRCLASS = {20K40 (12A45)},
  MRNUMBER = {545172},
MRREVIEWER = {Neal\ Koblitz},
       URL = {http://www.numdam.org/item?id=BSMF_1979__107__203_0},
}

@article {MR1871965,
    AUTHOR = {Bae, Sunghan and Gekeler, Ernst-U. and Yin, Linsheng},
     TITLE = {Distributions and {$\Gamma$}-monomials},
   JOURNAL = {Math. Ann.},
  FJOURNAL = {Mathematische Annalen},
    VOLUME = {321},
      YEAR = {2001},
    NUMBER = {3},
     PAGES = {463--478},
      ISSN = {0025-5831,1432-1807},
   MRCLASS = {33B15 (11R18 11R60 33E50)},
  MRNUMBER = {1871965},
MRREVIEWER = {David\ Goss},
       DOI = {10.1007/s002080100200},
       URL = {https://doi.org/10.1007/s002080100200},
}

@article {MR1638625,
    AUTHOR = {Das, Pinaki},
     TITLE = {Algebraic gamma monomials and double coverings of cyclotomic
              fields},
   JOURNAL = {Trans. Amer. Math. Soc.},
  FJOURNAL = {Transactions of the American Mathematical Society},
    VOLUME = {352},
      YEAR = {2000},
    NUMBER = {8},
     PAGES = {3557--3594},
      ISSN = {0002-9947,1088-6850},
   MRCLASS = {11R18 (11R32)},
  MRNUMBER = {1638625},
MRREVIEWER = {Lawrence\ Washington},
       DOI = {10.1090/S0002-9947-00-02417-X},
       URL = {https://doi.org/10.1090/S0002-9947-00-02417-X},
}

@article {MR3782449,
    AUTHOR = {K\"{o}ck, Bernhard and Tait, Joseph},
     TITLE = {On the de-{R}ham cohomology of hyperelliptic curves},
   JOURNAL = {Res. Number Theory},
  FJOURNAL = {Research in Number Theory},
    VOLUME = {4},
      YEAR = {2018},
    NUMBER = {2},
     PAGES = {Paper No. 19, 17},
      ISSN = {2522-0160},
   MRCLASS = {14F40 (14G17 14H37)},
  MRNUMBER = {3782449},
MRREVIEWER = {Nobuo Tsuzuki},
       DOI = {10.1007/s40993-018-0111-4},
       URL = {https://doi.org/10.1007/s40993-018-0111-4},
}

@book {Harari,
    AUTHOR = {Harari, David},
     TITLE = {Galois cohomology and class field theory},
    SERIES = {Universitext},
      NOTE = {Translated from the 2017 French original by Andrei Yafaev},
 PUBLISHER = {Springer, Cham},
      YEAR = {[2020] \copyright 2020},
     PAGES = {xiv+338},
      ISBN = {978-3-030-43901-9; 978-3-030-43900-2},
   MRCLASS = {11R34 (11R37 12G05)},
  MRNUMBER = {4174395},
MRREVIEWER = {Yasushi Mizusawa},
       DOI = {10.1007/978-3-030-43901-9},
       URL = {https://doi.org/10.1007/978-3-030-43901-9},
}

@incollection {MR3502937,
    AUTHOR = {Banaszak, Grzegorz and Kedlaya, Kiran S.},
     TITLE = {Motivic {S}erre group, algebraic {S}ato-{T}ate group and
              {S}ato-{T}ate conjecture},
 BOOKTITLE = {Frobenius distributions: {L}ang-{T}rotter and {S}ato-{T}ate
              conjectures},
    SERIES = {Contemp. Math.},
    VOLUME = {663},
     PAGES = {11--44},
 PUBLISHER = {Amer. Math. Soc., Providence, RI},
      YEAR = {2016},
      ISBN = {978-1-4704-1947-9},
   MRCLASS = {14C30 (11G35 14C15)},
  MRNUMBER = {3502937},
MRREVIEWER = {Annette\ Huber},
       DOI = {10.1090/conm/663/13348},
       URL = {https://doi.org/10.1090/conm/663/13348},
}

@article {MR3494170,
    AUTHOR = {Lombardo, Davide},
     TITLE = {On the {$\ell$}-adic {G}alois representations attached to
              nonsimple abelian varieties},
   JOURNAL = {Ann. Inst. Fourier (Grenoble)},
  FJOURNAL = {Universit\'{e} de Grenoble. Annales de l'Institut Fourier},
    VOLUME = {66},
      YEAR = {2016},
    NUMBER = {3},
     PAGES = {1217--1245},
      ISSN = {0373-0956,1777-5310},
   MRCLASS = {11G10 (11F80 14K15)},
  MRNUMBER = {3494170},
MRREVIEWER = {Jan\ Nekov\'{a}\v{r}},
       URL = {http://aif.cedram.org/item?id=AIF_2016__66_3_1217_0},
}

@article {MR4530050,
    AUTHOR = {Cantoral-Farf\'{a}n, Victoria and Commelin, Johan},
     TITLE = {The {M}umford-{T}ate conjecture implies the algebraic
              {S}ato-{T}ate conjecture of {B}anaszak and {K}edlaya},
   JOURNAL = {Indiana Univ. Math. J.},
  FJOURNAL = {Indiana University Mathematics Journal},
    VOLUME = {71},
      YEAR = {2022},
    NUMBER = {6},
     PAGES = {2595--2603},
      ISSN = {0022-2518,1943-5258},
   MRCLASS = {14K15 (11G10 14F42)},
  MRNUMBER = {4530050},
}

@article {MR0931215,
    AUTHOR = {Coleman, Robert and McCallum, William},
     TITLE = {Stable reduction of {F}ermat curves and {J}acobi sum {H}ecke
              characters},
   JOURNAL = {J. Reine Angew. Math.},
  FJOURNAL = {Journal f\"{u}r die Reine und Angewandte Mathematik. [Crelle's
              Journal]},
    VOLUME = {385},
      YEAR = {1988},
     PAGES = {41--101},
      ISSN = {0075-4102,1435-5345},
   MRCLASS = {11G30 (11L10 11R18 14G25 14H25)},
  MRNUMBER = {931215},
MRREVIEWER = {Joe\ P.\ Buhler},
}

@article {Banaszak2003,
    AUTHOR = {Banaszak, Grzegorz and Gajda, Wojciech and Kraso\'{n}, Piotr},
     TITLE = {On {G}alois representations for abelian varieties with complex
              and real multiplications},
   JOURNAL = {J. Number Theory},
  FJOURNAL = {Journal of Number Theory},
    VOLUME = {100},
      YEAR = {2003},
    NUMBER = {1},
     PAGES = {117--132},
      ISSN = {0022-314X},
   MRCLASS = {11F80 (11G10 11G15)},
  MRNUMBER = {1971250},
MRREVIEWER = {Thomas A. Weston},
       DOI = {10.1016/S0022-314X(02)00121-X},
       URL = {https://doi.org/10.1016/S0022-314X(02)00121-X},
}

@article {Hazama89,
    AUTHOR = {Hazama, Fumio},
     TITLE = {Algebraic cycles on nonsimple abelian varieties},
   JOURNAL = {Duke Math. J.},
  FJOURNAL = {Duke Mathematical Journal},
    VOLUME = {58},
      YEAR = {1989},
    NUMBER = {1},
     PAGES = {31--37},
      ISSN = {0012-7094},
   MRCLASS = {14C30 (14K05)},
  MRNUMBER = {1016412},
       DOI = {10.1215/S0012-7094-89-05803-1},
       URL = {https://doi.org/10.1215/S0012-7094-89-05803-1},
}

@article {MR3882288,
    AUTHOR = {Lombardo, Davide},
     TITLE = {Computing the geometric endomorphism ring of a genus-2
              {J}acobian},
   JOURNAL = {Math. Comp.},
  FJOURNAL = {Mathematics of Computation},
    VOLUME = {88},
      YEAR = {2019},
    NUMBER = {316},
     PAGES = {889--929},
      ISSN = {0025-5718,1088-6842},
   MRCLASS = {11F80 (11G10 11Y99)},
  MRNUMBER = {3882288},
MRREVIEWER = {John\ L.\ Boxall},
       DOI = {10.1090/mcom/3358},
       URL = {https://doi.org/10.1090/mcom/3358},
}

@article {MR3904148,
    AUTHOR = {Costa, Edgar and Mascot, Nicolas and Sijsling, Jeroen and
              Voight, John},
     TITLE = {Rigorous computation of the endomorphism ring of a {J}acobian},
   JOURNAL = {Math. Comp.},
  FJOURNAL = {Mathematics of Computation},
    VOLUME = {88},
      YEAR = {2019},
    NUMBER = {317},
     PAGES = {1303--1339},
      ISSN = {0025-5718,1088-6842},
   MRCLASS = {11G10 (11Y16 14H40 14K15 14Q05)},
  MRNUMBER = {3904148},
MRREVIEWER = {Matthew\ Bisatt},
       DOI = {10.1090/mcom/3373},
       URL = {https://doi.org/10.1090/mcom/3373},
}

@article {MR4280568,
    AUTHOR = {Costa, Edgar and Lombardo, Davide and Voight, John},
     TITLE = {Identifying central endomorphisms of an abelian variety via
              {F}robenius endomorphisms},
   JOURNAL = {Res. Number Theory},
  FJOURNAL = {Research in Number Theory},
    VOLUME = {7},
      YEAR = {2021},
    NUMBER = {3},
     PAGES = {Paper No. 46, 14},
      ISSN = {2522-0160,2363-9555},
   MRCLASS = {11G10 (11Y40 14K15)},
  MRNUMBER = {4280568},
       DOI = {10.1007/s40993-021-00264-y},
       URL = {https://doi.org/10.1007/s40993-021-00264-y},
}

@misc{lmfdb,
  shorthand    = {LMFDB},
  author       = {The {LMFDB Collaboration}},
  title        = {The {L}-functions and modular forms database},
  howpublished = {\url{https://www.lmfdb.org}},
  year         = {2023},
  note         = {[Online; accessed 9 October 2023]},
}

@article {Heidi2,
    AUTHOR = {Emory, Melissa and Goodson, Heidi and Peyrot, Alexandre},
     TITLE = {Towards the {S}ato-{T}ate groups of trinomial hyperelliptic
              curves},
   JOURNAL = {Int. J. Number Theory},
  FJOURNAL = {International Journal of Number Theory},
    VOLUME = {17},
      YEAR = {2021},
    NUMBER = {10},
     PAGES = {2175--2206},
      ISSN = {1793-0421},
   MRCLASS = {11G30 (11G10)},
  MRNUMBER = {4322828},
MRREVIEWER = {Davide Lombardo},
       DOI = {10.1142/S1793042121500822},
       URL = {https://doi.org/10.1142/S1793042121500822},
}

@misc{fitekedlayasuth2023,
      title={{S}ato-{T}ate groups of abelian threefolds}, 
      author={Francesc Fité and Kiran S. Kedlaya and Andrew V. Sutherland},
      year={2023},
      eprint={2106.13759},
      archivePrefix={arXiv},
      primaryClass={math.NT}
}

@ARTICLE{EmoryGoodson2022,
AUTHOR = {Emory, Melissa and Goodson, Heidi},
     TITLE = {Sato-{T}ate distributions of {$y^2 = x^p - 1$} and {$y^2 =
              x^{2p}-1$}},
   JOURNAL = {J. Algebra},
  FJOURNAL = {Journal of Algebra},
    VOLUME = {597},
      YEAR = {2022},
     PAGES = {241--265},
      ISSN = {0021-8693},
   MRCLASS = {11M50 (11G10 11G20 14G10)},
  MRNUMBER = {4406397},
       DOI = {10.1016/j.jalgebra.2022.01.002},
       URL = {https://doi.org/10.1016/j.jalgebra.2022.01.002},
}

@article {EmoryGoodson2024,
    AUTHOR = {Emory, Melissa and Goodson, Heidi},
     TITLE = {Nondegeneracy and {S}ato-{T}ate distributions of two families
              of {J}acobian varieties},
   JOURNAL = {Res. Number Theory},
  FJOURNAL = {Research in Number Theory},
    VOLUME = {12},
      YEAR = {2026},
    NUMBER = {1},
     PAGES = {Paper No. 11, 24},
      ISSN = {2522-0160,2363-9555},
   MRCLASS = {11G40 (11G10 11G20 11M50 14G10)},
  MRNUMBER = {5019078},
       DOI = {10.1007/s40993-026-00701-w},
       URL = {https://doi.org/10.1007/s40993-026-00701-w},
}

@incollection {KedlayaSutherland2009,
    AUTHOR = {Kedlaya, Kiran S. and Sutherland, Andrew V.},
     TITLE = {Hyperelliptic curves, {$L$}-polynomials, and random matrices},
 BOOKTITLE = {Arithmetic, geometry, cryptography and coding theory},
    SERIES = {Contemp. Math.},
    VOLUME = {487},
     PAGES = {119--162},
 PUBLISHER = {Amer. Math. Soc., Providence, RI},
      YEAR = {2009},
   MRCLASS = {11G40 (11G30 11M50)},
  MRNUMBER = {2555991},
       DOI = {10.1090/conm/487/09529},
       URL = {https://doi.org/10.1090/conm/487/09529},
}

@incollection {LarioSomoza2018,
    AUTHOR = {Lario, Joan-Carles and Somoza, Anna},
     TITLE = {The {S}ato-{T}ate conjecture for a {P}icard curve with complex
              multiplication (with an appendix by {F}rancesc {F}it\'{e})},
 BOOKTITLE = {Number theory related to modular curves---{M}omose memorial
              volume},
    SERIES = {Contemp. Math.},
    VOLUME = {701},
     PAGES = {151--165},
 PUBLISHER = {Amer. Math. Soc., Providence, RI},
      YEAR = {2018},
   MRCLASS = {11G10 (11G15 11G30 11G40)},
  MRNUMBER = {3755912},
MRREVIEWER = {Igor V. Nikolaev},
       DOI = {10.1090/conm/701/14151},
       URL = {https://doi.org/10.1090/conm/701/14151},
}

@article {Hazama2,
    AUTHOR = {Hazama, Fumio},
     TITLE = {Algebraic cycles on certain abelian varieties and powers of
              special surfaces},
   JOURNAL = {J. Fac. Sci. Univ. Tokyo Sect. IA Math.},
  FJOURNAL = {Journal of the Faculty of Science. University of Tokyo.
              Section IA. Mathematics},
    VOLUME = {31},
      YEAR = {1985},
    NUMBER = {3},
     PAGES = {487--520},
      ISSN = {0040-8980},
   MRCLASS = {14C30 (11F41 11G10 14K10)},
  MRNUMBER = {776690},
MRREVIEWER = {Jerome William Hoffman},
}

@article {Chi,
    AUTHOR = {Chi, W\^{e}n Ch\^{e}n},
     TITLE = {{$l$}-adic and {$\lambda$}-adic representations associated to
              abelian varieties defined over number fields},
   JOURNAL = {Amer. J. Math.},
  FJOURNAL = {American Journal of Mathematics},
    VOLUME = {114},
      YEAR = {1992},
    NUMBER = {2},
     PAGES = {315--353},
      ISSN = {0002-9327},
   MRCLASS = {11G10 (11F80 11S20)},
  MRNUMBER = {1156568},
MRREVIEWER = {Kenneth A. Ribet},
       DOI = {10.2307/2374706},
       URL = {https://doi.org/10.2307/2374706},
}

@article {yelton2014images,
    AUTHOR = {Yelton, Jeffrey},
     TITLE = {Images of 2-adic representations associated to hyperelliptic
              {J}acobians},
   JOURNAL = {J. Number Theory},
  FJOURNAL = {Journal of Number Theory},
    VOLUME = {151},
      YEAR = {2015},
     PAGES = {7--17},
      ISSN = {0022-314X,1096-1658},
   MRCLASS = {14H40 (11F50 11G05 11G30 12F10)},
  MRNUMBER = {3314198},
MRREVIEWER = {Rolf\ Berndt},
       DOI = {10.1016/j.jnt.2014.10.020},
       URL = {https://doi.org/10.1016/j.jnt.2014.10.020},
}

@article {MR3218802,
    AUTHOR = {Fit\'{e}, Francesc and Sutherland, Andrew V.},
     TITLE = {Sato-{T}ate distributions of twists of {$y^2=x^5-x$} and
              {$y^2=x^6+1$}},
   JOURNAL = {Algebra Number Theory},
  FJOURNAL = {Algebra \& Number Theory},
    VOLUME = {8},
      YEAR = {2014},
    NUMBER = {3},
     PAGES = {543--585},
      ISSN = {1937-0652,1944-7833},
   MRCLASS = {11G20 (11G10 14K05)},
  MRNUMBER = {3218802},
MRREVIEWER = {Hui\ June\ Zhu},
       DOI = {10.2140/ant.2014.8.543},
       URL = {https://doi.org/10.2140/ant.2014.8.543},
}

@incollection {MR3502940,
    AUTHOR = {Fit\'{e}, Francesc and Sutherland, Andrew V.},
     TITLE = {Sato-{T}ate groups of {$y^2=x^8+c$} and {$y^2=x^7-cx$}},
 BOOKTITLE = {Frobenius distributions: {L}ang-{T}rotter and {S}ato-{T}ate
              conjectures},
    SERIES = {Contemp. Math.},
    VOLUME = {663},
     PAGES = {103--126},
 PUBLISHER = {Amer. Math. Soc., Providence, RI},
      YEAR = {2016},
      ISBN = {978-1-4704-1947-9},
   MRCLASS = {11M50 (11G10 11G20 14G10 14K15)},
  MRNUMBER = {3502940},
MRREVIEWER = {Daniel\ Sadornil},
       DOI = {10.1090/conm/663/13351},
       URL = {https://doi.org/10.1090/conm/663/13351},
}

@book {serre-IV,
    AUTHOR = {Serre, Jean-Pierre},
     TITLE = {Oeuvres/{C}ollected papers. {IV}. 1985--1998},
    SERIES = {Springer Collected Works in Mathematics},
      NOTE = {Reprint of the 2000 edition [MR1730973]},
 PUBLISHER = {Springer, Heidelberg},
      YEAR = {2013},
     PAGES = {viii+694},
      ISBN = {978-3-642-39839-1},
   MRCLASS = {01A75 (01A70)},
  MRNUMBER = {3185222},
}

@article {MR3906177,
    AUTHOR = {Lombardo, Davide and Lorenzo Garc\'{\i}a, Elisa},
     TITLE = {Computing twists of hyperelliptic curves},
   JOURNAL = {J. Algebra},
  FJOURNAL = {Journal of Algebra},
    VOLUME = {519},
      YEAR = {2019},
     PAGES = {474--490},
      ISSN = {0021-8693,1090-266X},
   MRCLASS = {11R34 (14H10 14H45)},
  MRNUMBER = {3906177},
MRREVIEWER = {Igor\ A.\ Rapinchuk},
       DOI = {10.1016/j.jalgebra.2018.08.035},
       URL = {https://doi.org/10.1016/j.jalgebra.2018.08.035},
}

@book {Washington,
    AUTHOR = {Washington, Lawrence C.},
     TITLE = {{Introduction to cyclotomic fields}},
    SERIES = {Graduate Texts in Mathematics},
    VOLUME = {83},
   EDITION = {Second},
 PUBLISHER = {Springer-Verlag, New York},
      YEAR = {1997},
     PAGES = {xiv+487},
      ISBN = {0-387-94762-0},
   MRCLASS = {11R18 (11-01 11-02 11R23)},
  MRNUMBER = {1421575},
MRREVIEWER = {T. Mets\"{a}nkyl\"{a}},
       DOI = {10.1007/978-1-4612-1934-7},
       URL = {https://doi.org/10.1007/978-1-4612-1934-7},
}

@misc{OurScripts,
AUTHOR = {Gallese, Andrea and Goodson, Heidi and Lombardo, Davide},
TITLE = {{Computations with the monodromy fields of Fermat {J}acobians}},
YEAR = {2024},
NOTE = {Available at \url{https://github.com/G4ll/The-monodromy-field-of-Fermat-Jacobians}}
}

@article{Katz1974,
author = {Katz, Nicholas M. and Messing, William},
journal = {Inventiones mathematicae},
pages = {73-78},
title = {Some Consequences of the Riemann Hypothesis for Varieties over Finite Fields.},
url = {http://eudml.org/doc/142251},
volume = {23},
year = {1974},
}

@article {MR3766118,
    AUTHOR = {Lombardo, Davide},
     TITLE = {Galois representations attached to abelian varieties of {CM}
              type},
   JOURNAL = {Bull. Soc. Math. France},
  FJOURNAL = {Bulletin de la Soci\'{e}t\'{e} Math\'{e}matique de France},
    VOLUME = {145},
      YEAR = {2017},
    NUMBER = {3},
     PAGES = {469--501},
      ISSN = {0037-9484,2102-622X},
   MRCLASS = {11G10 (11F80 11G05 14K15 14K22)},
  MRNUMBER = {3766118},
MRREVIEWER = {Remke\ Kloosterman},
       DOI = {10.24033/bsmf.2745},
       URL = {https://doi.org/10.24033/bsmf.2745},
}

@article {MR0608640,
    AUTHOR = {Ribet, Kenneth A.},
     TITLE = {Division fields of abelian varieties with complex
              multiplication},
   JOURNAL = {M\'{e}m. Soc. Math. France (N.S.)},
  FJOURNAL = {M\'{e}moires de la Soci\'{e}t\'{e} Math\'{e}matique de France.
              Nouvelle S\'{e}rie},
    NUMBER = {2},
      YEAR = {1980/81},
     PAGES = {75--94},
      ISSN = {0037-9484},
   MRCLASS = {14K15 (14K22)},
  MRNUMBER = {608640},
MRREVIEWER = {M.\ Bashmakov},
}

@misc{auffarth2024jacobian,
      title={{On the Jacobian variety of the Accola-Maclachlan curve of genus four}}, 
      author={Robert Auffarth and Sebastián Reyes-Carocca and Anita M. Rojas},
      year={2024},
      eprint={2401.13277},
      archivePrefix={arXiv},
      primaryClass={math.AG}
}

@article {MR4496693,
    AUTHOR = {Zywina, David},
     TITLE = {Determining monodromy groups of abelian varieties},
   JOURNAL = {Res. Number Theory},
  FJOURNAL = {Research in Number Theory},
    VOLUME = {8},
      YEAR = {2022},
    NUMBER = {4},
     PAGES = {Paper No. 89, 53},
      ISSN = {2522-0160,2363-9555},
   MRCLASS = {14K15 (11F80 14D05)},
  MRNUMBER = {4496693},
MRREVIEWER = {Ilya\ Karzhemanov},
       DOI = {10.1007/s40993-022-00391-0},
       URL = {https://doi.org/10.1007/s40993-022-00391-0},
}

@article {MR0491708,
    AUTHOR = {Gross, Benedict H. and Rohrlich, David E.},
     TITLE = {Some results on the {M}ordell-{W}eil group of the {J}acobian
              of the {F}ermat curve},
   JOURNAL = {Invent. Math.},
  FJOURNAL = {Inventiones Mathematicae},
    VOLUME = {44},
      YEAR = {1978},
    NUMBER = {3},
     PAGES = {201--224},
      ISSN = {0020-9910,1432-1297},
   MRCLASS = {14G25 (10B15)},
  MRNUMBER = {491708},
MRREVIEWER = {M.\ Bashmakov},
       DOI = {10.1007/BF01403161},
       URL = {https://doi.org/10.1007/BF01403161},
}

@article {MR0480542,
    AUTHOR = {Gross, Benedict H.},
     TITLE = {On the periods of abelian integrals and a formula of {C}howla
              and {S}elberg},
      NOTE = {With an appendix by David E. Rohrlich},
   JOURNAL = {Invent. Math.},
  FJOURNAL = {Inventiones Mathematicae},
    VOLUME = {45},
      YEAR = {1978},
    NUMBER = {2},
     PAGES = {193--211},
      ISSN = {0020-9910,1432-1297},
   MRCLASS = {14K22 (14K15 33A25)},
  MRNUMBER = {480542},
MRREVIEWER = {Neal\ Koblitz},
       DOI = {10.1007/BF01390273},
       URL = {https://doi.org/10.1007/BF01390273},
}

@misc{part2,
      title={{M}onodromy groups and exceptional {H}odge classes, {II}: {S}ato-{T}ate groups}, 
      author={Andrea Gallese and Heidi Goodson and Davide Lombardo},
      year={2025},
      eprint={2405.20394},
      archivePrefix={arXiv},
      primaryClass={math.NT},
      url={https://arxiv.org/abs/2405.20394}, 
}

\end{document}